\documentclass[12pt,a4paper]{amsart}
\usepackage{amsmath, amssymb}
\usepackage{enumitem}
\usepackage{cancel}
\usepackage{mathtools}
\usepackage{xcolor}
\usepackage{amsmath,amstext,amssymb,amsthm,amsfonts}
\usepackage[all]{xy} \xyoption{poly}
\usepackage[colorlinks=true, linkcolor=blue, urlcolor=cyan, citecolor=red,bookmarks=true]{hyperref}

\newcommand{\ad}{\operatorname{ad}}

\newcommand{\ima}{\operatorname{im}}

\newcommand{\rank}{\operatorname{rank}}

\newcommand{\Id}{\operatorname{Id}}

\newcommand{\an}{\operatorname{an}}
\newcommand{\pr}{\operatorname{pr}}

\newcommand{\Ker}{\operatorname{Ker}}

\newcommand{\End}{\operatorname{End}}

\newcommand{\Ad}{\operatorname{Ad}}

\newcommand{\Aut}{\operatorname{Aut}}

\newcommand{\Sk}{\operatorname{Sk}}

\newcommand{\fg}{\mathfrak g}

\newcommand{\fc}{\mathfrak c}

\newcommand{\fh}{\mathfrak h}
\newcommand{\ft}{\mathfrak t}

\newcommand{\fsl}{\mathfrak{sl}}

\newcommand{\nen}{\newenvironment}
\newcommand{\ol}{\overline}

\newcommand{\lra}{\longrightarrow}

\newcommand{\iso}{\overset{\sim}{\lra}}

\newcommand{\Thm}[1]{Theorem~\ref{#1}}
\newcommand{\Prop}[1]{Proposition~\ref{#1}}
\newcommand{\Lem}[1]{Lemma~\ref{#1}}
\newcommand{\Cor}[1]{Corollary~\ref{#1}}

\newcommand{\Claim}[1]{Claim~\ref{#1}}

\newcommand{\Rem}[1]{Remark~\ref{#1}}

\nen{thm}[1]{\label{#1}{\bf Theorem.\ } \em}{}
\nen{prop}[1]{\label{#1}{\bf Proposition.\ } \em}{}
\nen{lem}[1]{\label{#1}{\bf Lemma.\ } \em}{}
\nen{cor}[1]{\label{#1}{\bf Corollary.\ } \em}{}
\nen{conj}[1]{\label{#1}{\bf Conjecture.\ } \em}{}

\nen{claim}[1]{\label{#1}{\bf Claim.\ } \em}{}

\nen{defn}[1]{\label{#1}{\bf Definition.\ } }{}
\nen{exa}[1]{\label{#1}{\bf Example.\ } }{}
\nen{rem}[1]{\label{#1}{\em Remark.\ } }{}
\nen{note}[1]{\label{#1}{\em Note.\ } }{}
\nen{exer}[1]{\label{#1}{\em Exercise.\ } }{}
\nen{sket}[1]{\label{#1}{\em Sketch of proof.\ }}{}
\nen{quest}[1]{\label{#1}{\bf Question.\ } \em}{}
\nen{hyp}[1]{\label{#1}{\bf Hypoth\`ese.\ } \em}{}
\title{Conjugation theorem for Affine Kac-Moody superalgebras}
\author{Peleg Bar-Sever}
\begin{document}
\begin{abstract}
   In this thesis, we extend a conjugacy Theorem of Cartan subalgebras, originally established for symmetrizable Kac-Moody algebras, to the broader context of affine Kac-Moody superalgebras. Along the way, we obtain several results that deepen our understanding of the structure of affine Kac-Moody superalgebras. Additionally, we achieve findings regarding the relations within the imaginary root spaces of symmetrizable affine Kac-Moody superalgebras, providing new insights into their algebraic properties.   
\end{abstract}
\maketitle  
\section{Introduction}

In 1983, Peterson and Kac established a fundamental result for symmetrizable Kac-Moody algebras: any two Cartan subalgebras are conjugated by the group of inner automorphisms (See ~\cite{Kac and Peterson}). A Cartan subalgebra, in this context, is defined as a maximal commutative and ad-diagonalizable subalgebra.  In some sense, together with Proposition 5.9 in ~\cite{Kbook}, this completes the work of classification of  Kac-Moody algebras: it demonstrates that isomorphic Kac-Moody algebras necessarily have the same Cartan matrix, up to a permutation of indices and left multiplication by an invertible diagonal matrix. 

In this thesis, we aim to extend this conjugation theorem to the setting of affine Kac-Moody superalgebras, which naturally brings new challenges. Our main result is as follows: 

\begin{thm}{Conjugacy}
    Let $\mathfrak{g}$ be an indecomposable affine Kac-Moody superalgebra, $\mathfrak{h}$  its Cartan subalgebra, and let $\mathfrak{a}$  be any Kac-Moody superalgebra with the Cartan subalgebra $\mathfrak{h}_{\mathfrak{a}}.$ If $\iota:\mathfrak{g}\stackrel{\sim}{\rightarrow}\mathfrak{a}$  is an isomorphism and $\mathfrak{h}'=\iota^{-1}\left(\mathfrak{h}_{\mathfrak{a}}\right)\subset\mathfrak{g},$  then $\mathfrak{h}'=\phi(\mathfrak{h})$ for some inner automorphism of $\mathfrak{g}.$ 
\end{thm}

This result is particularly significant because it applies to all affine Kac-Moody superalgebras, including the two families of non-symmetrizable algebras: $S(2|1;b)$  and $\mathfrak{q}^{(2)}_n.$ 

Extending the conjugation theorem to the affine superalgebra case, introduces additional subtleties. Unlike the classical case, where conjugacy implies that the Cartan matrices of isomorphic Kac-Moody algebras differ only by a permutation of indices and left multiplication by an invertible diagonal matrix, the superalgebra case requires a broader perspective. Specifically, one must also consider isotropic reflexions. A detailed formulation of this refinement will be provided in Chapter 5. 

This thesis is organized as follows: 

\textbf{Chapter 2} lays the groundwork by providing an overview of Kac-Moody Lie superalgebras, including their definitions and fundamental properties. We introduce a slightly modified approach to their construction, which is shorter and designed to offer greater intuitive clarity. In addition, we introduce the notion of the principal subalgebra, which plays a central role in the proof of the conjugacy theorem. Lastly, we describe the construction of the minimal Kac-Moody group, originally developed by Kac and Peterson (\cite{Kac and Peterson}) and later explored extensively in various textbooks and papers (e.g., ~\cite{Kumar}). 

\textbf{Chapter 3} focuses on the structural properties of affine Kac-Moody superalgebras, covering key concepts such as their triangular decomposition, root systems, and the process of affinization.  Alongside other related results, we provide a uniform proof, avoiding case-by-case analysis, of a theorem originally established by Serganova (see ~\cite{Sint}), regarding the structure of the even part of symmetrizable affine Kac-Moody superalgebras, expressed in terms of principal algebra modules. 

In \textbf{Chapter 4}  we explore the relations within the imaginary root spaces of affine symmetrizable Kac-Moody superalgebras, and compute their dimensions. In the specific case of  $\mathfrak{g}=A\left(2k,2l\right)^{\left(4\right)},$ we reveal a strikingly distinct structure for $\mathcal{K},$  where $\mathcal{K}$  denotes the subalgebra spanned by the imaginary root spaces and the Cartan subalgebra.  We establish the following theorem: 

\begin{thm}{sqdelta}
    Let $\fg$ be an indecomposable affine symmetrizable Kac-Moody superalgebra. 

    \begin{enumerate}
        \item  The subalgebra $\mathcal{K}_{\overline{0}}=\mathcal{K}\cap\mathfrak{g}_{\overline{0}},$ is an infinite-dimensional Heisenberg Lie algebra.
        \item If $\mathfrak{g}\neq A\left(2k|2l\right)^{\left(4\right)},$  then $\mathcal{K}=\mathcal{K}_{\overline{0}}$  and thus by $(i)$:  
    $$\left[\mathfrak{g}_{s\delta},\mathfrak{g}_{q\delta}\right]=\begin{cases}
    0 & s+q\neq0\\
    \mathbb{C}K & s+q=0
    \end{cases}$$
        \item If $A\left(2k|2l\right)^{\left(4\right)},$  then $\mathfrak{g}_{s\delta}\subset\mathcal{K}\cap\mathfrak{g}_{\overline{1}}$  if and only if $\overline{s}=1$  or $\overline{s}=3,$  where $$\overline{s}=s\mod 4.$$  

    If $i+j=0,$  then $\left[\mathfrak{g}_{i\delta},\mathfrak{g}_{j\delta}\right]=\mathbb{C}K.$  Otherwise, we have: 
     $$\left[\mathfrak{g}_{i\delta},\mathfrak{g}_{j\delta}\right] =
        \begin{cases}
        0 & i \neq 0, \, \overline{i} = 0 \\
        0 & \overline{i+j} = 0 \\
        \mathfrak{g}_{(i+j)\delta} & \overline{i} = 2, \, \overline{j} = 1,3 \\
        \mathbb{C} x, \quad x \text{ is a nonzero element in } \mathfrak{g}_{(i+j)\delta} & \overline{i} = \overline{j} = 1,3
    \end{cases}$$
    \end{enumerate}  
\end{thm}

In this chapter, we rely on several results we establish, concerning the representations of $B(k|l)$  and $D(k|l).$  These results are explicitly formulated and proved in the appendix.

Finally, in \textbf{Chapter 5}, we prove the main theorem of this thesis, addressing both the symmetrizable and non-symmetrizable cases. We provide a detailed discussion of the motivation behind this conjugacy result and its corollaries. 

\subsection{Acknowledgement.}
First and foremost, I extend my deepest gratitude to my advisor, professor Maria Gorelik. Through her boundless patience, kindness, and remarkable guidance, she dedicated an incredible amount of time to help me navigate this vast and profound subject. Her mentorship has not only made it possible for me to understand and engage deeply with this challenging topic, but also allowed me to contribute to the ongoing efforts in this field.
 
I would like also to thank my life partner, family, and friends for their constant love and support. Knowing I can always rely on them has been a source of great strength and encouragement throughout this journey.

Finally, I am grateful to the Weizmann Institute and my peers, for making an exceptional and pleasant environment for study and research. The collaborative and inspiring atmosphere they have created, was more than I could have ever hoped for. 
Supported by the Minerva foundation with funding from the Federal German Ministry for Education and Research.

\section{Kac-Moody Superalgebras}
Broadly speaking, Kac-Moody superalgebras are a generalization of Kac-Moody algebras to the "super" case.
There are several equivalent ways for defining this notion. The approach used in ~\cite{Sint} is to define a Cartan datum consists of a square matrix $A,$ and a set of indices $\tau.$ We impose certain axioms on this Cartan datum and introduce formulas for refelxions, under which we require the axioms to hold for the new matrices resulting from implying the reflexions on the matrix $A$ and the set $\tau$. In this context, we introduce a realization of this Cartan datum, from which we construct the Kac-Moody superalgebra, similarly to the process in the non-super case as described in Chapter 1 of ~\cite{Kbook}, with some adjustments. In particular, we treat the generators $e_i,f_i$ as odd if and only if $i\in\tau.$ However, at first glance, the axioms we impose might seem a bit arbitrary. 

Another approach for defining Kac-Moody superalgebras is presented in ~\cite{GHS}. This approach is similar, but broader in scope. In this chapter, we present a different perspective, which offers a shorter alternative to the definition in ~\cite{GHS}.  Additionally, it may provide a clearer intuition for the axioms we impose on the matrix $A$ in the previous definitions described above. Our approach is similar to one used by H. Yamane for symmetrizable affine Kac-Moody superalgebras (see ~\cite{Yamane}).
\subsection{Set up}
\textit{Cartan datum} is a pair $(A,\tau),$ where $A=(a_{ij})$ is an arbitrary complex $\ell\times \ell$-matrix , and 
 $\tau\subset\{1,\ldots,\ell\}$ is a subset.

 \textit{Realization of a Cartan datum} is a triple $(\fh, \Sigma,\Sigma^{\vee}),$  assigned to the Cartan datum,
where $\fh$ is an even vector space of dimension $2l-\text{rank}A$, $\Sigma=\{\alpha_i\}_{i=1}^{\ell}$ is a linearly independent subset of
$\fh^*$,  $\Sigma^{\vee}=\{\alpha_i^{\vee}\}_{i=1}^{\ell}$ is a linearly independent subset of
$\fh$ and $\langle \alpha^{\vee}_i,\alpha_j\rangle=a_{ij}$.  

 Equipped with realization of a Cartan datum, we construct the so-called \textit{half-baked Lie superalgebra} $\widetilde{\mathfrak{g}}\left(A,\tau\right).$ This is a Lie superalgebra generated by the symbols 
 $\widetilde{e}_{i},\widetilde{f}_{i}, i=1,\ldots,\ell$, which are odd if and only if  $i\in\tau,$ and by $\mathfrak{h},$ subject to the following defining relations:
 \begin{enumerate}
     \item $\mathfrak{h}$  is commutative.
     \item $\left[h,\widetilde{e}_{i}\right]=\alpha_{i}\left(h\right)\widetilde{e}_{i}\,\,\,\,\forall h\in\mathfrak{h}$
     \item $\left[h,\widetilde{f}_{i}\right]=-\alpha_{i}\left(h\right)\widetilde{f}_{i}\,\,\,\,\forall h\in\mathfrak{h}$
     \item $\left[\widetilde{e}_{i},\widetilde{f}_{j}\right]=\delta_{ij}\alpha_{i}^{\vee}$
 \end{enumerate}
 
 We then define the \textit{contragredient 
 Lie superalgebra} $\fg(A,\tau):=\widetilde{\mathfrak{g}}\left(A,\tau\right)/\mathfrak{I},$ 
 where $\mathfrak{I}$  is the maximal ideal in $\widetilde{\mathfrak{g}}\left(A,\tau\right)$ that has a zero intersection with $\mathfrak{h}.$  Frequently, we omit $\tau$ in the notation and write $\mathfrak{g}(A)$ instead of $\fg(A,\tau).$ We denote by $e_{i} \text{ and }\,f_{i}$ the images of $\widetilde{e}_{i},\widetilde{f}_{i}$ in $\mathfrak{g}(A,\tau)$ and identify $\mathfrak{h}$ with its image in $\mathfrak{g}(A,\tau).$ The elements $e_i, f_i$  are called \textit{Chevalley
generators} and the subalgebra $\mathfrak{h}$  of $\mathfrak{g}(A,\tau)$ is called the \textit{Cartan subalgebra}. 

Every definition and lemma in  ~\cite{Kbook}, Chapter I,  remains valid for $\mathfrak{g}(A,\tau)$ (except Proposition 1.7 (b) which is valid for $A\not=(0)$), including Theorem 1.2, which leads to the \textit{triangular decomposition} $\mathfrak{g}\left(A,\tau\right)=\mathfrak{n}^{-}\oplus\mathfrak{h}\oplus\mathfrak{n},$  where $\mathfrak{n}$ (resp. $\mathfrak{n}^{-}$)  is generated by $\left\{ e_{i}\right\} _{i=1}^{l}$  (resp. $\left\{ f_{i}\right\} _{i=1}^{l}$), and gives the \textit{root space decomposition} with respect to $\mathfrak{h}:$ $$\mathfrak{g}\left(A,\tau\right)=\bigoplus_{\alpha\in\Delta}\mathfrak{g}_{\alpha},$$ where $\mathfrak{g}_{\alpha}=\left\{ x\in\mathfrak{g}\left(A\right)\,|\,\left[h,x\right]=\alpha\left(h\right)x\,\,\forall h\in\mathfrak{h}\right\},$  and $\Delta=\left\{ \alpha\in\mathfrak{h^{*}\,}|\,\mathfrak{g}_{\alpha}\neq0, \alpha\neq0 \right\}.$   
The set $\Delta$  is called \textit{root system}, elements of $\Delta$  are called \textit{roots} and $\mathfrak{g}_\alpha$  is called the \textit{root space of  $\alpha.$} One can see that for every root $\alpha,$ $\mathfrak{g}_\alpha\subset\mathfrak{g}(A,\tau)_{\overline{0}}$  or $\mathfrak{g}_\alpha\subset\mathfrak{g}(A,\tau)_{\overline{1}}.$ A root $\alpha$ is called \textit{even} in the former case, and \textit{odd} in the latter. 

The existence of the triangular decomposition implies: 

\begin{equation}\label{eq:root_disjoint_union} 
\Delta = \Delta^+(\Sigma) \coprod (-\Delta^+(\Sigma)) \quad \text{where} \quad \Delta^+(\Sigma) = \Delta \cap \mathbb{N}\Sigma.\ \footnotemark
\end{equation}
\footnotetext{\(\mathbb{N}\) stands for the set of strictly positive integers.}

Roots in $\Delta^+(\Sigma)$ (resp. $-(\Delta^+(\Sigma))$ are called \textit{positive} (resp. \textit{negative}).  

We call $(\fg(A,\tau),\fh,\Sigma,\Sigma^{\vee})$  a \textit{quadruple
associated to the Cartan datum}. Two quadruples $(\fg(A,\tau),\fh,\Sigma,\Sigma^{\vee})$  and $(\fg(A_1,\tau),\fh_1,\Sigma_1,\Sigma_1^{\vee})$ are called \textit{isomorphic}, if there exists a Lie superalgebra isomorphism  $\phi: \fg(A,\tau)\iso\fg(A_1,\tau)$ such that
$\phi(\fh)=\fh_1,$  $\phi\left(\Sigma^{\vee}\right)=\Sigma_{1}^{\vee},$ and $\phi^*(\Sigma_1)=\Sigma$.

One can see that $\fg(A)=\fg(A_1)\times \fg(A_2)$ if 
 $A=\left(\begin{array}{cc}
A_{1} & 0\\
0 & A_{2}
\end{array}\right).$ We call $\mathfrak{g}(A)$ \textit{indecomposable}, if after any index permutation (i.e. a permutation of rows of the matrix and the same permutation of the columns), $A$ can not be written in the above form. 

$\mathfrak{g}(A,\tau)$  is called\textit{ quasisimple} if for any ideal $i\subset\mathfrak{g}(A,\tau)$ either $i\subset\mathfrak{h}$  or $i+\mathfrak{h}=\mathfrak{g}(A,\tau).$ Equivalently, it is quasisimple if every ideal of $\mathfrak{g}(A,\tau)$ is either in the center of $\mathfrak{g}(A,\tau),$ or contains $[\mathfrak{g}(A,\tau),\mathfrak{g}(A,\tau)].$ 

We call $A$ \textit{regular}  if for any $i,j\in\tau,$ $a_{ij}=0$  implies $a_{ji}=0.$ By ~\cite{Sint}  Lemma 2.4, If $A$  is regular, then $\mathfrak{g}(A,\tau)$ is quasisimple if and only if $A$  is indecomposable and not zero. 

We say that $\mathfrak{g}(A,\tau)$ is \textit{symmetrizable} if there exists an invertible diagonal matrix $D,$ and a symmetric matrix $B,$  such that $A=DB.$ There is an isomorphism between $\mathfrak{g}(A,\tau)$  and $\mathfrak{g}(DA,\tau)$, defined by taking $\alpha_i^\vee$  to $d_{ii}\alpha_i^\vee,$  where $\alpha_i^{\vee}\in\Sigma^\vee$  and $D=\text{diag}(d_{11},...,d_{ll}).$ 

For every Lie superalgebra $\mathfrak{g},$ we denote by $\mathfrak{c}(\mathfrak{g})$  its centre.  
One has:
\begin{equation} \label{gggcg}  
\mathfrak{g}^{[\mathfrak{g},\mathfrak{g}]}=\mathfrak{c}(\mathfrak{g}),  
\end{equation}
  
where $$\mathfrak{g}^{\left[\mathfrak{g},\mathfrak{g}\right]}=\left\{ x\in\mathfrak{g}\,|\,\left[x,\left[\mathfrak{g},\mathfrak{g}\right]\right]=0\right\} .$$

\subsection{Reflexions and contragredient base}
      The subset $\Sigma_1\subset\Delta$  is called a \textit{root base} if it is linearly independent and $\Delta=\Delta^+(\Sigma_1)\coprod (-\Delta^+(\Sigma_1)).$ A root base $\Sigma_1$ is further termed a \textit{contragredient base} if there exists a quadruple $(\fg(A'),\fh',\Sigma',\Sigma'^{\vee})$ and a Lie superalgebra isomorphism $\phi:\mathfrak{g}\left(A\right)\longrightarrow\mathfrak{g}\left(A'\right),$  such that $\phi\left(\mathfrak{h}\right)=\mathfrak{h'}$  and $\phi^{*}\left(\Sigma'\right)=\Sigma_1.$ 

      We say that a base $\Sigma'\subset\Delta$ is \textit{reflectable} at a root $\alpha\in\Delta^{+}(\Sigma'),$ if
 there exists a root base  $\Sigma''$ with the property:
 $$\Delta^+(\Sigma')\setminus (\mathbb{N}\alpha)=\Delta^+(\Sigma'')\setminus (-\mathbb{N}\alpha).$$
In this case, we write $\Sigma''=r_{\alpha}\Sigma'$  and call the correspondence $r_{\alpha}:\Sigma'\rightarrow\Sigma''$ \textit{reflexion}. 
If $\Sigma''$  exists, is unique. The term reflexion (rather than reflection) is used in ~\cite{GHS}.
\subsubsection{}
\begin{claim}{Properties of reflexions}
Let $\Sigma'$ be a root base, reflectable at a root $\alpha,$ and  $\Sigma''=r_{\alpha}\Sigma'.$  Then:
\begin{enumerate}
    \item $\alpha\in\Sigma',$ $-\alpha\in\Sigma''$  and $r_{-\alpha}\Sigma''=\Sigma'.$  
    \item Let $\alpha\not=\beta\in\Sigma'.$  Then $\beta+s\alpha\in\Sigma'',$  where  $s=\max\left\{ n\,|\,\left(\beta+n\alpha\right)\in\Delta\right\} .$  
 
\end{enumerate} 
\end{claim}
\begin{proof}
    By definition of reflexion, one has: $$\Delta^{+}\left(\Sigma'\right)\backslash\left(\mathbb{N}\alpha\right)=\Delta^{+}\left(\Sigma''\right)\backslash\left(-\mathbb{N}\alpha\right)$$
Clearly, $\alpha\notin\Delta^{+}\left(\Sigma'\right)\backslash\left(\mathbb{N}\alpha\right)$  and therefore $\alpha\notin\Delta^{+}\left(\Sigma''\right)\backslash\left(-\mathbb{N}\alpha\right).$ Since $\alpha\notin-\mathbb{N}\alpha,$  it follows that $\alpha\notin\Delta^{+}\left(\Sigma''\right),$  which implies $-\alpha\in\Delta^{+}\left(\Sigma''\right).$ Therefore $r_{-\alpha}\Sigma''=\Sigma'$. Assume, by contrary, that $\alpha\notin\Sigma'.$  Then, we may write $\alpha=\Sigma_{i=1}^{n}k_{i}\beta_{i},$  where $\beta_{i}\in\Sigma', k_{i}\in\mathbb{Z}_{>0}.$  Since $\alpha\notin\Sigma',$  it follows that $\beta_{i}\not\in\mathbb{Z}\alpha.$   Hence, $\beta_{i}\in\Delta^{+}\left(\Sigma'\right)\backslash\left(\mathbb{N}\alpha\right)=\Delta^{+}\left(\Sigma''\right)\backslash\left(-\mathbb{N}\alpha\right).$   Thus, $\beta_{i}\in\Delta^{+}\left(\Sigma''\right),$ so $\alpha\in \Delta^{+}(\Sigma'')$, a contradiction.  This gives (i).

We start the proof of (ii) from the following observation: if $\beta+j\alpha=\sum_{i=1}^l \nu_i,$  where $\nu_i\in\mathbb{N}\Sigma'\setminus\mathbb{N}\alpha,$ then $l=1.$ 
Note that, since $\beta\neq\alpha,$ we have $\beta+s\alpha\notin\mathbb{\mathbb{Z}\alpha},$ which implies by definition of reflexion that $\beta+s\alpha\in\Delta^{+}\left(\Sigma''\right).$ 
Thus, we may write:
$$\beta+s\alpha=\sum_{i=1}^{l}k_{i}\mu_{i} ,\ k_{i}\in\mathbb{Z}_{>0},\ \mu_{i}\in\Sigma''.$$
If all $\mu_i$ are not proportional to $\alpha$, then $\mu_i\in\Delta^+(\Sigma')$ and the above observation gives
$\beta+s\alpha\in \Sigma''$ as required. Now assume that $\mu_1$ is proportional to $\alpha$. Then, $\mu_i\not\in\mathbb{N}\alpha$ for $i\not=1$ and, 
by (i),
$\mu_1=-\alpha$. Therefore $\beta+(s+k_1)\alpha=\sum_{i=1}^2 k_i\mu_i$ where 
$\mu_i\in\Delta^+(\Sigma')$ and the above observation gives $\beta+(s+k_1)\alpha=\mu_2,$
that is $\beta+(s+k_1)\alpha\in\Delta,$ which contradicts the definition of $s$.
\end{proof}

With the notation introduced in the claim above, we write:  $r_{\alpha}\left(\alpha\right)=-\alpha,$  and $r_{\alpha}\left(\beta\right)=\beta+s\alpha.$ 
We say that a root base $\Sigma'$  is \textit{fully reflectable} if it is reflectable at every $\alpha\in\Sigma'.$

\subsection{The subalgebras $\mathfrak{g}\left\langle \alpha_{i}\right\rangle$} 

Let $\mathfrak{g}(A,\tau)$ be an arbitrary contragredient superalgebra with $\Sigma=\{\alpha_i\}_{i=1}^{\ell}$ and let $\mathfrak{g}\left\langle\alpha_i \right\rangle $ be the subalgebra which is generated by
$\fg_{\pm\alpha_i}$ (these are one-dimensional spaces).  Note that:
\begin{equation}\label{eq:fsi}
\mathfrak{g}\left\langle\alpha_i \right\rangle\cong\left\{\begin{array}{lll}
\fsl_2 & & \text{ for } a_{ii}\not=0,\ i\not\in\tau\\
\mathfrak{osp}(1|2) & & \text{ for } a_{ii}\not=0,\ i\in\tau\\
\text{Heisenberg Lie algebra} & & \text{ for } a_{ii}=0, i\not\in\tau\\
\fsl(1|1) & & \text{ for } a_{ii}=0, i\in\tau\\
\end{array}\right.\end{equation}
If $\left\langle \alpha_{i}^{\vee},\alpha_{i}\right\rangle =a_{ii}\not=0$, we define $s_{\alpha_i}\in GL(\fh), GL(\fh^*)$  by  the usual formulas:
\begin{equation} \label{eq:fsalpha}
s_{\alpha_i}\mu = \mu - \frac{2 \langle \alpha_{i}^{\vee},\mu\rangle }{\left\langle \alpha_{i}^{\vee},\alpha_{i}\right\rangle } \alpha_i, \quad 
s_{\alpha_i}h = h - \frac{2 \langle h, {\alpha_i}\rangle }{\left\langle \alpha_{i}^{\vee},\alpha_{i}\right\rangle } \alpha_{i}^{\vee}, \quad \text{for } \mu\in\fh^*,\ h\in\fh.
\end{equation}

\subsection{Kac-Moody superalgebra definition}

Let $\Sk$ (``skeleton'') be the set of all root bases obtained 
from $\Sigma$ by all possible chains of reflexions. 
We say that a contragredient superalgebra $\fg(A,\tau)$ 
is a \textit{Kac-Moody superalgebra} if all root bases in $\Sk$ are 
contragredient and fully reflectable, and in addition $a_{ii}\neq0$  for every $i\notin\tau.$

\subsubsection{Properties of Cartan datum}
Consider the following conditions:
\begin{itemize}
\item[(a)] $\Sigma$ is reflectable at $\alpha_i$ and $r_{\alpha_i}\Sigma$ is  a contragredient base;
\item[(b)]  $\Sigma$ is reflectable at $\alpha_i$;
\item[(c)] $e_i, f_i$ are $\ad$-locally nilpotent;
\item[(d)] \textbf{Matrix conditions.} 
\begin{itemize}
    \item (A1) If $a_{ii}=0,\ i\not\in\tau,$  then  $a_{ij}=0$ for all $j.$   
    \item (A2) $a_{ij}=0 \Longrightarrow a_{ji}=0$
    \item (A3) If $a_{ii}\not=0,$ then:$$\begin{cases}
\frac{2a_{ij}}{a_{ii}}\in\mathbb{Z}_{\leq0} & i\notin\tau\\
\frac{a_{ij}}{a_{ii}}\in\mathbb{Z}_{\leq0} & i\in\tau
\end{cases}$$
\end{itemize}

\end{itemize}

\begin{claim}{}
\begin{enumerate}
\item
One has $ (a)\ \Longrightarrow\ (b)\ \Longrightarrow (c)$.
\item If $a_{ii}\not=0$, then
$(a)\ \Longleftrightarrow\ (b)\Longleftrightarrow\ (c)\ \Longleftrightarrow (d)$.
\item  If $a_{ii}=0$, then: \\
  $(P1)$ If $i\not\in \tau$, then $(c)
 \Longleftrightarrow a_{ij}=0$ for all $j$.\\  
 $(P2)$ If  $i\in \tau$, then (c) automatically holds. \\
 $(P3)$ $(a)\ \Longleftrightarrow (d)$.
\end{enumerate}
\end{claim}
\begin{proof}
By \Claim{Properties of reflexions}, if $\Sigma$ is reflectable at $\alpha_i$, then 
 $r_{\alpha_i}\Sigma=\{\alpha'_j\}$, where $\alpha'_i=-\alpha_i$ and for $j\not=i$ one has 
$\alpha'_j=\alpha_j+s_j\alpha_i$ where $s_j=\max \{n|\ (\alpha_j+n\alpha_i)\in\Delta\}$.
This gives (i).

Assume $a_{ii}\not=0.$ The equivalence  $(c)\Longleftrightarrow  (d),$ follows from ~(\ref{eq:fsi}), and the representation theory of $\mathfrak{sl}_2$  and of $\mathfrak{osp}(1|2).$

Thus, for $(ii),$ it suffices to show that $\left(c\right)\Longrightarrow\left(a\right).$ 
If $e_i, f_i$ are $\ad$-locally nilpotent, then 
$$\phi:=\exp(\ad f_i)\exp(-\ad e_i)\exp(\ad f_i)$$
is an automorphism of $\fg(A)$ with the property $\phi|_{\fh}=s_{\alpha_i}$ (see~\cite{Kbook}, Lemma 3.8),
which gives a quadruple $(\fg(A), \fh,s_{\alpha_i}\Sigma, s_{\alpha_i}\Sigma^{\vee}).$
Thus, $\Sigma$ is reflectable at $\alpha_i$ and $r_{\alpha_i}\Sigma=s_{\alpha_i}\Sigma$
is a contragredient base. This establishes (ii).

Next, assume $a_{ii}=0.$ 

For proving $(P1),$ suppose $i\not\in\tau,$  condition $(c)$ holds and let $1\leq j\leq l.$ By assumption, there exists a minimal $s\in\mathbb{N}$ such that $\text{ad}^{s}f_{i}\left(f_{j}\right)=0.$ Using induction, we can establish the following relation:

\begin{equation}\label{eq:a{ij}=0_formula}
\text{ad}e_i\left(\text{ad}^{s}f_{i}\left(f_{j}\right)\right)=-s\cdot a_{ij}\cdot\text{ad}^{s-1}f_{i}\left(f_{j}\right)    
\end{equation}

Hence, $a_{ij}=0.$ Conversely, suppose $a_{ij}=0$  for all $j.$ A direct computation shows that, for any $j,$ 

$$\text{ad}e_{j}\left(\text{ad}^{2}f_{i}\left(f_{j}\right)\right)=\text{ad}e_{i}\left(\text{ad}^{2}f_{i}\left(f_{j}\right)\right)=0.$$

Consequently, $\text{ad}e_{\alpha}\left(\text{ad}^{2}f_{i}\left(f_{j}\right)\right)=0$ for all $\alpha\in\Sigma.$ 
Thus, $\text{ad}^{2}f_{i}\left(f_{j}\right)$ generates an ideal that has a zero intersection with the Cartan subalgebra, and so $\text{ad}^{2}f_{i}\left(f_{j}\right)=0.$ Therefore, $f_i$  is ad-locally nilpotent. Applying the same reasoning, we deduce that $e_i$  is also ad-locally nilpotent. It gives $(P1).$

$(P2)$ is a direct consequence of (\ref{eq:fsi}).

Finally, we aim to prove $(P3)$. We show $(a)$ implies $a_{ij}=0\Longrightarrow a_{ji}=0$  for all $j.$ Assume, for sake of contradiction that there exists $j$ such that $a_{ij}=0$, $a_{ji}\not=0$
and (a) holds.
The conditions $a_{ij}=0$, $a_{ji}\not=0$ imply
$[e_i,e_j]\not=0$ (since $\left[f_{j},\left[e_{j},e_{i}\right]\right]=a_{ji}e_{i}\neq0),$ and $[e_i,[e_i, e_j]]=0$ (this we saw for the case $i\not\in\tau,$ and it clearly holds for the case $i\in\tau$  as well). Thus 
 $\alpha_j+n\alpha_i\in\Delta$ if and only if $n\in\{0,1\}$.
By \Claim{Properties of reflexions}, if $r_{\alpha_i}\Sigma$ exists, then $\alpha'_i=-\alpha_i$ and
$\alpha'_j=\alpha_i+\alpha_j$,
so $e'_i$, $e'_j$ are proportional to $f_i$ and $[e_i,e_j]$ respectively.
Since $a_{ij}=0,$  by (\ref{eq:a{ij}=0_formula}) we have $[e'_i,e'_j]=0.$ But  $\alpha'_i+\alpha'_j\in\Delta$, 
which contradicts to (a). Hence (a) implies 
$a_{ij}=0\ \Longrightarrow\ a_{ji}=0$ for all $j$,  thereby establishing $(A2)$. Additionally, in the case where $i\not\in\tau,$  since $\left(a\right)\Longrightarrow\left(c\right),$  $(P1)$ implies $(A1).$  Thus, $\left(a\right)\Longrightarrow\left(d\right).$ 
Finally, assume that (d) holds (and $a_{ii}=0$). We proved already that for any $j,$ this gives $\alpha_j+2\alpha_i\not\in\Delta.$  In addition, we know that
$\alpha_i+\alpha_j\in\Delta$ if and only if $a_{ij}\neq0$. We set 
\begin{equation}\label{eq:isoralphaSigma}
\alpha'_j=\left\{\begin{array}{ll}
-\alpha_i\ & \text{ if } i=j,\\
\alpha_j\ & \text{ if } i\not=j,\ a_{ij}=0\\
\alpha_j+\alpha_i \ & \text{ if } a_{ij}\not=0\end{array}\right.
\end{equation}
Note that $\{\alpha'_j\}$ is a linearly independent subset $\Delta^+$.

We introduce $e'_j,f'_j$ by the formulas
$$e'_j:=\left\{\begin{array}{ll}
f_i\ & \text{ if } i=j,\\
e_j\ & \text{ if } i\not=j,\ a_{ij}=0\\
\ [e_i,e_j] \ & \text{ if } a_{ij}\not=0\end{array}\right.,
\ \ \  \ \ \ 
f'_j:=\left\{\begin{array}{ll}
-e_i\ & \text{ if } i=j,\\
f_j\ & \text{ if } i\not=j,\ a_{ij}=0\\
\frac{-[f_i,f_j]}{a_{ij}} \ & \text{ if } a_{ij}\not=0\end{array}\right.
$$
and take   $h'_j:=[e'_j,f'_j]$. Then
\begin{equation}\label{eq:isoralphah}
h'_i=\begin{cases}
-\alpha_{i}^{\vee} & i\in\tau\\
\alpha_{i}^{\vee} & i\notin\tau
\end{cases} ,\ \ h'_j=\alpha_{j}^{\vee}\ \ \text{ if } a_{ij}=0,\ \ \ 
h_{j}'=\begin{cases}
\alpha_{j}^{\vee}+\frac{a_{ji}}{a_{ij}}\alpha_{i}^{\vee} & i\in\tau\\
-\alpha_{j}^{\vee}+\frac{a_{ji}}{a_{ij}}\alpha_{i}^{\vee} & \text{else}
\end{cases} \text{if } a_{ij}\not=0,\end{equation}
so $\{h'_j\}_{i=1}^{\ell}$ are linearly independent.
The relations $[e'_s,f'_j]=\delta_{js}h'_s $
follow from the fact that $(\alpha'_s-\alpha'_j)\not\in\Delta$ for $s\not=j$.
If $a_{ij}\not=0$, then $[e'_i,e'_j]=[f_i,[e_i,e_j]]=a_{ij} e_j$,
so  $\fh$ and
$e'_1,\ldots, e'_{\ell}, f'_1,\ldots, f'_{\ell}$ generate $\fg(A,\tau)$. We set $A':=(a'_{kj})$ with $a'_{kj}:=\langle h'_k,\alpha'_j\rangle$ and $$\tau':=\{j: \alpha'_j\in\Delta_{\ol{1}}\}.$$
It is easy to see that $\rank A=\rank A'$, so $(\fh,\{\alpha'_j\}_{j=1}^n, \{h_j\}_{j=1}^n)$ is a realization of
the Cartan datum $(A',\tau')$.
Consider a half-baked Lie superalgebra $\tilde{\fg}(A',\tau')$. By above, $\fg$ is a homomorphic image
of  $\tilde{\fg}(A',\tau')$ (the homomorphism is equal to the identity on $\fh$ and sends $\tilde{e}_j$, $\tilde{f}_j$ to $e'_j$, $f'_j$ respectively). The kernel of this homomorphism has zero intersection with $\fh$, so
this induces a surjective homomorphism $\iota:\fg\to \fg(A',\tau')$ which is equal to the identity on $\fh$.
Therefore $\Ker\iota$ is an ideal of $\fg$ which has zero intersection with $\fh$. Hence $\iota$ is an isomorphism.
Hence $\Sigma':=\{\alpha_j'\}_{j=1}^n$ is a contragredient base of $\Delta$.  By above, $-\alpha_i\in\Sigma'$ and
for $\alpha_j\not=\alpha_i$ we have
$\alpha'_j=\alpha_j+s_j\alpha_i$ where $s_j=\max\{n|\ (\alpha_j+n\alpha_i)\in\Delta\}$. Hence
$\Sigma'=r_{\alpha_i}\Sigma$. This establishes the implication (d) $\Longrightarrow$ (a) and completes the proof.
\end{proof}
 
We have derived an explicit formula for the reflexion: 
\begin{itemize}
    \item If $a_{ii} \neq 0$, the formula is given by $r_{\alpha_i}(\alpha_j) = s_{\alpha_i}(\alpha_j)$. Such a reflexion is also called reflection. 
    \item If $a_{ii} = 0$, the formula is given by $r_{\alpha_i}(\alpha_j) = \alpha_j'$, where $\alpha_j'$ is defined by the expression in (\ref{eq:isoralphaSigma}). We call such a reflexion \textit{isotropic.} 
\end{itemize}

In addition, from the claim above, the following can be deduced:
\begin{enumerate}
    \item If $\mathfrak{g}\left(A,\tau\right)$ is an indecomposable contragredient superalgebra such that all root bases in its skeleton are fully reflectable and contragredient, and there exists $i\notin\tau$  with $a_{ii}=0,$ then $A=(0)$ and $\tau=\emptyset.$ In this case, $[\fg(A,\tau),\fg(A,\tau)]$  is the three dimensional Heisenberg Lie algebra.
 \item A contragredient superalgebra $\fg(A,\tau)$ is Kac-Moody superalgebra if and only if $\fg(A,\tau)$ is a direct product of indecomposable Kac-Moody superalgebras.
\item Purely even Kac-Moody superalgebras are Kac-Moody algebras.
\item  $\mathfrak{g}\left(A,\tau\right)$ is a Kac-Moody Lie superalgebra if and only if the conditions
\begin{equation}\label{eq:Aadmissible}
\begin{array}{l}
a'_{ii}\neq0 \ \text{ if } \ i\not\in\tau' ;\\
a'_{ij}=0\ \ \Longrightarrow a'_{ji}=0;\\
\frac{2a'_{ij}}{a'_{ii}}\in\mathbb{Z}_{\leq 0}\ \text{ if } a'_{ii}\not=0, \ i\not\in\tau';\\
\frac{a'_{ij}}{a'_{ii}}\in\mathbb{Z}_{\leq 0}\ \text{ if }   a'_{ii}\not=0, \ i\in\tau';\end{array}\end{equation} hold for all Cartan datum $\left(A',\tau'\right)$  that obtained from original $\left(A,\tau\right)$ by a chain of isotropic reflexions. 
Note that non-isotropic reflexions do not change the Cartan matrix (therefore, in the non-super case, $\mathfrak {g}({A})$  is a Kac-Moody algebra if and only if the above axioms hold for the Cartan matrix $A).$
\end{enumerate}
        
\subsection{Real roots} \label{realroots}
We say that $\alpha\in\Delta$  is a \textit{real root}, if $\alpha\in\Sigma'$  for some $\Sigma'\in\text{Sk}.$  The set of real roots is denoted by $\Delta^{re}.$ If $\alpha$ is real, then $\fg_{\pm\alpha}$ and $[\fg_{\alpha},\fg_{-\alpha}]$
are one-dimensional. We call a real root $\alpha$ {\em isotropic} 
if $\alpha(h)=0$ for all $h\in [\fg_{\alpha},\fg_{-\alpha}]$, and
{\em anisotropic} otherwise. 
 If $\Sigma'$
is a contragredient base, then $\alpha'_i\in\Sigma'$ is isotropic if
and only if $a'_{ii}=0$. We denote by $\Delta^{\an}$ and $\Delta^\text{iso}$
the sets of anisotropic and isotropic roots respectively. We call a Kac-Moody superalgebra $\fg(A,\tau)$ {\em anisotropic} if 
$a_{ii}\not=0$ for all $i$ (in this case all real roots are anisotropic). 

We call a root $\alpha$
{\em imaginary} if $\alpha$ and $\frac{\alpha}{2}$ are not real roots. Denote by $\Delta^{\text{im}}$  the set of imaginary roots. 

We define the \textit{Weyl group} $W$ as a subgroup of $GL(\fh^*)$ generated by
the reflection $s_{\alpha}$  where $\alpha\in\Delta^{\an}$. The formula
 (\ref{eq:Aadmissible}) ensures that $W\Sigma\subset \Delta$. 
By above, if $a_{ii}\not=0$ for all $i=1,\ldots,\ell$, then $\Sk=\left\{ w\Sigma\right\} _{w\in W}$.

If $\alpha\in\Delta^{\an}$ is odd, then $2\alpha\in\Delta;$  we
set $(2\alpha)^{\vee}=\alpha^{\vee}/2$ and 
introduce $s_{2\alpha}$ by the usual formula (see~(\ref{eq:fsalpha})). Then 
 $s_{2\alpha}=s_\alpha$ and $W$ is generated by all reflections $s_\alpha$,  where $\alpha$  is even and $\alpha\not\in\Delta^{\ima}.$

\subsection{Classification}
For $u\in\mathbb{R}^{n}$  we say that $u>0$ (resp. $u\geq0$)  if all its coordinates of $u$ are strictly positive (resp. non-negative). For $\alpha,\beta\in\mathbb{R}_{\geq0}\Sigma$  we say that $\alpha\geq\beta$  if and only if $\alpha-\beta\in\mathbb{R}_{\geq0}\Sigma.$   

The indecomposable matrices satisfying ~(\ref{eq:Aadmissible}) and $a_{ii}\not=0$ for all $i,$ were classified by Vinberg in 1971, in the following way (see also ~\cite{{Kbook}}, Chapter IV, or \cite{K78}):

(Fin) $\det A\neq0$; there exists $u>0$ s.t. $Au>0$; $Av\geq0$ implies $v>0$ or $v=0.$ 

(Aff) corank$A=1$; there exists $u>0$ s.t. $Au=0$; $Av\geq0$ implies $Av=0$.

(Ind) there exists $u>0$ s.t. $Au<0$; $Av\geq0, v\geq0$ implies $v=0.$ 

Let $\fg(A,\tau)$ be an anisotropic Kac-Moody superalgebra with an indecomposable matrix $A$.
 Since $\fg(A,\tau)$ is  isomorphic to $\fg(DA,\tau)$
for any diagonalizable invertible matrix $D$, without loss of generality we assume that $a_{ii}=2$ for all
$i=1,\ldots,l$. The above classification of matrices,
divides the anisotropic Kac-Moody superalgebras into three distinct types: finite (Fin), affine (Aff), and indefinite (Ind), see~\cite{K78}, according to the type of $A$. 

Given a Kac-Moody superalgebra $\mathfrak{g}(A,\tau),$ we define the \textit{positive cone} by: $$Q_{\mathbb{R}}^{++}=\bigcap_{\Sigma'\in Sk}\mathbb{R}_{\geq0}\left(\Sigma'\right).$$
In the anisotropic case, since $Sk=\left\{ w\Sigma\right\} _{w\in W},$ one has $Q_{\mathbb{R}}^{++}=\bigcap_{w\in W}\mathbb{R}_{\geq0}\left(w\Sigma\right).$ 

We shall present a lemma that demonstrates how, in the anisotropic case, this classification can be reformulated in terms of the positive cone. To do so, we first require the following preliminary lemma.
\subsubsection{}
\begin{lem}{lem:w}
Let $\mathfrak{g}(A,\tau)$  be an anisotropic Kac-Moody superalgebra
with  a set of simple roots $\Sigma=\{\alpha_i\}_{i=1}^l$. Assume that 
 $\lambda\in\fh^*$ satisfies $\frac{\langle \lambda,\alpha_{i}^{\vee}\rangle}{a_{ii}} \leq 0$ for every $1\leq i\leq l$. Then
  $(w\lambda- \lambda)\in\mathbb{R}_{\geq 0}\Sigma$ 
    for all $w\in W$. If, in addition, $\lambda\in\mathbb{R}_{\geq0} \Sigma$,    then $W\lambda\subset Q_{\mathbb{R}}^{++}$. 
\end{lem}
\begin{proof}
Recall that $W$ is the Coxeter group generated by $s_{\alpha_i}$ for $i=1,\ldots l$.
    We prove the inequality   $w\lambda\geq \lambda$  by induction on the length of $w$.  
   If the length is zero, we have $w=\Id$ and $w\lambda=\lambda$, so the inequality holds. 
   If the length is positive, then $w=w's_{\alpha_i}$ for some $i\in\{1,\ldots,l\}$ and $w'\in W$ is of the smaller length.
 By Lemma 3.11 in ~\cite{Kbook}, we have $w'\alpha_{i}\in\Delta^+\subset\mathbb{R}_{\geq 0}\Sigma$, so 
  $$w\lambda-w'\lambda=w'(s_{\alpha_{i}}\lambda-\lambda)=-\frac{2\langle \lambda,\alpha_{i}^{\vee}\rangle}{a_{ii}}w'\alpha_{i}
\in \mathbb{R}_{\geq 0}\Sigma.$$
  By the induction hypothesis, $(w'\lambda-\lambda)\in \mathbb{R}_{\geq 0}\Sigma$, so $(w\lambda-\lambda)\in \mathbb{R}_{\geq 0}\Sigma$. This completes the induction. 

Assume, in addition, that  $\lambda\in\mathbb{R}_{\geq0} \Sigma$. Then, by above, $W\lambda\in \mathbb{R}_{\geq0} \Sigma$,
so $\lambda\in  \bigcap_{w\in W}\mathbb{R}_{\geq0}\left(w\Sigma\right)$.
Since $\fg(A,\tau)$ is anisotropic we have 
 $Q_{\mathbb{R}}^{++}=\bigcap_{w\in W}\mathbb{R}_{\geq0}\left(w\Sigma\right)$, so $\lambda\in Q_{\mathbb{R}}^{++}$. 
 Since  $Q_{\mathbb{R}}^{++}$ is $W$-stable we obtain $W\lambda\subset Q_{\mathbb{R}}^{++}$ as required.
\end{proof}

\subsubsection{}
\begin{lem}{Positive_cone}
Let $A$ be an $l\times l$  Cartan matrix with $a_{ii}\not=0$ for all $i.$  Then:
\begin{enumerate}
    \item $A$ is of type (Fin) if and only if $Q_{\mathbb{R}}^{++}=\left\{ 0\right\}.$
    \item $A$ is of type (Aff) if and only if $Q_{\mathbb{R}}^{++}=\mathbb{R}_{\geq0}\delta$ for some $\delta\in\Delta$.
    \item $A$ is of type (Ind) if and only if $\dim Q_{\mathbb{R}}^{++}>1.$  
\end{enumerate}
\end{lem}
\begin{proof}
Let $\mathbb{R}^l\subset\fh^*$ be a vector space with the basis $\Sigma=\{\alpha_i\}_{i=1}^l$.
Then for $u\in\mathbb{R}^l$ we have $u\geq 0$ if and only if $u\in \mathbb{R}_{\geq 0}\Sigma$.
Note that the $i$th coordinate of $A\mu$ is equal to $\langle \mu,\alpha_i^{\vee}\rangle$, so
\begin{equation}\label{eq:Au>0}
\begin{array}{l}
Au=0\ \Longleftrightarrow\ Au=u\\
Au\geq 0\ \Longleftrightarrow\ Au\in u+\mathbb{R}_{\geq 0}\Sigma.\end{array}
\end{equation}
We are starting with the proof of $(iii).$ Assume $A$  is of (Ind) type. Then the set $\left\{ u>0\,|\,Au<0\right\} $  is not empty. Take $\mu$ in this set. Then $\mu\in\mathbb{R}_{\geq 0}\Sigma$ and $\langle \mu,\alpha_{i}^{\vee}\rangle<0$
for all $i=1,\ldots,l$ (since $A\mu<0$). By \Lem{lem:w}, $W\mu\subset\mathbb{Q}_{\mathbb{R}}^{++}$.
If  $\mu$ and $s_{\alpha_i}\mu$ are linearly dependent, then
$s_{\alpha_i}\mu=-\mu$, so $\mu$ or $s_{\alpha_i}\mu$
 do not lie in the cone $\mathbb{Q}_{\mathbb{R}}^{++}$. Therefore $\mu$ and $s_{\alpha_i}\mu$ are linearly independent, 
 so  $W\mu\subset\mathbb{Q}_{\mathbb{R}}^{++}$ implies $\dim Q_{\mathbb{R}}^{++}>1$.  

Next, assume $\dim\mathbb{Q}_{\mathbb{R}}^{++}>1.$ Then, $\mathbb{Q}_{\mathbb{R}}^{++}$ contains a ball. Therefore  $\mathbb{Q}_{\mathbb{R}}^{++}$ contains $\mu=u_1\alpha_1+...+u_l\alpha_l$ with $u_1,\ldots,u_l\in\mathbb{Z}_{>0}$.
(this means that $\mu\in (\mathbb{Q}_{\mathbb{R}}^{++}\cap\mathbb{Z}_{>0}\Sigma)$).
The orbit $W\mu$ lies in $\mathbb{Q}_{\mathbb{R}}^{++}$.  By~\ref{realroots}, $W\Sigma\subset \mathbb{Z}\Sigma$, so
$W\mu\subset (\mathbb{Q}_{\mathbb{R}}^{++}\cap\mathbb{Z}\Sigma).$ In particular,
$W\mu\subset \mathbb{Z}_{\geq 0}\Sigma$. Therefore the orbit $W\mu$ contains
$\nu$ such that  $(w\nu-\nu)\not\in \mathbb{Z}_{\geq 0}\Sigma\setminus\{0\}$.
Then for any $i=1,\ldots,l$ we have 
 $s_{\alpha_{i}}\left(\nu\right)-\nu=\langle \nu,\alpha_{i}^{\vee}\rangle\alpha_{i}\in \mathbb{Z}_{\leq 0}\alpha_i$,
 so $\langle \nu,\alpha_{i}^{\vee}\rangle\leq 0$. Thus, $A\nu\leq 0$. Assume, for sake of contradiction, that $A$  is of type (Aff).
Since $\mathbb{Q}_{\mathbb{R}}^{++}$ contains a ball,
 we can choose $\mu'\in (\mathbb{Q}_{\mathbb{R}}^{++}\cap\mathbb{Z}_{>0}\Sigma)$ such that $\mu'\not\in\mathbb{C}\mu$
 and construct $\nu'\in W\mu'$ such that $A\nu'\leq 0$.
 Since $A$ is of type (Aff), we have $A\nu=A\nu'=0$ which, by~(\ref{eq:Au>0}) implies $W\mu=W\nu=\nu$ and $W\nu'=W\mu'=\nu'$.
 Since 
  $\text{corank}A=1,$  we get that $\mu,\mu'$ are linearly dependent, a contradiction. Hence, $A$  is of type (Ind).
   This gives $(iii)$.

Now take $A$ of type (Aff).  
 If $\mu\geq 0$ is a vector such that $A\mu=0$, then $\mu\in\mathbb{R}_{\geq 0}\Sigma$ and, 
by~(\ref{eq:Au>0}),   $W\mu=\mu$, so 
$\mu\in \bigcap_{w\in W}\mathbb{R}_{\geq 0}\left(w\Sigma\right)=Q_{\mathbb{R}}^{++}$.  In particular,
$Q_{\mathbb{R}}^{++}\not=0$. By (iii) we have $\dim Q_{\mathbb{R}}^{++}=1$. 

Now assume that $\dim Q_{\mathbb{R}}^{++}=1$ that is
$\mathbb{Q}_{\mathbb{R}}^{++}=\mathbb{R}_{\geq0}\delta$ for some $\delta\not=0$.
Since $\mathbb{Q}_{\mathbb{R}}^{++}$ is $W$-invariant, this gives $W\delta\subset \mathbb{R}\delta$ which implies
 $\langle\delta,\alpha_i^{\vee}\rangle=0$ for all $i$ such that $\delta\not\in\mathbb{R}\alpha_i$.
If $\delta\in \mathbb{R}\alpha_i$ for some $i$, then $s_{\alpha_i}\delta=-\delta$ which means that 
the cone $\mathbb{Q}_{\mathbb{R}}^{++}$ contains $\delta$ and $-\delta$ which is impossible.
Therefore $\langle\delta,\alpha_i^{\vee}\rangle=0$ for all $i$ that is $A\delta=0$. 
Since $\delta\in \mathbb{Q}_{\mathbb{R}}^{++}$ we obtain $\delta\geq 0$ and $A\delta=0$, so $A$ is type (Aff).
This completes the proof.
\end{proof}

\subsubsection{Classification in terms of the positive cone } 
The Lemma above naturally leads to the classification of Kac-Moody superalgebras (including those that are not anisotropic) in terms of the positive cone.

We divide indecomposable Kac-Moody superalgebras into three types:

(Fin) $Q_{\mathbb{R}}^{++}=\left\{ 0\right\}.$

(Aff)  $Q_{\mathbb{R}}^{++}=\mathbb{\mathbb{R}}_{\geq0}\delta.$ 

(Ind) $\dim Q_{\mathbb{R}}^{++}>1.$ 

\subsubsection{Classification in terms of Gelfand-Kirillov dimension}
Consider the natural filtration $\left\{ \mathcal{F}^{1}\subset\mathcal{F}^{2}\subset...\right\}$ of $\mathfrak{g}(A,\tau)$ constructed with respect to a chosen set of generators $S.$  The Gelfand-Kirillov dimension (denoted $GK\dim$) of $\mathfrak{g}(A,\tau)$ is defined as: $$\text{GKdim}(\mathfrak{g}(A,\tau)) = \limsup_{i \to \infty} \frac{\log \dim \mathcal{F}^i}{\log i}.$$ 
This definition do not depend on the choice of generators. 

The classification of indecomposable $\fg(A,\tau)$ can be made also by the growth (see ~\cite{vdL}, ~\cite{H} and ~\cite{H-V}): 

(Fin)  $\dim\fg(A,\tau)<\infty.$  

(Aff)  $\fg(A,\tau)$  is infinite-dimensional, and of finite Gelfand-Kirillov-dimension 

(Ind) $\fg(A,\tau)$  is of infinite Gelfand-Kirillov-dimension.

All superalgebras of type (Fin) are symmetrizable and 
all purely anisotropic superalgebras of type (Aff) are symmetrizable. Surprisingly,
there are no symmetrizable superalgebras among indecomposable  non-anisotropic in the (Ind) case.

\subsubsection{Finite type}
A finite-dimensional simple Lie superalgebra is called \textit{classical} if its even part, $\mathfrak{g}\left(A,\tau\right)_{\overline{0}},$ is reductive. It is called \textit{basic} if it is classical and admits a non-degenerate even super-symmetric invariant bilinear form. All indecomposable finite-dimensional Kac-Moody superalgebras are both basic and classical, with one exceptional - $\mathfrak{g}(n|n).$ These algebras fall into the following list: $$\mathfrak{gl}(n|n)\ n\geq1 
 ,\mathfrak{sl}\left(m|n\right) \ m\neq n,\mathfrak{osp}\left(m|2n\right),F\left(4\right),G\left(3\right),D\left(2|1;\alpha\right)\,\left(\alpha\in\mathbb{C}\backslash\left\{ 0,1\right\} \right)$$
Frequently, we used the following type notation (see ~\cite{Ksuper}): 

$A(m|n)=\fsl(m+1|n+1)$ for $m\not=n$,
$B(m|n)=\mathfrak{osp}(2m+1|2n)$, $C(n)=\mathfrak{osp}(2|2n-2)$,
$D(m|n)=\mathfrak{osp}(2m|2n)$ (with $m\geq 1$).
For the exact definition of each algebra, see \cite{Ksuper}. 

Note that $\mathfrak{sl}\left(n+1|n+1\right)$  has a one-dimensional center, denoted by $\mathbb{C}I.$ We define  $\mathfrak{psl}\left(n+1|n+1\right):=\mathfrak{sl}\left(n+1|n+1\right)/\mathbb{C}I$ and denote $A\left(n|n\right)=\mathfrak{psl}\left(n+1|n+1\right).$  This is a finite-dimensional simple and basic Lie superalgebra, but not a Kac-Moody superalgebra.  
\subsubsection{(Aff) and (Ind) type}
The full classification of symmetrizable Kac-Moody superalgebras of type (Aff) was carried out by Johan van de Leur (see ~\cite{vdL} or \cite{vdLbook}). The classification of non-symmetrizable affine algebras was done by Crystal Hoyt (see ~\cite{Sint}, or ~\cite{H}, and ~\cite{H-V}). This topic will be further discussed in details in Chapter 3 of this thesis. 

The full classification of (Ind) type was conducted by Crystal Hoyt, who showed that the only non-anisotropic (Ind) type algebras are $Q^{\pm}(m,n,q)$
(see ~\cite{H}  and ~\cite{Sint}).

\subsection{The principal Lie algebra $\mathfrak{g}_{pr}$} 
Given a Kac-Moody superalgebra $\mathfrak{g},$ we introduce its so-called \textit{principal Lie algebra} $\mathfrak{g}_{pr}\subset{\mathfrak{g}}_{\overline{0}},$ which serves as the best approximation to the "largest almost contragredient Lie superalgebra" within $\mathfrak{g}_{\overline{0}}.$ 

\subsubsection{}
We say that a root $\alpha$ is \textit{principal,}  if there exists a root base $\Sigma',$ obtained from $\Sigma$  by a chain of isotropic reflexions, such that $\alpha\in\Sigma'.$ The set of principal roots is denoted by $\Sigma_{pr}.$ By ~\cite{H}, $\Sigma_{pr}$  is finite.
We set $\Sigma_{pr}=\left\{ \alpha_{i}\right\} _{i\in J}$ and denote by $B_{pr}$  the matrix $\left\langle \alpha_{i}^{\vee},\alpha_{j}\right\rangle .$ We introduce
$$\pi:=\left\{ \alpha\in\Sigma_{pr}\,|\,\alpha\in\Delta_{\overline{0}}\right\} \cup\left\{ 2\alpha\,|\alpha\in\Sigma_{pr}\cap\Delta_{\overline{1}}\right\}.$$ Note that $\pi\subset\Delta_{\overline{0}}.$ Set $\pi=\left\{ \alpha_{i}'\right\} _{i\in I}$  and denote by $B_\pi$  the matrix $\left\langle \left(\alpha_{i}'\right)^{\vee},\alpha'_{j}\right\rangle .$ Notice that $B_\pi=A$ if $a_{ii}\neq0$ for all $i.$

We define $\fg_{pr}$ as the subalgebra generated by $\mathfrak{g}_{\pm\alpha}$, with $\alpha\in\pi$. 

\subsubsection{} \label{surjective_homomorphism}
We have a surjective homomorphism $\psi:\left[\mathfrak{g}\left(B_{\pi}\right),\mathfrak{g}\left(B_{\pi}\right)\right]\rightarrow\mathfrak{g}_{pr},$ which maps the intersection of the Cartan subalgebra of $\mathfrak{g}(B_\pi)$  with $\left[\mathfrak{g}\left(B_{\pi}\right),\mathfrak{g}\left(B_{\pi}\right)\right]$  to $\mathfrak{h}$ and Chevalley generators of the weights $\pm\alpha$ to elements in $\mathfrak{g}_{\pm\alpha}.$ By ~\cite{Sint}, Lemma 3.7 (see ~\cite{Shay} 3.5.1 and 3.5.4 for details), if $B_\pi$  is symmetrizable, then $\fc:=\ker\psi$ lies in the centre of $\mathfrak{g}\left(B_{\pi}\right)$, so $\psi$  induces an isomorphism between  $\left[\mathfrak{g}\left(B_{\pi}\right),\mathfrak{g}\left(B_{\pi}\right)\right]/\mathfrak{c}$
  and $\fg_{pr}$.
\subsubsection{}
\begin{defn}{}
We say that $\mathfrak{g}(A,\tau)$ is\textit{ principally symmetrizable} if $B_{\pi}$  is symmetrizable.
\end{defn}

By~\cite{H},
an indecomposable Kac-Moody superalgebra $\fg(A,\tau)$ is not principally symmetrizable if and only if
$A$ is not a symmetrizable matrix and
$a_{ii}\not=0$ for all $i$.

\subsubsection{}
\begin{prop}{properties of surjective homomorphism}
  Let $\mathfrak{g}(A,\tau)$ be a principally symmetrizable Kac-Moody superalgebra.   
\begin{enumerate}
\item Let $B_1,\ldots, B_k$ be indecomposable blocks of
the matrix $B_{\pi}$ and $I_j:=\psi\bigl([\fg(B_j),\fg(B_j)]\bigr)$ 
for $j=1,\ldots,k$. Then 
$$\mathfrak{g}_{pr}=\sum\limits_{i=1}^k I_i,\ \ (I_i\cap I_j)\subset \mathfrak{c}(\mathfrak{g}_{pr})
\ \ \text{ for all }\ 1\leq i\not=j\leq k$$
and each $I_i$ is  an ideals in $\fg_{pr}+\fh$.
\item  $\psi^{-1}(\mathfrak{c}(\mathfrak{g}_{pr}))\subset \fc(\fg(B_{\pi}))$.
\end{enumerate}
\end{prop}
\begin{proof}
     (i) follows from the fact that $$\ker\psi\subset\mathfrak{c}\left(\mathfrak{g}\left(B_{\pi}\right)\right).$$
     For (ii) recall that
 $\psi^{-1}(\fc(\fg_{pr}))$  lies in the intersection of the Cartan subalgebra of 
 $\fg(B_{\pi})$ with $[\fg(B_{\pi}),\fg(B_{\pi})]$. Moreover,
 $\psi^{-1}(\fc(\fg_{pr}))$ commutes with $[\fg(B_{\pi}),\fg(B_{\pi})]$.
Now the assertion follows from~(\ref{gggcg}).
\end{proof}
\subsubsection{}\label{principal_anisotropic}
Recall that $\Delta^{an}$  is the set of anisotropic real roots. 
By ~\cite{GHS}, 4.3.12, the Weyl group $W$ is the Coxeter group generated by the reflections $s_\alpha$  with $\alpha\in\Sigma_{pr}.$ Moreover, any anisotropic root is $W$  conjugated to a principal root. That is, $W\Sigma_{pr}=\Delta^{an}.$ 
This leads to the following corollary.
\subsubsection{}
\begin{cor}{principal_anisotropic}
    One has $\mathfrak{g}_\alpha\subset\mathfrak{g}_{pr}$  for any even anisotropic root $\alpha$  and $\mathfrak{g}_{2\alpha}\subset\mathfrak{g}_{pr}$  for any odd anisotropic root $\alpha.$ 
\end{cor}
\begin{proof}
    Let $\alpha\in\Delta^{an}.$ By above argument, there exists $w\in W$  and $\beta\in\Sigma_{pr}$ such that $\alpha=w\beta.$  We prove the claim by induction on the length of $w.$ 
    If the length is $0,$  then $\alpha\in\Sigma_{pr},$  and thus result follows by the definition of $\mathfrak{g}_{pr}.$ 
    Assume the length of $w$  is $n.$ Then there exists an even root $\alpha_i\in\Sigma_{pr}$ and $w'\in W$ such that $w=s_{\alpha_i}w',$  where $w'$  has shorter length.  Define $\beta'=w'\beta.$ 
    If $\beta$  is even, then $\alpha$  is even, and by induction hypothesis, we have $$\mathfrak{g}_{\beta'}\subset\mathfrak{g}_{pr}.$$
    Since $\alpha_i$  is an even principal root, we also have $\mathfrak{g_{\pm\alpha_{i}}}\subset\mathfrak{g}_{pr}.$ 
    Since $\alpha=s_{\alpha_{i}}\beta',$ we can use ~\cite{Kbook} Lemma 3.8 to obtain: 
    $$\mathfrak{g}_{\alpha}=\exp f_{i}\left(\exp\left(-e_{i}\right)\right)\exp\left(f_{i}\right)\mathfrak{g}_{\beta'},$$ where $f_i,e_i\in\mathfrak{g_{\pm\alpha_{i}}}$.  Therefore, $\mathfrak{g}_{\alpha}\subseteq\mathfrak{g}_{pr}$  as wished.  
    
    If $\beta$  is odd, then $\alpha$  is odd, and by induction hypothesis, $\mathfrak{g}_{2\beta'}\subset\mathfrak{g}_{pr}.$
    Since $2\alpha=s_{\alpha_i}(2\beta'),$ again by  ~\cite{Kbook} Lemma 3.8, we obtain:
$$\mathfrak{g}_{2\alpha}=\exp f_{i}\left(\exp\left(-e_{i}\right)\right)\exp\left(f_{i}\right)\mathfrak{g}_{2\beta'},$$
where $f_i,e_i\in\mathfrak{g}_{\alpha_i}.$ Therefore we have $\mathfrak{g}_{2\alpha}\subset\mathfrak{g}_{pr},$ as desired.  
\end{proof}
\subsubsection{}
From the corollary above we get that \begin{equation} \label{root_space_in_pr}
    \mathfrak{g}_{\alpha} \subset \mathfrak{g}_{pr} \quad \text{for all } \alpha \in \Delta_{\overline{0}} \backslash \Delta^{im},
\end{equation}
and $\mathfrak{g}_{pr}$  is a subalgebra generated by the root spaces $\mathfrak{g}_\alpha$
 for all $\alpha\in\Delta_0\backslash\Delta^{im}.$ 
\subsubsection{}
\begin{prop}{Cartan_intersect_pr}

\begin{equation} \label{lcl}
    \begin{aligned}
        &\text{(i)} \quad \left(\mathfrak{h} \cap \mathfrak{g}_{pr}\right) = \sum_{\alpha\in\pi} \mathbb{C}\alpha^{\vee} = \sum_{\alpha\in\Delta^{an}} \mathbb{C}\alpha^{\vee} \\[8pt]
        &\text{(ii)} \quad \left(\mathfrak{h} \cap \mathfrak{g}_{pr} \right) + \mathfrak{c} \left(\mathfrak{g}_{\overline{0}}\right) =
        \begin{cases}
            \mathfrak{h} & \text{in type } \text{(Fin)} \\
            \mathfrak{h} \cap \left[\mathfrak{g}, \mathfrak{g} \right] & \text{in types } \text{(Aff), (Ind)}
        \end{cases}
    \end{aligned}
\end{equation}        
\end{prop}
\begin{proof}
    By ~\cite{Sint}, Lemma 3.6, $[\mathfrak{g}_\alpha,\mathfrak{g}_{-\beta}]=\delta_{\alpha\beta}\alpha^{\vee}$  for $\alpha,\beta\in\pi.$ This implies the formula  $\left(\mathfrak{h} \cap \mathfrak{g}_{pr}\right) = \sum_{\alpha\in\pi}\mathbb{C}\alpha^{\vee}.$ By ~(\ref{root_space_in_pr}), $$\sum_{\alpha\in\Delta^{an}} \mathbb{C}\alpha^{\vee}\subset(\mathfrak{h}\cap\mathfrak{g}_{pr}).$$ 
    Clearly $\sum_{\alpha\in\pi}\mathbb{C}\alpha^{\vee}\subset\sum_{\alpha\in\Delta^{an}}\mathbb{C}\alpha^{\vee},$ completing $(i).$ 

    For $(ii)$ note that in the case when $a_{ii}\neq0$ for all $i,$ we have $\Sigma=\Sigma_{pr},$ so $$\mathfrak{h}\cap\mathfrak{g}_{pr}=\mathfrak{h}\cap\left[\mathfrak{g},\mathfrak{g}\right].$$ 
    Consider the case when $A$  is symmetrizable. Let $(-,-)$  be the restriction of a non-degenerate invariant bilinear form on $\mathfrak{h}.$ Then $\left(\mathfrak{h}\cap\mathfrak{g}_{pr}\right)+\mathfrak{c}\left(\mathfrak{g}_{\overline{0}}\right)$ is the orthogonal compliment to $$\mathfrak{h}\cap\mathfrak{g}_{pr}\cap\mathfrak{c}\left(\mathfrak{g}_{\overline{0}}\right)=\mathfrak{c}\left(\mathfrak{g}_{pr}\right),$$
and $\mathfrak{h}\cap[\mathfrak{g},\mathfrak{g}]$ is the orthogonal compliment to $\mathfrak{c}(\mathfrak{g})$  in case  $\mathfrak{g}\neq\mathfrak{psl}^{\left(1\right)}\left(n|n\right)$ (see ~\ref{psl_affine} for definition), and $\mathbb{C}K\subset\mathfrak{c}\left(\mathfrak{g}\right)$  otherwise (see Chapter 3 for definition of $K$).   

In (Fin) case, $\mathfrak{g}(B_\pi)$ is a finite-dimensional Kac-Moody algebra, so $\mathfrak{g}_{pr},$ which is a holomorphic image of $\mathfrak{g}(B_\pi),$ is also a finite dimensional Kac-Moody algebra. Thus, $\mathfrak{c}(\mathfrak{g}_{pr})=0,$  and therefore $\left(\mathfrak{h}\cap\mathfrak{g}_{pr}\right)+\mathfrak{c}\left(\mathfrak{g}_{\overline{0}}\right)=\mathfrak{h}.$ 

Consider the symmetrizable (Aff) case. If $\mathfrak{g}\neq\mathfrak{psl}^{\left(1\right)}\left(n|n\right),$ then $\mathfrak{c}(\mathfrak{g})$ is one dimensional, spanned by the central extension $K.$ Now, $2\delta$  is an even imaginary root in which $\mathfrak{g}_{\pm2\delta}\subseteq\mathfrak{g}_{pr}.$  Thus, $[\mathfrak{g}_{2\delta},\mathfrak{g}_{-2\delta}]=\mathbb{C}K\subseteq\mathfrak{g}_{pr}.$ It follows from \Prop{properties of surjective homomorphism} that $\mathfrak{c}\left(\mathfrak{g}_{pr}\right)\subseteq\mathfrak{c}\left(\mathfrak{g}\right).$    Hence, $\mathfrak{c}\left(\mathfrak{g}_{pr}\right)=\mathfrak{c}\left(\mathfrak{g}\right).$ If $\mathfrak{g}=\mathfrak{psl}^{\left(1\right)}\left(n|n\right),$  then $\dim\mathfrak{c}\left(\mathfrak{g}\right)=2,$  and $\mathfrak{c}\left(\mathfrak{g}_{pr}\right)$ is a proper subspace of $\mathfrak{c}(\mathfrak{g}).$ Hence, $\mathfrak{c}\left(\mathfrak{g}_{pr}\right)=\mathbb{C}K.$
This implies $\left(\mathfrak{h}\cap\mathfrak{g}_{pr}\right)+\mathfrak{c}\left(\mathfrak{g}_{\overline{0}}\right)=\mathfrak{h} \cap \left[\mathfrak{g}, \mathfrak{g} \right]$ as required.

The remaining cases are non-symmetrizable Kac-Moody superalgebras with $a_{ii}=0$ for some $i.$ 
For the non (Aff) case,  $Q(m,n,q)^{\pm},$  note that both $A$  and $B_\pi$  are invertible $3\times3$ matrices. In particular, $\mathfrak{c}(\mathfrak{g}(B_\pi))=0.$ Thus, $\mathfrak{g}_{pr}\simeq\mathfrak{g}(B_{\pi})$ and $\dim\mathfrak{h}=3=\dim(\mathfrak{h}\cap\mathfrak{g}_{pr}).$ Thus, $\mathfrak{h}\cap\mathfrak{g}_{pr}=\mathfrak{h}=\mathfrak{h}\cap\left[\mathfrak{g},\mathfrak{g}\right].$  

Finally,  for (Aff) non-symmetrizable Kac-Moody superalgebras, i.e. $\mathfrak{q}_n^{(2)}$ and $S(2
|1;b),$ the desired formula follows from the description presented in ~\ref{queer_twisted} and ~\ref{S(2|1;b)}.       
\end{proof}

\subsubsection{}
By ~\cite{H}, the matrix $B_\pi$ is symmetrizable for every Kac-Moody superalgebra $\mathfrak{g}=\mathfrak{g}(A,\tau)$  such that $a_{ii}=0$  for some $i.$ Moreover, $B_{pr}$ and $B_\pi$ are generalized Cartan matrices.  Since all Cartan matrices with non-zero diagonal entries of types (Fin) and (Aff) are symmetrizable, the matrix $B_{\pi}$  is symmetrizable for every Kac-Moody superalgebra of type (Fin) or (Aff).  

\subsubsection{}
From case-by-case description of $\mathfrak{g}_{pr}$  given in ~\cite{Sint} Section 9, it follows that the matrix $B_\pi$  is indecomposable if A is of type (Ind). Moreover, for each indecomposable block $B_s$ of the matrix $B_\pi,$  $B_\pi$  is a Cartan matrix of the same type (i.e., (Fin), (Aff) or (Ind)) as $A,$  and $\text{corank}B_s=\text{corank}A.$  

\subsubsection{} 
Consider the case when $\mathfrak{g}$  is symmetrizable. Then, $\mathfrak{g}$ admits an even and non-degenerate bilinear form and $\mathfrak{g}=\mathfrak{g}(A,\tau),$ where $A$ is symmetric. 

Let $\iota:\mathfrak{h}\rightarrow\mathfrak{h^{*}}$ be the isomorphism induced by $(-,-)$. By the construction of the bilinear form, described in ~\cite{Kbook} Chapter 2, one has $\left\langle \alpha_i,h\right\rangle =\left(\alpha_{i}^{\vee},h\right)$  for any $h\in\mathfrak{h}$ and $\alpha_i\in\Sigma.$ It is easy to check, using the formulas described in ~(\ref{eq:isoralphaSigma}) and ~(\ref{eq:isoralphah}), that  $\iota(\alpha^{\vee})=\alpha$  for any real root. In particular, $B_\pi=\left(\left(\alpha,\beta\right)\right)_{\alpha,\beta\in\pi}.$ 

Let $\left(-,-\right)_{\pi}$ be a non-degenerate invariant bilinear form on $\mathfrak{g}(B_\pi)$ constructed as in ~\cite{Kbook}, Section 2. Since this form is invariant, the centre of $\mathfrak{g}(B_\pi)$ is orthogonal to $\left[\mathfrak{g}\left(B_{\pi}\right),\mathfrak{g}\left(B_{\pi}\right)\right].$ So $\mathfrak{c}$  lies in the kernel of $(-,-)_\pi$ when restricted to $\left[\mathfrak{g}\left(B_{\pi}\right),\mathfrak{g}\left(B_{\pi}\right)\right].$ Hence, $(-,-)_\pi$ induces an invariant bilinear form on $\mathfrak{g}_{pr}$ which we denote by $(-,-)'.$ 

\subsubsection{}
\begin{lem}{Coincide_bilinear}
    If $\mathfrak{g}$ is symmetrizable, then the restriction of $\left(-,-\right)$  to $\mathfrak{g}_{pr}$  coincides with $(-,-)'.$ 
\end{lem}
\begin{proof}
    By the definition of $\psi,$ we have $\psi{(\mathfrak{n}^{\pm}(B_\pi)})\subset\mathfrak{n}^{\pm},$ and $\psi$ sends the Cartan subalgebra of $\mathfrak{g}(B)_\pi$ to $\mathfrak{h}\cap\mathfrak{g}_{pr}.$ 
    Therefore
    \begin{equation}
    \left( \mathfrak{h} \cap \mathfrak{g}_{pr}, \mathfrak{g}_{pr} \cap \mathfrak{n}^{\pm} \right) 
    = \left( \mathfrak{h} \cap \mathfrak{g}_{pr}, \mathfrak{g}_{pr} \cap \mathfrak{n}^{\pm} \right)' = 0.
    \label{eq:intersection_bil_0}
\end{equation}
and
\begin{align*}
    \left( \mathfrak{g}_{pr} \cap \mathfrak{n}^{+}, \mathfrak{g}_{pr} \cap \mathfrak{n}^{+} \right) 
    &= \left( \mathfrak{g}_{pr} \cap \mathfrak{n}^{+}, \mathfrak{g}_{pr} \cap \mathfrak{n}^{+} \right)' = 0, \notag \\
    \left( \mathfrak{g}_{pr} \cap \mathfrak{n}^{-}, \mathfrak{g}_{pr} \cap \mathfrak{n}^{-} \right) 
    &= \left( \mathfrak{g}_{pr} \cap \mathfrak{n}^{-}, \mathfrak{g}_{pr} \cap \mathfrak{n}^{-} \right)' = 0.
\end{align*}
It remains to verify:
\begin{enumerate}
    \item the restrictions of $(-,-)$  and of $(-,-)'$  to $\mathfrak{g}_{pr}\cap\mathfrak{n}^{+}\otimes\mathfrak{g}_{pr}\cap\mathfrak{n}^{-}$  are equal;
    \item the restrictions of $(-,-)$  and of $(-,-)'$  to $\left(\mathfrak{g}_{pr}\cap\mathfrak{h}\right)\otimes\left(\mathfrak{g}_{pr}\cap\mathfrak{h}\right)$  are equal. 
\end{enumerate}
For all $\alpha\in\pi,$ we fix $e_{\pm\alpha}\in\mathfrak{g}_{\pm\alpha}.$ 
For $(i),$  recall that $\mathfrak{g}_{pr}\cap\mathfrak{n}^{\pm}$  is generated by $e_{\pm\alpha}$  with $\alpha\in\pi.$ We set $\mathcal{F}^{1}\left(\mathfrak{g}_{pr}\cap\mathfrak{n}^{+}\right)$  to be the span of $e_\alpha$  for $\alpha\in\pi$  and 
$$\mathcal{F}^{i+1}\left(\mathfrak{g}_{pr}\cap\mathfrak{n}^{+}\right):=\mathcal{F}^{i}\left(\mathfrak{g}_{pr}\cap\mathfrak{n}^{+}\right)+\left[\mathcal{F}^{1}\left(\mathfrak{g}_{pr}\cap\mathfrak{n}^{+}\right),\mathcal{F}^{i}\left(\mathfrak{g}_{pr}\cap\mathfrak{n}^{+}\right)\right].$$
One readily sees that $\left\{ \mathcal{F}^{i}\left(\mathfrak{g}_{pr}\cap\mathfrak{n}^{+}\right)\right\} _{i\in\mathbb{\mathbb{Z}}_{>0}}$  is an exhausting increasing filtration of $\mathfrak{g}_{pr}\cap\mathfrak{n}^{+}.$ Similarly we define a filtration $\left\{ \mathcal{F}^{i}\left(\mathfrak{g}_{pr}\cap\mathfrak{n}^{-}\right)\right\} _{i\in\mathbb{\mathbb{Z}}_{>0}}$  for $\mathfrak{g}_{pr}\cap\mathfrak{n}^{-}.$
Note that all terms $\mathcal{F}^{i}\left(\mathfrak{g}_{pr}\cap\mathfrak{n}^{\pm}\right)$  are $\ad\mathfrak{h}$-invariant and 
\begin{equation}
    \left[ \mathcal{F}^{1} \left( \mathfrak{g}_{pr} \cap \mathfrak{n}^{+} \right), 
    \mathcal{F}^{j+1} \left( \mathfrak{g}_{pr} \cap \mathfrak{n}^{-} \right) \right] 
    \subset \mathcal{F}^{j} \left( \mathfrak{g}_{pr} \cap \mathfrak{n}^{-} \right), 
    \quad \text{for } j \geq 1,
    \label{eq:filtration_bracket}
\end{equation}
 since $\left[e_{\alpha},f_{\beta}\right]\in\mathfrak{h}.$ 

Let us verify that the restiction of the two bilinear forms to $\mathcal{F}^{i}\left(\mathfrak{g}_{pr}\cap\mathfrak{n}^{+}\right)\otimes\mathcal{F}^{j}\left(\mathfrak{g}_{pr}\cap\mathfrak{n}^{-}\right)$
are equal. We have 
$$0\neq[e_\alpha,e_{-\alpha}]=(e_\alpha,e_{-\alpha})\alpha^{\vee}=(e_\alpha,e_{-\alpha})'\alpha^{\vee}$$
so $\left(e_{\alpha},e_{-\alpha}\right)=\left(e_{\alpha},e_{-\alpha}\right)'.$ 
For $\beta\neq\alpha,$  $\alpha,\beta\in\pi,$  we have $\left(e_{\alpha},e_{-\beta}\right)=\left(e_{\alpha},e_{-\beta}\right)'=0.$ Thus the assertion holds for $i=j=1.$ For $x_{\pm}\in\left(\mathfrak{g}_{pr}\cap\mathfrak{n}^{\pm}\right),$ since the bilinear forms are invariant, we have: 
$$\left(\left[e_{\alpha},x_{+}\right],x_{-}\right)=-\left(x_{+},\left[e_{\alpha},x_{-}\right]\right),\,\left(\left[e_{\alpha},x_{+}\right],x_{-}\right)'=-\left(x_{+},\left[e_{\alpha},x_{-}\right]\right)'.$$
Combining with ~(\ref{eq:filtration_bracket}), we conclude that the assertion for the pair $(i,j),$ where $i,j>1,$  follows from the assertion for the pair $(i-1,j-1).$ Similarly, for $j>1,$  the assertion for $(i,j)$ follows from ~(\ref{eq:intersection_bil_0}). Using induction, this establishes the assertion for all $(i,j)$  and proves $(i).$ 

For $(ii)$ note that: 
$$\mathfrak{g}_{pr}\cap\mathfrak{h}=\left[\mathfrak{g}_{pr}\cap\mathfrak{n}^{+},\mathfrak{g}_{pr}\cap\mathfrak{n}^{-}\right]\cap\mathfrak{h}.$$
Take any $h\in(\mathfrak{g}_{pr}\cap\mathfrak{h})$ and $a_{\pm}\in\left(\mathfrak{g}_{pr}\cap\mathfrak{n}^{\pm}\right).$ One has $\left(h,\left[a_{+},a_{-}\right]\right)=\left(\left[h,a_{+}\right],a_{-}\right)$ and similarly for $(-,-)'.$ Since $\left[h,a_{+}\right]\in(\mathfrak{g}_{pr}\cap\mathfrak{n}^{+}),$  the formula $\left(\left[h,a_{+}\right],a_{-}\right)=\left(\left[h,a_{+}\right],a_{-}\right)'$  follows from $(i).$ This proves $(ii)$ and completes the proof.  
\end{proof}
\subsubsection{}

\begin{cor}{restriction_of_bili}
    If $\mathfrak{g}$  is symmetrizable, then the kernel of the restriction of $(-,-)$ to $\mathfrak{g}_{pr}$ coincides with the centre of $\mathfrak{g}_{pr}$ and lies in $\mathfrak{h}\cap\mathfrak{g}_{pr}.$   
\end{cor}
\begin{proof}
    This follows from \Lem{Coincide_bilinear} and the fact that the kernel of the invariant bilinear form on $\left[\mathfrak{g}\left(B_{\pi}\right),\left(B_{\pi}\right)\right],$ is equal to $\mathfrak{c}(\mathfrak{g}(B_\pi)).$
\end{proof}
\subsection{Integrable modules}
We call a $[\mathfrak{g},\mathfrak{g}]-$module $V$ \textit{integrable} if for each real root $\alpha$  and each $u\in\mathfrak{g}_{\alpha},$  the action of $u$  on $V$  is locally nilpotent. We call a $\mathfrak{g}-$module integrable if it is integrable as $[\mathfrak{g},\mathfrak{g}]-$module. 
The adjoint representation of any Kac-Moody superalgebra is integrable (this follows from the definition). 

The following lemma is well-known. 

\subsubsection{}
\begin{lem}{Integrable_prinicpal}
    A $[\mathfrak{g},\mathfrak{g}]-$module $V$ is integrable if and only if for each $\alpha\in\pi$  and each $u\in\mathfrak{g}_{\pm\alpha},$  the action of $u$ on $V$ is locally nilpotent. 
\end{lem}
\begin{proof}
   Assume that for each $\alpha\in\pi$  and each $u\in\mathfrak{g}_{\pm\alpha}$ the action of $u$ on $V$  is locally nilpotent. Then for each $\alpha\in\Sigma_{pr}$  and each $u\in\mathfrak{g}_{\pm\alpha},$  the action is locally nilpotent. Let $\beta$  be a real root and let $u\in\mathfrak{g}_\beta$ be a non-zero element. If $\beta$  is isotropic, then by ~(\ref{eq:fsi}), $\mathfrak{g}\left\langle \beta\right\rangle \simeq\mathfrak{sl}\left(1|1\right).$ Thus, $u^2=0.$ If $\beta\in\Delta^{an},$  then by ~\ref{principal_anisotropic}, $\beta=w\gamma,$ where $w$ lies in the Weyl group and $\gamma\in\Sigma_{pr}.$ Moreover, the Weyl group $W$  is the Coxeter group generated by the reflections $s_\alpha$  with $\alpha\in\Sigma_{pr}.$ We will prove that $u$ acts locally nilpotent on $V$  by induction on the length of $w.$ 

   If the length is zero, then $u\in\mathfrak{g}_\gamma,$  and thus $u$  acts locally nilpotent. Otherwise, for some $\alpha\in\Sigma_{pr},$  we have $w=s_{\alpha}w',$  where the length of $w'$  is less than the length of $w.$  Then, $s_{\alpha}\beta=w'\gamma,$  and,  by induction, for any $u'\in\mathfrak{g}_{s_\alpha \beta}$  the action of $u'$  on $V$  is locally nilpotent. By ~\cite{Kbook}, Lemma 3.8, $$u=\left(\text{exp}\ad f\right)\left(\text{exp}\ad\left(-e\right)\right)\left(\text{exp}\ad f\right)\left(u'\right),$$ for some $u'\in\mathfrak{g}_{s_\alpha \beta},$ $e\in\mathfrak{g}_\alpha$  and $f\in\mathfrak{g}_{-\alpha}.$  

   Let $E,F,U,U'$  be the images of $e,f,u,u'$  in $\End(V).$ By (3.8.1) in ~\cite{Kbook},  one has $\left(\text{exp}a\right)b\left(\text{exp}-a\right)=\text{exp}\left(\text{ad}a\right)\left(b\right)$  for any operator $b$ on $V$  such that $\left(\ad a\right)^{N}b=0$  for some $N.$  Therefore we obtain 
\begin{align*}
    &\quad \quad \quad U = \left(\text{exp} \ad F\right)\left(\text{exp} \ad \left(-E\right)\right) 
    \left(\text{exp} \ad f\right)\left(U'\right) \\
    &\quad = \left(\text{exp} F\right)\left(\text{exp} \left(-E\right)\right) 
    \left(\text{exp} F\right)\left(U'\right) 
    \left(\text{exp} -F\right)\left(\text{exp} \left(-E\right)\right) 
    \left(\text{exp} -F\right).
\end{align*}
Take any $v\in V.$  Then 
$$U^{s}\left(v\right)=\left(\text{exp}F\right)\left(\text{exp}\left(-E\right)\right)\left(\text{exp}F\right)\left(U'\right)^{s}\left(\text{exp}-F\right)\left(\text{exp}\left(-E\right)\right)\left(\text{exp}-F\right)\left(v\right).$$
Since $E,F'U'$  acts locally nilpotent, $U^sv=0$  for $s$  large enough, so $U$  acts locally nilpotent as required. 
\end{proof}
\subsubsection{}
\label{Integrable for derived B_pi}
Recall that if $B_\pi$  is symmetrizable, then $\mathfrak{g}_{pr}\simeq\left[\mathfrak{g}\left(B_{\pi}\right),\mathfrak{g}\left(B_{\pi}\right)\right]/\mathfrak{c}.$ Thus, by \Lem{Integrable_prinicpal}, any integrable $\mathfrak{g}$-module has a structure of an integrable $\left[\mathfrak{g}\left(B_{\pi}\right),\mathfrak{g}\left(B_{\pi}\right)\right]-$module (which is annihilated by $
\mathfrak{c}).$ 
\subsection{Inner automorphisms}
Let $\mathfrak{s}$  be a Kac-Moody algebra with a Cartan subalgebra $\mathfrak{t}.$  An $[\mathfrak{s},\mathfrak{s}]-$module is called \textit{semisimply integrable} if it is integrable and the Cartan algebra acts semisimply. 

In ~\cite{Kac and Peterson}, the authors constructed a so-called minimal Kac-Moody group $\mathcal{S}_{min}$  associated to $\mathfrak{s}.$ 

The group $\mathcal{S}_{min}$  is constructed in the following way. First, we define $\mathcal{S}^{*}$ as the free product of the additive groups $\mathfrak{s}_\alpha$  for $\alpha\in\Delta^{re}$. 
Next, for any semisimply integrable representation $(V,\gamma)$ of the derived Lie algebra $\left[\mathfrak{s},\mathfrak{s}\right],$ we construct a corresponding representation of $\mathcal{S}^{*},$ denoted by $\gamma^{*}:\mathcal{S}^{*}\longrightarrow\text{Aut}\left(V\right),$ and defined as follows: $$\gamma^{*}\left(i_{\alpha}\left(e\right)\right)=\text{exp }\gamma\left(e\right),$$ where $i_{\alpha}:\mathfrak{s}_{\alpha}\longrightarrow\mathcal{S}^{*}$ is the natural embedding.  
Now, we define the subgroup $\mathcal{R}^{*}\subset\mathcal{S}^{*}$ as the intersection of all $\ker\gamma^{*}$ (running over all semisimply integrable representations of $\left[\mathfrak{s},\mathfrak{s}\right]$).
Finally, we define  $\mathcal{S}_{min}:=\mathcal{S}^{*}/\mathcal{R}^{*}.$ 

For any semisimply integrable representation $(V,\gamma)$ of the derived Lie algebra $\left[\mathfrak{s},\mathfrak{s}\right],$ we associate a representation $\left(V,\Gamma\right)$ of the group $\mathcal{S}_{min},$  taking for every $e\in\mathfrak{s}_\alpha,$  $$\Gamma\left(q\left(i_{\alpha}\left(e\right)\right)\right)=\exp\gamma\left(e\right),$$
 where $q:\mathcal{S}^{*}\longrightarrow\mathcal{S}_{min}$ is the canonical homomorphism.

 For the adjoint representation we denote the corresponding representation of $\mathcal{S}_{min}$ by $\Ad.$ For each $x\in\mathcal{S}_{min},$  the element $\Ad(x)$ is an algebra automorphism of $\mathfrak{s}.$  We call such automorphism an \textit{inner automorphism}.    
 \subsubsection{Conjugacy theorem for symmetrizable Kac-Moody algebra}

The following theorem is proven in ~\cite{Kac and Peterson}  and in ~\cite{Kumar}, Chapter 10. 

\begin{thm}{Conjugacy-KM}
    Let $\mathfrak{s}$ be a symmetrizable Kac-Moody algebra with a Cartan subalgebra $\mathfrak{t}.$ If $\mathfrak{t}'$  is an ad-diagonalizable commutative subalgebra of $\mathfrak{s},$ then there exists $x\in\mathcal{S}_{min}$  such that $\Ad(x)(\mathfrak{t}')\subset\mathfrak{t}.$
\end{thm}
\subsubsection{Inner automorphisms of an Kac-Moody superalgebras}
Let $\mathfrak{g}$  be a Kac-Moody superalgebra such that $B_\pi$  is symmetrizable. We construct $\mathcal{S}_{min}$  for $\mathfrak{g}(B_\pi).$ By ~\ref{Integrable for derived B_pi}, $\mathcal{S}_{min}$  admits a natural action on any integrable $\mathfrak{g}-$module. Note that this action preserves the $\mathbb{Z}_{2}-$grading on the module. In particular, $\mathcal{S}_{min}$ acts on the adjoint $\mathfrak{g}-$module and we denote this action by $\Ad.$ 
\subsubsection{}
\begin{lem}{Properties of Ad(x)}
\begin{enumerate}
        \item For each $x \in \mathcal{S}_{\min}$, the element  $\Ad(x)$ is an automorphism of the Lie superalgebra $\mathfrak{g}.$
        \item If $(-,-)$ is an invariant bilinear form on $\mathfrak{g}$, then $\Ad(x)$ preserves $(-,-)$.
\end{enumerate}
\end{lem}
\begin{proof}
    First note that $\Ad (x)$  is a composition of linear transformations of the form $\exp\ad  u,$  for some $u\in\mathfrak{g}_{\alpha},$ and $\alpha\in\pi.$ So without loss of generality, we may assume $\Ad\left(x\right)=\exp\ad u.$ Let $x,y\in\mathfrak{g}.$  We need to show that 
    $$\left[\exp\ad u\left(x\right),\exp\ad u\left(y\right)\right]=\exp\ad u(\left[x,y\right]).$$
The classical argument (see, for example,~\cite{Humphry}) is applicable here. We present this argument
for completeness.We prove by induction on $n$  that: 
$$\frac{\left(\ad u\right)^{n}}{n!}\left(\left[x,y\right]\right)=\sum_{i=0}^{n}\left[\frac{1}{i!}\left(\ad u\right)^{i}\left(x\right),\frac{1}{\left(n-i\right)!}\left(\ad u\right)^{n-i}\left(y\right)\right].$$
The case where $n=0$  is trivial. 
Next, we have:
\begin{align*}
    \frac{\left(\ad u\right)^{n}}{n!}\left(\left[x,y\right]\right) 
    &= \frac{1}{n!}\left[u,\left(\ad u\right)^{n-1}\left(\left[x,y\right]\right)\right] \\
    &\stackrel{*}{=} \frac{1}{n!} \left[u, (n-1)! \sum_{i=0}^{n-1} \left[\frac{1}{i!} \left(\ad u\right)^{i}(x),\frac{1}{(n-1-i)!} \left(\ad u\right)^{n-1-i}(y)\right] \right] \\
    &\stackrel{**}{=} \frac{1}{n} \sum_{i=0}^{n-1} \left( \left[\left[u,\frac{1}{i!}\left(\ad u\right)^{i}(x)\right], \frac{1}{(n-1-i)!} \left(\ad u\right)^{n-(i+1)}(y) \right] \right. \\
    &\quad + \left. \left[\frac{1}{i!} \left(\ad u\right)^{i}(x), \left[u, \frac{1}{(n-1-i)!} \left(\ad u\right)^{n-1-i}(y)\right]\right] \right) \\
    &= \sum_{i=0}^{n} \left[\frac{1}{i!} \left(\ad u\right)^{i}(x), \frac{1}{(n-i)!} \left(\ad u\right)^{n-i}(y)\right],
\end{align*}
where $*$  is by induction hypothesis, and $**$  is following from Jacobi identity and the fact that $u\in\mathfrak{g}_{\overline{0}}.$   
Now $(i)$ follows by the exact same proof as in \cite{Humphry}, Chapter 2.3. 

For $(ii),$ we wish to show $\left(x|y\right)=\left(\exp u\left(x\right)|\exp u\left(y\right)\right).$ Denote by $k$ the minimal integer such that $\left(\ad u\right)^{k+1}\left(x\right)=\left(\ad u\right)^{k+1}\left(u\right)=0.$ By a straightforward induction and the invariance of the bilinear form, we obtain: $$\left(\left(\ad u\right)^{i}\left(x\right)|\left(\ad u\right)^{j}\left(y\right)\right)=\left(-1\right)^{i+j}\left(\left(\ad u\right)^{j}\left(x\right)|\left(\ad u\right)^{i}\left(y\right)\right),$$
for every $i,j\in\mathbb{Z}_{\geq0}.$ In particular, $$\left((\ad u)^{i}(x)|y\right)=(-1)^{i}\left(x|(\ad u)^{i}(y)\right).$$ Therefore, 
\begin{align*}
    \left(\exp\ad u\left(x\right)|y\right) &= \left(x|y\right)+\left(\left[u,x\right]|y\right)+\dots+\left(\frac{\ad^{k}u\left(x\right)}{k!}|y\right) \\
    &= \sum_{i=0}^{k}\left(-1\right)^{i}\left(x|\,\ad^{i}u\left(y\right)\right) 
    = \sum_{i=0}^{k}\left(x|\,\ad^{i}\left(-u\right)\left(y\right)\right) 
    = \left(x|\exp\ad \left(-u\right)\left(y\right)\right)
\end{align*}
Thus, $\left(\exp\ad u\left(x\right)|y\right)=\left(x|\exp\ad (-u)\left(y\right)\right)$  for every $x,y\in\mathfrak{g}.$ This gives: $$\left(\exp\ad u\left(x\right)|\exp\ad u\left(y\right)\right)=\left(x|\exp\ad\left(-u\right)\left(\exp\ad u\left(y\right)\right)\right)=\left(x|y\right),$$ as wished.

\end{proof}
Using the Lemma above, from now on, for any $x\in\mathcal{S}_{min}$  we consider $\Ad(x)$  as an automorphism of $\mathfrak{g},$  and not only of $\mathfrak{g}(B_\pi),$  in case $B_\pi$ is symmetrizable. 

We say that $\phi\in\Aut(\mathfrak{g})$  is inner if $\phi=\Aut(x)$ for some $x\in\mathcal{S}_{min}.$

\section{Affine Kac-Moody Lie superalgebras}

In contrast to the (Ind) type, the affine Kac-Moody superalgebras can be constructed in a very concrete manner, making the relations within the algebra much easier to understand. For the affine symmetrizable case, this is achieved through a process known as \textit{affinization}, which will be explained in detail later in this chapter. 
As for the non-symmetrizable case, there are only one family, $S(2|1;b),$ with $b\in\mathbb{C}\backslash\mathbb{Z},$ and one series $\mathfrak{q}_{n}^{\left(2\right)},$  where $n\geq3.$ This was discovered by Crystal Hoyt (see ~\cite{H} and ~\cite{H-V}). 

\subsection{Process of affinization}

For the process of affinization, we fix a finite-dimensional Kac-Moody superalgebra $\mathfrak{\dot{g}},$ and an automorphism $\sigma$  of finite order. Additionally, this automorphism must preserve the non-degenerate invariant bilinear form on $\mathfrak{\dot{g}},$ which exists because $\mathfrak{\dot{g}}$ is of finite type. 

There is a one-to-one correspondence between affine Kac-Moody superalgebras, up to isomorphism, that can be obtained through affinization from a given $\mathfrak{\dot{g}},$ and the group $\text{Aut}_{\left(-,-\right)}\left(\mathfrak{\dot{g}}\right)/\text{Inn},$ where $\text{Aut}_{\left(-,-\right)}\left(\mathfrak{\dot{g}}\right)$  denotes the group of automorphisms of $\mathfrak{\dot{g}}$ preserving the bilinear form,  and $\text{Inn}$ is the subgroup of Inner automorphisms. The classification of $\text{Aut}\left(\mathfrak{g}\right)/\text{Inn}$ was completed by Vera Serganova (see \cite{Sint2}). This result directly translates into a classification of all possible affinizations of $\mathfrak{\dot{g}}.$ These algebras are symmetrizable, as the bilinear form on $\mathfrak{\dot{g}}$ naturally extends to its affinizations while remaining invariant and non-degenerate (details will be provided later in this chapter). 
Johan van de Leur, in his thesis, proved that all symmetrizable affine Kac-Moody superalgebras are obtained via affinization  (see ~\cite{vdL}). It follows that the classification of automorphisms in $\text{Aut}_{\left(-,-\right)}\left(\mathfrak{\dot{g}}\right)/\text{Inn},$  provides a complete classification of symmetrizable affine Kac-Moody superalgebras. In particular, all anisotropic affine Kac-Moody superalgebras are symmetrizable.  

The requirement for the automorphism to preserve the bilinear form is crucial. While every automorphism preserves the Killing form (this is easy to see from definitions), the Killing form can be degenerate in certain cases, for example, in $D\left(2|1;\alpha\right).$ In ~\cite{Sint2}, Serganova proved that, apart from the identity, there are no automorphisms of  $D\left(2|1;\alpha\right)$  that preserve the non-degenerate invariant bilinear form admitted by this algebra.   

We categorize the process of affinization into two cases:

\begin{itemize}
    \item \textbf{Non-twisted affinization}, where the automorphism is the identity. 
    \item \textbf{Twisted affinization}, where a non-trivial automorphism is used. 
\end{itemize}

\subsection{Non twisted affinization}
We begin the process with a simple finite-dimensional Kac-Moody superalgebra $\mathfrak{\dot{g}}$. Next, we define the \textit{non-twisted loop algebra} $\bigoplus_{k \in \mathbb{Z}} \mathbb{C} t^k \otimes \mathfrak{\dot{g}}$, with the relations $\left[t^i x, t^j y \right] = t^{i+j} \left[ x, y \right]$. The non-twisted Lie superalgebra of $\mathfrak{\dot{g}}$, denoted by $\mathfrak{g}$  (a notation used also for twisted affinization) and also by $\mathfrak{\dot{g}}^{\left(1\right)}$ (notation specifically for non twisted affinization), is obtained from a central extension of the loop algebra, together with a derivation $d = t \frac{d}{dt}$. Thus, we have $\mathfrak{g} = \bigoplus_{k \in \mathbb{Z}} \mathbb{C} t^k \otimes \mathfrak{\dot{g}} \oplus \mathbb{C} K \oplus \mathbb{C} d$, with the relations $\left[d, t^k a \right] = k t^k a$ and $\left[t^i x, t^j y \right] = t^{i+j} \left[ x, y \right] +i \delta_{i, -j} \left( x, y \right) \cdot K$, where $K$ is central, and $\left( \cdot, \cdot \right)$ is a fixed non-degenerate even bilinear form on $\mathfrak{\dot{g}}$ (the non-degenerate bilinear form is unique up to scalar).  

In particular, for $i\in\left\{ 0,1\right\} ,$ $\mathfrak{g}_{\overline{i}}=\bigoplus_{k\in\mathbb{Z}}\mathbb{C}t^{k}\otimes\mathfrak{\dot{g}}_{\overline{i}}\oplus\mathbb{C}K\oplus\mathbb{C}d$

We can extend the bilinear form on $\mathfrak{\dot{g}}$ to its affinization as follows:
\[
\left(t^i a, t^j b \right) = \delta_{i, -j} \left(a, b \right), \quad \left(K, K\right) = \left(d, d\right) = 0, \quad \left(K, d\right) = 1, \quad \left(a t^m, K\right) = \left(a t^m, d\right) = 0.
\]
  
\subsubsection{Root system and Cartan subalgebra}
The Cartan subalgebra of $\mathfrak{g}$ is $$\mathfrak{h} := \mathfrak{\dot{h}} \oplus \mathbb{C}K \oplus \mathbb{C}d,$$ where $\mathfrak{\dot{h}}$ is the Cartan subalgebra of $\mathfrak{\dot{g}}$. Fix a set of simple roots $\dot{\Sigma}$ for $\mathfrak{\dot{g}}$, and let $\theta$ be the highest root in $\mathfrak{\dot{g}}$ with respect to $\dot{\Sigma}$. For every $\alpha \in \dot{\Delta}$, we extend $\alpha$ to $\mathfrak{h}$ by setting $\alpha(K) = \alpha(d) = 0$.

We define a functional $\delta \in \mathfrak{h}^*$ by:
\[
\delta(\mathfrak{\dot{h}} \oplus \mathbb{C}K) = 0, \quad \delta(d) = 1.
\]
Thus, $\Sigma := \dot{\Sigma} \cup \{\delta - \theta\}$ forms a set of simple roots for $\mathfrak{g}$. We denote $\alpha_{0}:=\delta-\theta.$

The set of positive roots is:

$$\Delta_+ = \left\{ m\delta + \alpha \right\}_{m \in \mathbb{Z}_{>0}, \alpha \in \dot{\Delta}} \cup \left\{ m\delta \right\}_{m \in \mathbb{Z}_{>0}},$$

and the set of negative roots is:

$$\Delta_- = \left\{ m\delta + \alpha \right\}_{m \in \mathbb{Z}_{<0}, \alpha \in \dot{\Delta}} \cup \left\{ m\delta \right\}_{m \in \mathbb{Z}_{<0}}.$$

The even (resp. odd; isotropic) roots are of the form $m\delta + \alpha$, where $\alpha$ is either zero or an even (resp. odd; isotropic) root of $\mathfrak{\dot{g}}$.

The imaginary roots are of the form $m\delta$ for $m \neq 0$, and therefore $\Delta^{im} \subset \Delta_{\overline{0}}$.

\subsubsection{Root spaces and relations}

The root spaces are of the form $\mathfrak{g}_{\alpha + k\delta} = \mathbb{C}t^k e_{\alpha}$, and $\mathfrak{g}_{k\delta} = \{ t^k h \mid h \in \mathfrak{h} \}$, where $e_{\alpha} \in \mathfrak{\dot{g}}_{\alpha}$.

It follows immediately that $\left[\mathfrak{g}_{k\delta}, \mathfrak{g}_{m\delta}\right] = 0$ for $k \neq -m$. Since the restriction of $(-,-)$ to $\mathfrak{h}$  is non-degenerate, one has $\left[\mathfrak{g}_{k\delta},\mathfrak{g}_{-k\delta}\right]=\mathbb{C}K.$

\subsubsection{The non-twisted affine Kac-Moody algebra $\mathfrak{psl}(n|n)^{(1)}$} \label{psl_affine}
The algebra $\mathfrak{gl}\left(n|n\right)$  is the only indecomposable finite Kac-Moody superalgebra that is not simple. There are some issues when examining the non-twisted affine algebra:  $$\mathfrak{gl}\left(n|n\right)^{\left(1\right)}=\bigoplus_{k\in\mathbb{Z}}\mathbb{C}t^{k}\otimes\mathfrak{gl}\left(n|n\right)\oplus\mathbb{C}K\oplus\mathbb{C}d,$$ obtained through the same process as before. Indeed, denote $Id_{n}:=\left(\begin{array}{cc}
I_{n} & 0\\
0 & 0
\end{array}\right)$  and consider the element $tId_n.$ For a given set of simple roots for $\mathfrak{gl}\left(n|n\right),$  one can see that $tId_n\notin\widetilde{\mathfrak{n}}^{+},$  since $Id_n\notin\mathfrak{sl}\left(n|n\right)=\left[\mathfrak{gl}\left(n|n\right),\mathfrak{gl}\left(n|n\right)\right].$ Additionally, consider the element $tId.$ We claim that the ideal generated by this element has a zero intersection with the Cartan subalgebra. Indeed, since $Id$ is a central element in $\mathfrak{gl}(n|n),$ then for every $k\neq-1$ and $a\in\mathfrak{gl}(n|n),$ we have $\left[t^{k}a,tId\right]=0.$ For $k=-1,$  note that we must have $a\in \mathfrak{sl}\left(n|n\right),$  and therefore: $$\left[t^{-1}a,tId\right]=\underbrace{str\left(aId\right)}_{=0}K=0,$$  where $str$  is the supertrace. This implies that for $k\not=0$ the elements $t^kId_n$  and $t^kId$  lie in the ideal having zero intersection with $\mathfrak{h}.$ Thus, these elements have zero image in the Kac-Moody superalgebra. The affine Kac-Moody superalgebra, which we denote by
$\mathfrak{psl}\left(n|n\right)^{\left(1\right)},$ is the 
subquotient of $\mathfrak{gl}\left(n|n\right)^{\left(1\right)}$, which as
 a $\mathfrak{psl}(n|n)$-module, takes the form
$$\mathfrak{psl}(n|n)^{(1)}=\bigoplus_{k\in\mathbb{Z}\backslash\left\{ 0\right\} }\mathbb{C}t^{k}\otimes\mathfrak{psl}\left(n|n\right)\oplus\mathfrak{gl}\left(n|n\right)\oplus\mathbb{C}K\oplus\mathbb{C}d.$$

\subsection{Twisted affinization}
Let $\mathfrak{\dot{g}}$ be either a simple Kac-Moody superalgebra or $\mathfrak{psl}\left(n|n\right),$ and let $\sigma: \mathfrak{\dot{g}} \to \mathfrak{\dot{g}}$ be an automorphism of order $n$. Denote $\xi: = e^{\frac{2\pi i}{n}}$, so the eigenvalues of $\sigma$ are $$\{\xi^j \mid 0 \leq j \leq n-1\}.$$ Denote by $\mathfrak{\dot{g}}^j$ the eigenspace corresponding to $\xi^j$. For simplicity, we write $\mathfrak{\dot{g}}^0 = \mathfrak{\dot{g}}^\sigma,$ $\mathfrak{\dot{h}}^\sigma := \mathfrak{\dot{h}} \cap \mathfrak{\dot{g}}^\sigma,$ and $\overline{s} = s \mod n.$

This gives the following $\mathbb{Z}_{n}$ grading of $\mathfrak{\dot{g}}$:

$$\mathfrak{\dot{g}} = \bigoplus_{i=0}^{n-1} \mathfrak{\dot{g}}^i.$$

By \cite{vdL}, $\mathfrak{\dot{g}}^{\sigma}$ is a simple Kac-Moody superalgebra with Cartan subalgebra $\mathfrak{\dot{h}}^{\sigma}$, while each $\mathfrak{\dot{g}}^i$ is a module over $\mathfrak{\dot{g}}^{\sigma}$. Moreover, $\mathfrak{\dot{g}}^{1}$  is generated by its highest vector. 
We define the \textit{twisted loop algebra of order $n$} as

$$\bigoplus_{k \in \mathbb{Z}} \mathbb{C} t^k \otimes \mathfrak{\dot{g}}^{\overline{k}},$$

with the Lie bracket
\[
[t^i x, t^j y] = t^{i+j} [x, y].
\]

The twisted Lie superalgebra of $\mathfrak{\dot{g}}$, denoted by $\mathfrak{g}$ or $\mathfrak{\dot{g}}^{(n)}$, is obtained, similarly to the non-twisted case, by a central extension of the loop algebra, together with a derivation $d = t \frac{d}{dt}$. This gives:

$$\mathfrak{g} = \bigoplus_{k \in \mathbb{Z}} \mathbb{C} t^k \otimes \mathfrak{\dot{g}}^{\overline{k}} \oplus \mathbb{C}K \oplus \mathbb{C}d,$$

with the following relations:
\[
[d, t^k a] = k t^k a, \quad [t^i x, t^j y] = t^{i+j} [x, y] + \delta_{i, -j} (x, y) \cdot K,
\]
where $K$ is central and $(-, -)$ is a non-degenerate even bilinear form on $\mathfrak{\dot{g}}$. The bilinear form can be extended to its affinization, similarly as in the non-twisted case. 

We naturally extend the automorphism $\sigma$ to $\mathfrak{g}$ by defining:
$$\widetilde{\sigma}\left(t^{k}x\right)=\xi^{-k}t^{k}\sigma\left(x\right), \widetilde{\sigma}\left(d\right)=d, \widetilde{\sigma}\left(K\right)=K.$$ Then, the set of fixed points of $\widetilde{\sigma}$  coincides with the set of elements of the twisted algebra  $\mathfrak{g}.$ 
Indeed, $t^{k}x$  is a fixed point if and only if $t^{k}x=t^{k}\xi^{-k}\sigma\left(x\right).$ This is equivalent to $\sigma\left(x\right)=\xi^{k}x,$ that is, $x\in\mathfrak{\dot{g}}^{\overline{k}}.$ 
\subsubsection{Root system and Cartan subalgebra}
The Cartan subalgebra of $\mathfrak{g}$  is $$\mathfrak{{h}:=\mathfrak{\dot{h}}^{\sigma}\oplus\mathbb{C}}K\oplus\mathbb{C}d.$$
To construct a set of simple roots for $\mathfrak{g}$, we start with a set of simple roots $\dot{\Sigma}^\sigma$ of $\mathfrak{\dot{g}}^{\sigma}$. Let $\theta$ denote the unique highest weight of $\mathfrak{\dot{g}}^{\overline{1}}$ as a $\mathfrak{\dot{g}}^{\sigma}$-module. We define $\delta \in \mathfrak{h^{*}}$ by $\delta\left(\mathfrak{\dot{h}}^{\sigma}\oplus\mathbb{C}K\right) = 0$ and $\delta(d) = 1.$ The set of simple roots for $\mathfrak{g}$ is then given by $\Sigma := \dot{\Sigma}^\sigma \cup \left\{ \delta - \theta \right\}$. For all $\alpha\in\Sigma$, $\alpha^{\vee}\in [\fg,\fg]\cap \fh=\dot{\fh}^{\sigma}\oplus\mathbb{C}K$,
 so $\delta(\alpha^{\vee})=0$; hence $\delta$  lies in the kernel of the Cartan matrix $\mathfrak{h}.$ 

The construction of the root system in the twisted case is more complex than in the non-twisted case. Essentially, the roots take the form \( \alpha + m\delta \), where \( \alpha \) is a weight in  $\mathfrak{\dot{g}}^{\overline{m}}$  as a  $\mathfrak{\dot{g}}^{\sigma}$-module, and $m\in\mathbb{Z}$. The descriptions of roots given in ~\cite{vdL} and ~\cite{Sint} 
 imply that $\alpha+m\delta$  is a real root if and only if $\alpha$ is not the zero weight space of  $\mathfrak{\dot{g}}^{\overline{m}}$ as a $\mathfrak{\dot{g}}^{\sigma}$-module. The imaginary roots take the form $m\delta,$ where $m\in\mathbb{Z}.$ The parity of the roots is determined by the parity of the 
$\alpha$-weight vectors (where $\alpha$ can be also 0) of  $\mathfrak{\dot{g}}^{\overline{m}}$ as a  $\mathfrak{\dot{g}}^{\sigma}$-module.

A detailed, case-by-case description of the root systems is provided in~\cite{vdL}, Table 5, and explicit calculations can be found in~\cite{vdLbook}, Chapter 6.

The twisted case may exhibit significantly different properties compared to the non-twisted case. For example, in the algebra $A\left(2k,2l\right)^{\left(4\right)},$ not all imaginary roots are even. This algebra will be explored in greater detail in Chapter 4 of this thesis.

\subsubsection{Root spaces}
The real root spaces are of the form  $\mathfrak{g}_{\alpha+k\delta} = \mathbb{C} t^{k} e_{\alpha}$, where \( e_{\alpha} \) is a weight vector corresponding to the weight \( \alpha \) in  $\mathfrak{\dot{g}}^{\overline{k}}$  as a $\mathfrak{\dot{g}}^{\sigma}$-module (the non-zero weight spaces are one-dimensional).  
The imaginary root spaces are of the form:  

$$\mathfrak{g}_{k\delta} = \left\{ t^{k} v \, \middle| \, v \text{ is a zero-weight vector of } \mathfrak{\dot{g}}^{\overline{k}} \text{ as a } \mathfrak{\dot{g}}^{\sigma}\text{-module} \right\}.$$
 
In other words, $\mathfrak{g}_{m\delta}$ can be seen as the zero-weight space of $\mathfrak{\dot{g}}^{\overline{m}}$ by identifying $t^{m}\otimes v\in\mathfrak{g}_{m\delta}$ with $v$. This identification is an isomorphism of vector spaces. 

In contrast to the non-twisted case, it is not clear whether the imaginary root spaces commute when \( i + j \neq 0 \), and in fact, this is not always the case. In the next chapter, we will examine the structure of the algebra spanned by the Cartan subalgebra and the imaginary root spaces.

\subsection{Classification of all symmetrizable affine Kac-Moody superalgebras}
A full classification of the symmetrizable twisted and non-twisted affine Kac-Moody superalgebras, was provided by Johan van de Leur. Up to isomorphism, the non-twisted Kac-Moody superalgebras, which are not purly even, are:  
$$A\left(k,l\right)^{\left(1\right)},B\left(k,l\right)^{\left(1\right)},D\left(k,l\right)^{\left(1\right)},F\left(4\right)^{\left(1\right)},G\left(3\right)^{\left(1\right)}, D(2|1;\alpha)^{(1)},$$
where $A(k,k)^{(1)}=\mathfrak{psl}^{\left(1\right)}\left(k+1|k+1\right),$  as defined above. The twisted Kac-Moody superalgebras, up to isomorphism (of Kac-Moody superalgebras, but not $\mathbb{Z}_{n}-$grading), which are not purly even, are:
\begin{equation}
\begin{aligned}
A\left(2k, 2l-1\right)^{(2)}, &\quad A\left(2k-1, 2l-1\right)^{(2)} \text{, where } (k,l) \neq (1,1), \\
&\quad A\left(2k,2l\right)^{\left(4\right)}, \quad D\left(k,l\right)^{(2)}.
\end{aligned}
\label{eq:Classification affine}
\end{equation}

Originally, van de Leur included \( G\left(3\right)^{(2)} \) in the classification, however, $G(3)^2\simeq G(3)^1$ (see ~\cite{Shay}, Section 8).

Note that $C(l+1)=\mathfrak{osp}(2|2l)=D(1|l),$  so in fact $C\left(l+1\right)^{\left(2\right)}$  included in this list as well. 

A list of automorphisms that give rise to the above twisted Lie superalgebras is described in~\cite{vdL}, Table 4 (originally was found by Serganova in ~\cite{Sint2} as mentioned before). These automorphisms are explicitly defined on the Chevalley generators relative to a fixed set of simple roots, and this definition extends naturally to the entire Lie superalgebra. For a general formula describing automorphisms applicable to any chosen set of simple roots, see~\cite{vdLbook}, Chapter 6.

The automorphism of order 2 acts as a signed permutation on the Chevalley generators and can thus be interpreted as a signed automorphism of the Dynkin diagram of  $\mathfrak{\dot{g}}.$

\subsection{The principal Lie algebra $\mathfrak{g}_{pr}$ in the affine symmetrizable case} 
Assume $\mathfrak{g}$  is an indecomposable affine symmetrizable Kac-Moody superalgebra. So there exists finite dimensional Kac-Moody algebra $\mathfrak{\dot{g}},$ such that $\mathfrak{g}$ is its affinization (twisted or non-twisted). 
\subsubsection{}
\label{properties_by_section_9}
Recall $\mathfrak{g}_{pr}\simeq\left[\mathfrak{g}\left(B_{\pi}\right),\mathfrak{g}\left(B_{\pi}\right)\right]/\mathfrak{c},$ where $\mathfrak{c}$  is an ideal lies in the centre of $\mathfrak{g}(B_\pi).$
By ~\cite{Sint}, Section $9,$ the following hold:
\begin{enumerate}
    \item $\fg(B_\pi)=\fg(B_1)\times\fg(B_2)\times\cdot\cdot\cdot\times\fg(B_k),$ where $k\in\left\{ 1,2,3\right\}.$ In addition, $k=3$ if and only if $\mathfrak{g}=D\left(2|1;\alpha\right)^{\left(1\right)},$  $\mathfrak{g}=D\left(2|1;\alpha\right),$ or $D(2|n).$ 
    \item For each $i=1,...,k,$  we have that  $\psi$  maps the centre $\mathbb{C}K_i$  of $\mathfrak{g}(B_i)$ to $\mathbb{C}K.$ In particular, $\mathfrak{c}$  has zero intersection with $\left[\mathfrak{g}\left(B_{i}\right),\mathfrak{g}\left(B_{i}\right)\right],$  so we have an embedding of this algebra to $\mathfrak{g}_{pr}.$ This embedding extends to an embedding of $\mathfrak{g}(B_i)$  to $\mathfrak{g}_{pr}\oplus\mathbb{C}d.$ 
    \item If the centre of $\mathfrak{g}(B_\pi)$ is of dimension $k,$ then the central ideal $\mathfrak{c}$  is of dimension $k-1.$ 
\end{enumerate}
\subsubsection{}
\begin{lem}{non-deg}
    The kernel of the restriction of $(-,-)$ to $\mathfrak{g}_{pr}$ is $\mathbb{C}K,$  and the restriction to $\mathfrak{g}_{pr}\oplus\mathbb{C}d$  is non-degenerate. 
\end{lem}
\begin{proof}
   As already pointed out in the proof of \Prop{Cartan_intersect_pr}, one has $\mathfrak{c}\left(\mathfrak{g}_{pr}\right)=\mathbb{C}K.$ Thus, by \Cor{restriction_of_bili},
   $$\ker\left(-,-\right)|_{\mathfrak{g}_{pr}}=\mathfrak{c}\left(\mathfrak{g}_{pr}\right)=\mathbb{C}K.$$
   Finally, since $(K,d)=1,$  we get that the restriction of $(-,-)$ to $\mathfrak{g}_{pr}\oplus\mathbb{C}d$ is non-degenerate. 

\end{proof}

\subsubsection{Example: $\mathfrak{sl}^{\left(1\right)}\left(m|n\right)$} \label{ex-sl(m|n)}
    In ~\cite{Sint}, various examples are provided, showing how $\mathfrak{g}_{pr}$  and $\mathfrak{g}\left(B_\pi\right)$  looks like in different cases. We first give example in which $\mathfrak{g}\left(B_\pi\right)$ is  decomposable. Consider  $\mathfrak{g}=\mathfrak{sl}^{\left(1\right)}\left(m|n\right).$ 
    First note that $\mathfrak{sl}\left(m|n\right)_{\overline{0}}=\mathfrak{sl}_{m}\times\mathfrak{sl}_{n}\times\mathbb{C}z,$  where  $z=\left(\begin{array}{cc}
m\cdot I_{n} & 0\\
0 & n\cdot I_{m}
\end{array}\right).$
Let $K_1,$ $K_2$ and $K_3$ be the centrals and $d_1,d_2$  and $d_3$  be the derivations of the algebras $\mathfrak{sl}_{m}^{\left(1\right)},\mathfrak{sl}_{n}^{\left(1\right)}$  and $\left(\mathbb{C}z\right)^{\left(1\right)}$  respectively.
We have $\mathfrak{g}\left(B_\pi\right)=\mathfrak{sl}_{m}^{\left(1\right)}\times\mathfrak{sl}_{n}^{\left(1\right)},$  $\mathfrak{c}=\mathbb{C}\left(K_{1}-aK_{2}\right),$  for some $a\in\mathbb{C},$ and $\mathfrak{g}_{pr}=\left(\left[\mathfrak{sl}_{m}^{\left(1\right)},\mathfrak{sl}_{m}^{\left(1\right)}\right]\times\left[\mathfrak{sl}_{n}^{\left(1\right)},\mathfrak{sl}_{n}^{\left(1\right)}\right]\right)/\mathfrak{c}.$ Finally: $$\mathfrak{g}_{\overline{0}}=\left(\bigoplus_{k\in\mathbb{Z}}\mathbb{C}t^{k}z\times\left[\mathfrak{sl}_{m}^{\left(1\right)},\mathfrak{sl}_{m}^{\left(1\right)}\right]\times\left[\mathfrak{sl}_{n}^{\left(1\right)},\mathfrak{sl}_{n}^{\left(1\right)}\right]\right)/\left\langle K_{i}-K_{j}\right\rangle _{1\leq i<j\leq3}\rtimes\mathbb{C}d,$$ where  $d=d_{1}+d_{2}+d_{3}.$ We may think of this $d$  as the "merging" of derivations. By using this $d,$  we are effectively considering a subalgebra of the one we would have if we included all the three derivations  in our algebra. 
\subsubsection{Example: $\mathfrak{osp}\left(2|2\right)^{\left(2\right)}$}
\label{ex: osp(2|2)}
In this example, the Cartan datum is $\left(\left(\begin{array}{cc}
2 & -2\\
-2 & 2
\end{array}\right),\tau=\left\{ 1,2\right\} \right).$ Then $\left\{ \alpha,\delta-\alpha\right\} $  is the corresponding set of simple roots, which are odd. We see that $2\alpha,2\left(\delta-\alpha\right)$  are even roots. Moreover, $\pi=\left\{ 2\alpha,2\delta-2\alpha\right\},$  which gives $\mathfrak{g}\left(B_\pi\right)=\mathfrak{sl}_{2}^{\left(1\right)},$ $\mathfrak{g}_{pr}=\left[\mathfrak{sl}_{2}^{\left(1\right)},\mathfrak{sl}_{2}^{\left(1\right)}\right]/\mathbb{C}\left(K_{1}-K_{2}\right).$   So this is an example where $\mathfrak{g}(B)$  is indecomposable, and $\mathfrak{g}_{\overline{0}}=\mathfrak{sl}_{2}^{\left(1\right)}\oplus N,\,N=\bigoplus_{i\in\mathbb{Z}}\mathfrak{g}_{\left(2i+1\right)\delta}.$

\subsection{Structure of $\mathfrak{g}_{\overline{0}}$ in the symmetrizable case}
Given a finite dimensional Kac-Moody superalgebra $\mathfrak{\dot{g}}$ and its affinization (either twisted or non-twisted)  $\mathfrak{g},$
we aim to provide a clear and detailed description of the structure of  $\mathfrak{g}_{\overline{0}}.$  
To be more precise, we wish to find a $\mathfrak{g}_{pr}-$ module $\mathfrak{j}\subset\mathfrak{g}_{\overline{0}},$ such that $\mathfrak{g}_{\overline{0}}=\mathfrak{g}_{pr}\oplus\mathfrak{j}.$ 

\subsubsection{}
Denote: $$\mathfrak{g}_{\alpha,pr}:=\mathfrak{g}_{\alpha}\cap\fg_{pr}\text{ and }\Delta_{pr}:=\left\{ \alpha\in\Delta\,|\,\fg_{\alpha,pr}\neq0\right\}.$$
Note that we do not always have $\Delta_{pr}=\Delta_{\overline{0}}$  (see ~\ref{ex: osp(2|2)}). 
Let $h\in\mathfrak{h}$  be a generic element,  that is, $Re\left(\alpha\left(h\right)\right)\neq0$  for every $\alpha\in\Delta.$ The generic element defines a triangular decomposition on $\mathfrak{g}$, by setting $\Delta^{\pm}=\left\{ \alpha\in\Delta\,|\,\pm Re\left(\alpha\left(h\right)\right)>0\right\}.$
This induces a triangular decomposition on the Lie algebra $\mathfrak{g}_{pr}+\mathfrak{h}. $ \\  
For $\nu\in\mathfrak{h}^{*}$, we denote by $L_{\fg_{\pr}}(\nu)$ the corresponding 
highest weight simple module over $\fg_{\pr}+ \mathfrak{h}.$ 

The following theorem provides us an "invariant description" of $\mathfrak{g}_{pr},$ that is, a description which does not rely on the presentation of $\mathfrak{g}$  as a Kac-Moody superalgebra.  
\subsubsection{}
\begin{thm}{prop1}
\begin{enumerate}
\item
As $\ad (\fg_{\pr}+\mathfrak{h})$-module we have
$$\mathfrak{g}_{\overline{0}}=(\mathfrak{g}_{pr}+\mathfrak{h})\oplus\left(\bigoplus_{s\in\mathbb{Z}}L_{\mathfrak{g}_{pr}}\left(s\delta\right)^{\oplus m_{s}}\right)$$
where $m_0=0$.
\item One has $$\left[\left[\mathfrak{g}_{\overline{0}},\mathfrak{g}_{\overline{0}}\right],\left[\mathfrak{g}_{\overline{0}},\mathfrak{g}_{\overline{0}}\right]\right]=\mathfrak{g}_{pr}.$$
\end{enumerate}
\end{thm}
\begin{proof}
By \Lem{non-deg},  we can write:
$$\mathfrak{g}_{\overline{0}}=(\fg_{\pr}+\mathfrak{h})\oplus N,$$
where $N:=\{u\in \mathfrak{g}_{\overline{0}}| \ (\fg_{\pr}+ \mathfrak{h},u)=0\}$ is a $(\fg_{\pr} +\mathfrak{h})$-submodule of
$\mathfrak{g}_{\overline{0}}$. By ~(\ref{root_space_in_pr}), we have $\Omega(N)\subset\Delta^{im}\cap\Delta_{\overline{0}}.$ Thus,
$N=\oplus_{s\in \mathbb{Z}} N_{s\delta},$  where $N_{s\delta}$  are $s\delta$  weight spaces. Furthermore, for every $\alpha\in\pi,$ we have: $$[\fg_{\alpha,\pr}, N_{s\delta}]\subset\fg_{\alpha+s\delta}\cap N=\left\{0 \right\}.$$ Thus, $[\fg_{\pr}, N_{s\delta}]=0$, and $N_{s\delta}$ is a $(\mathfrak{g}_{\pr}+\mathfrak{h})$-submodule of $N$.  
Hence, $N_{s\delta}= L_{\fg_{\pr}}(s\delta)^{\oplus m_s},$ where $m_s:=\dim N_{s\delta}$. Since  $\mathfrak{h}\subset\fg_{\pr}+\mathfrak{h}$, clearly $m_0=0$. This proves (i).

Applying \Thm{sqdelta} along with $(i),$ we obtain $\left[N,N\right]=\mathbb{C}K.$ In addition, by $(i)$ it follow that $[\fg_{\pr}, N]=0,$ and by definition of $\fg_{pr}$ we have $[\fg_{pr},\fg_{pr}]=\mathfrak{g}_{pr}.$ Therefore
$$\left[\mathfrak{g}_{\overline{0}},\mathfrak{g}_{\overline{0}}\right]=[(\fg_{\pr}+\mathfrak{h})\oplus N,
(\fg_{pr}+\mathfrak{h})\oplus  N]=\fg_{pr}+\mathbb{C}K+[\fh,N]=\fg_{\pr}\oplus N.$$
Consequently, we derive: 
$$\left[\left[\mathfrak{g}_{\overline{0}},\mathfrak{g}_{\overline{0}}\right],\left[\mathfrak{g}_{\overline{0}},\mathfrak{g}_{\overline{0}}\right]\right]=[\fg_{\pr}\oplus N,\fg_{\pr}\oplus N]=
\fg_{pr}+[N,N]=\fg_{\pr}+\mathbb{C}K=\fg_{\pr}$$
as required.

\end{proof}
\subsubsection{}
The following lemma also will be useful for us.

\begin{lem}{Codim of g_pr} 
One has:
$$\dim\mathfrak{h}-\dim\left(\mathfrak{h}\cap\mathfrak{g}_{pr}\right)=\dim\mathfrak{c}\left(\mathfrak{g}_{\overline{0}}\right),$$ where $\mathfrak{c}\left(\mathfrak{g}_{\overline{0}}\right)$ is the centre of  $\mathfrak{g}_{\overline{0}}.$ 

\end{lem}
\begin{proof}
The Cartan subalgebra of  $\mathfrak{g}$  is given by  $\mathfrak{h}=\mathfrak{\dot{h}}^{\sigma}\oplus\mathbb{C}d\oplus\mathbb{C}K$,  where $\sigma$ is the automorphism used in the process of affinization. If $\mathfrak{g}$ is non-twisted, we can regard $\sigma$  as the identity map, in which case $\mathfrak{\dot{h}}=\mathfrak{\dot{h}}^{\sigma}.$  Let $\mathfrak{h''\subset\mathfrak{\dot{h}}^{\sigma}\oplus\mathbb{C}}K$  be a subalgebra such that $\mathfrak{\mathfrak{\dot{h}}^{\sigma}\oplus\mathbb{C}}K=\mathfrak{h''}\oplus\mathfrak{c}\left(\mathfrak{g}_{\overline{0}}\right).$ By \Prop{Cartan_intersect_pr}, we know that $\left(\mathfrak{h}\cap\mathfrak{g}_{pr}\right)+\mathfrak{c}\left(\mathfrak{g}_{\overline{0}}\right)=\mathfrak{h}\cap\left[\mathfrak{g},\mathfrak{g}\right],$ and $\left(\mathfrak{h}\cap\mathfrak{g}_{pr}\right)\cap\mathfrak{c}\left(\mathfrak{g}_{\overline{0}}\right)=\mathbb{C}K.$ Therefore,    $$\mathfrak{h}\cap\mathfrak{g}_{pr}=\mathfrak{h''}\oplus\mathbb{C}K.$$
Hence, we conclude: $$\dim\mathfrak{h}-\dim\left(\mathfrak{h}\cap\mathfrak{g}_{pr}\right)=\dim\mathfrak{c}\left(\mathfrak{g}_{\overline{0}}\right).$$
\end{proof}
\begin{rem}{}
   From the classification of symmetrizable affine Kac-Moody superalgebras, (see ~\cite{vdL}), it can be observed that $\dim\mathfrak{c}\left(\mathfrak{g}_{\overline{0}}\right)$ is either $1$ or $2,$ in almost every case, depending on the specific algebra. The only exception is $\mathfrak{psl}\left(n|n\right)^{\left(1\right)},$  where $\dim\mathfrak{c}\left(\mathfrak{g}_{\overline{0}}\right)=3.$ 
\end{rem}
\subsection{The queer twisted Lie superalgebra $\mathfrak{q}_{n}^{\left(2\right)}$}
\label{queer_twisted}
\subsubsection{Realization}\label{realization of queer twisted}
 The queer twisted Lie superalgebra is also constructed through the process of affinization, but in this case, it is applied to the queer Lie superalgebra  $\mathfrak{q}\left(n\right),$  which is not Kac-Moody. The automorphism  used in this process, is described in ~\cite{Sint2}  and defined as follow: it is acting by $id$  on the even part, and by $-id$ on the odd part (see ~\cite{Gorelik and Serganova} for details). Note that such an automorphism, when applied to a finite-dimensional Kac-Moody superalgebra, or $\mathfrak{psl}(n|n),$ is an inner automorphism. 
 
 We present the following realization of  $\mathfrak{q}_{n}^{\left(2\right)},$  as described in  ~\cite{Gorelik and Serganova}: 
 $$\mathfrak{g}_{\overline{0}}=\mathfrak{sl}_{n}\otimes\mathbb{C}\left[t^{\pm2}\right]\oplus\mathbb{C}K\oplus\mathbb{C}d,\,\mathfrak{g}_{\overline{1}}=\mathfrak{sl}_{n}\otimes\mathbb{C}\left[t^{\pm2}\right].$$  The relations are given by: $$\left[x\otimes t^{m},y\otimes t^{n}\right]=\left[x,y\right]\otimes t^{m+n}+\delta_{m,-n}\left(1-\left(-1\right)^{m}\right)\text{tr}\left(xy\right)K.$$
 One immediate consequence of these relations is that $K$  does not lie in  $\left[\mathfrak{g}_{\overline{0}},\mathfrak{g}_{\overline{0}}\right]$, but lies within $\left[\mathfrak{g}_{\overline{1}},\mathfrak{g}_{\overline{1}}\right].$
 
Moreover, as noted in ~\cite{Sint}, the even part admits the decomposition $\mathfrak{g}_{\overline{0}}= \mathfrak{g}_{pr}\oplus\mathbb{C}K\oplus\mathbb{C}D$, where  $\mathfrak{g}_{pr}\simeq\left[\mathfrak{sl}_{n}^{\left(1\right)},\mathfrak{sl}_{n}^{\left(1\right)}\right]/\mathbb{C}K.$ 
  
\subsubsection{}
\begin{claim}{form of pr}
    One has $\left[\mathfrak{g}_{\overline{0}},\mathfrak{g}_{\overline{0}}\right]=\mathfrak{g}_{pr}.$
    \begin{proof}
This result follows from a direct computation, using the realization: $$\left[\mathfrak{g}_{\overline{0}},\mathfrak{g}_{\overline{0}}\right]=\left[\mathfrak{sl}_{n}\otimes\mathbb{C}\left[t^{\pm2}\right]\oplus\mathbb{C}K\oplus\mathbb{C}d,\quad \mathfrak{sl}_{n}\otimes\mathbb{C}\left[t^{\pm2}\right]\oplus\mathbb{C}K\oplus\mathbb{C}d\right]=[\mathfrak{sl}^{(1)}_{n},\mathfrak{sl}^{(1)}_{n}]/\mathbb{C}K=\mathfrak{g}_{pr}.$$
\end{proof}
\end{claim}
\subsubsection{}\label{codim_queer}
Note that \Lem{Codim of g_pr}  does not hold in the twisted queer case. Indeed, $\mathfrak{c}\left(\mathfrak{g}_{\overline{0}}\right)=\mathbb{C}K,$  and $d,K\notin\mathfrak{h}\cap\mathfrak{g}_{pr},$  so $\dim\mathfrak{h}-\dim\mathfrak{h}\cap\mathfrak{g}_{pr}=2.$  
\subsection{Affine Kac-Moody superalgebra $S(2|1;b)$}
\label{S(2|1;b)}
\subsubsection{}\label{Definition of S(2|1;b)}
By definition, $S(2|1;b)=\mathfrak{g}(A_b,\tau_b),$ where $$A_b=\left(\begin{array}{ccc}
0 & b & 1-b\\
-b & 0 & 1+b\\
-1 & -1 & 2
\end{array}\right),\quad \tau_b=\left\{ 1,2\right\} .$$
The algebra $S(2|1;b)$ is a Kac-Moody superalgebra if and only if $b\not\in\mathbb{Z}.$
The matrix $A_b$  has corank $1$ and hence the algebra has one-dimensional center. In addition, the Cartan subalgebra $\mathfrak{h}$ is of dimension $4$.

A realization of this superalgebra can be found in ~\cite{Sint}. The Kac-Moody superalgebra $S(2|1;b)$ can not be constructed via the affinization of a Lie superalgebra. 
\subsubsection{Structure of  $\mathfrak{g}_{\overline{0}}$}\label{structure even S(2|1;b)}
According to ~\cite{Sint}, the even part of $\mathfrak{g}$ decomposes as:
$$\mathfrak{g}_{\overline{0}}=\mathfrak{g}_{pr}\oplus\mathbb{C}d+\mathcal{V}ir,\,\,\,\mathfrak{g}_{pr}\cap\mathcal{V}_{ir}=\mathbb{C}K,$$
where $\mathfrak{g}_{pr}\oplus\mathbb{C}d$ is isomorphic to $\mathfrak{sl}_{2}^{\left(1\right)}$ and $\mathcal{V}_{ir}=\bigoplus_{s\in\mathbb{Z}}\mathbb{C}L_{s}\oplus\mathbb{C}K$ is isomorphic to Virasoro algebra,  with the relations: 
$$\left[L_{s},L_{t}\right]=\left(s-t\right)L_{s+t}+\frac{K}{12}\left(s^{3}-s\right)\delta_{s+t,0}.$$
Here, $L_0-d$  lies in the centre of $\mathfrak{c}\left(\mathfrak{g}_{\overline{0}}\right),$  and $\mathfrak{c}\left(\mathfrak{g}_{\overline{0}}\right)=\mathbb{C}K+\mathbb{C}(L_0-d).$

\subsubsection{}
\begin{lem}{Inside pr}
    One has $\left[\mathfrak{g}_{pr},\mathcal{V}_{ir}\right]\subset\mathfrak{g}_{pr}$
\begin{proof}
Let $x\in\mathfrak{g}_{\alpha,pr}\text{,  where }\alpha\in\Delta^{re}$ is even.  By ~\cite{Sint}, for every $i\in\mathbb{Z}$  we have $L_{i}\in\mathfrak{g}_{-i\delta}$. Thus, we can state that $\left[x,L_{i}\right]\in\mathfrak{g}_{\alpha-i\delta}.$ Note that all the imaginary roots in 
$S(2|1;b)$ are even, and so $\alpha-i\delta$  is an even real root.
By ~(\ref{root_space_in_pr}), it follows that $\left[x,L_{i}\right]\in\mathfrak{g}_{pr}$. Additionally, $\mathfrak{h}\cap\mathfrak{g}_{pr}$ commutes with $\mathcal{V}_{ir}.$ Hence, the result follows. 
\end{proof}
\end{lem}
Thus, we may write: 
$$\mathfrak{g}_{\overline{0}}=((\mathfrak{g}_{pr}\rtimes\mathcal{V}_{ir})\rtimes\mathbb{C}d)/\mathbb{C}\left(K_{1}-K_{2}\right).$$ 
\subsubsection{}
It is no longer true that $\left[\left[\mathfrak{g}_{\overline{0}},\mathfrak{g}_{\overline{0}}\right],\left[\mathfrak{g}_{\overline{0}},\mathfrak{g}_{\overline{0}}\right]\right]=\mathfrak{g}_{pr}.$ Indeed, one has: $\mathcal{V}_{ir}\subset\left[\left[\mathfrak{g}_{\overline{0}},\mathfrak{g}_{\overline{0}}\right],\left[\mathfrak{g}_{\overline{0}},\mathfrak{g}_{\overline{0}}\right]\right].$ 

In addition, \Lem{Codim of g_pr} hold. Indeed, $\dim\mathfrak{c}\left(\mathfrak{g}_{\overline{0}}\right)=2,$  and $\dim\mathfrak{h}-\dim\mathfrak{h}\cap\mathfrak{g}_{pr}=2.$ 

\section{Structure of $\mathcal{K}$}
\subsection{The Lie superalgebra $\mathcal{K}$ }
Let $\mathfrak{\dot{g}}$  be simple Kac-Moody superalgebra (or $\mathfrak{psl}\left(n|n\right)$), and $\mathfrak{g}$  the algebra obtained by twisted affinization. Define the following subalgebra of $\mathfrak{g:}$   $$\mathcal{K}:=\bigoplus_{m\in\mathbb{Z}}\mathfrak{g}_{m\delta},$$ where $\mathfrak{g}_{0\cdot\delta}=\mathfrak{h}.$ Note that $\mathcal{K}$ is also well-defined in case when $\mathfrak{g}$ is affine and non-symmetrizable Kac-Moody superalgebra. However, we will focus on the symmetrizable case.   

The main goal of this section is to understand the relations within  $\mathcal{K}$  and to calculate the dimensions of the imaginary root spaces that constitute  $\mathcal{K}.$ We have already discussed these relations in the non-twisted case. For most of the twisted Kac-Moody superalgebras, $\mathcal{K}$ behaves the same. The only exception is $\mathfrak{g}=A(2k,2l)^{(4)},$  for which $\mathcal{K}$ has a surprisingly different structure. 

By ~\cite{vdL}, $\mathfrak{\dot{g}}^{\sigma}$  is isomorphic to either $B(k|l)$  or $D(k|l).$ If not stated otherwise, from now on, whenever we consider modules over $B(k|l)$  or $D(k|l),$  we will always use the fixed set of simple roots described in ~\ref{notation appendix}. Consequently, when treating $\mathfrak{\dot{g}}^{i}$  as $\mathfrak{\dot{g}}^{\sigma}$ module, we will adopt the corresponding fixed set of simple roots for $\mathfrak{\dot{g}}^{\sigma}.$

In this section will prove the following theorem:
\subsection{}
\begin{thm}{sqdelta}
Let $\fg$ be an indecomposable affine symmetrizable Kac-Moody superalgebra. 

    \begin{enumerate}
        \item  The subalgebra $\mathcal{K}_{\overline{0}}=\mathcal{K}\cap\mathfrak{g}_{\overline{0}},$ is an infinite-dimensional Heisenberg Lie algebra.
        \item If $\mathfrak{g}\neq A\left(2k|2l\right)^{\left(4\right)},$  then $\mathcal{K}=\mathcal{K}_{\overline{0}}$  and thus by $(i)$:  
    $$\left[\mathfrak{g}_{s\delta},\mathfrak{g}_{q\delta}\right]=\begin{cases}
    0 & s+q\neq0\\
    \mathbb{C}K & s+q=0
    \end{cases}$$
        \item If $A\left(2k|2l\right)^{\left(4\right)},$  then $\mathfrak{g}_{s\delta}\subset\mathcal{K}\cap\mathfrak{g}_{\overline{1}}$  if and only if $\overline{s}=1$  or $\overline{s}=3,$  where $$\overline{s}=s\mod 4.$$  

    If $i+j=0,$  then $\left[\mathfrak{g}_{i\delta},\mathfrak{g}_{j\delta}\right]=\mathbb{C}K.$  Otherwise, we have: 
     $$\left[\mathfrak{g}_{i\delta},\mathfrak{g}_{j\delta}\right] =
        \begin{cases}
        0 & i \neq 0, \, \overline{i} = 0 \\
        0 & \overline{i+j} = 0 \\
        \mathfrak{g}_{(i+j)\delta} & \overline{i} = 2, \, \overline{j} = 1,3 \\
        \mathbb{C} x, \quad x \text{ is a nonzero element in } \mathfrak{g}_{(i+j)\delta} & \overline{i} = \overline{j} = 1,3
    \end{cases}$$
    \end{enumerate}
\end{thm}

\subsection{Centralizer of $\mathfrak{\dot{h}}^\sigma$}
We start with some properties of the centralizer of $\mathfrak{\dot{h}}^\sigma$ in $\mathfrak{\dot{g}},$ defined as: $$\mathfrak{z} := \{x \in \mathfrak{\dot{g}} \mid [h, x] = 0 \text{ for all } h \in \mathfrak{\dot{h}}^\sigma\}.$$
It is easy to see that:
\begin{equation}
    \mathfrak{z} = \mathfrak{\dot{h}} + \sum_{\alpha \in\dot{\Delta}: \, \langle \alpha, \mathfrak{\dot{h}}^{\sigma} \rangle = 0} \mathfrak{g}_{\alpha},
    \label{eq:center_decomposition}
\end{equation}
where $\mathfrak{\dot{h}}$  is the Cartan subalgebra of $\mathfrak{\dot{g}}.$  In addition, since $\mathfrak{z}$  is $\sigma-$invariant, we decompose it as $$\mathfrak{z}=\bigoplus_{i=0}^{m-1}\mathfrak{z}^{i},$$  
where $\mathfrak{z}^{i}=\mathfrak{z}\cap\mathfrak{\dot{g}}^{i},$  and $m$  is the order of $\sigma.$  

\subsubsection{}
\begin{lem}{centralizer}
    Let $m\in\mathbb{Z}\backslash\left\{ 0\right\}.$ Then $\mathfrak{g}_{m\delta}=\mathbb{C}t^{m}\otimes\left(\mathfrak{z}\cap\mathfrak{\dot{g}}^{\overline{m}}\right).$
    \begin{proof}
        Let $x=\sum_{k\in\mathbb{Z}}t^{k}x_{\overline{k}}+c_{1}K+c_{2}d$, where $x_{\overline{k}}\in\mathfrak{\dot{g}}^{\overline{k}}$.
Directly from definitions, it follows that:
$x\in\mathfrak{g}_{m\delta}$ $\longleftrightarrow$ $\left[h,x\right]=m\delta\left(h\right)x$ for every $h\in\mathfrak{h}$ $\longleftrightarrow$ $\left[h,x\right]=0$ for every $h\in\mathfrak{h}\cap\left[\mathfrak{g},\mathfrak{g}\right]$ and $\left[d,x\right]=mx$ $\longleftrightarrow$ $\left[x_{\overline{k}},h\right]=0$ for every $k\in\mathbb{Z}$, $h\in\mathfrak{h}\cap\left[\mathfrak{g},\mathfrak{g}\right]$, $c_{1}=c_{2}=0$ and $x_{\overline{k}}=0$ for $k\neq m$ $\longleftrightarrow$ $x\in\mathbb{C}t^{m}\otimes(\mathfrak{z}\cap\mathfrak{\dot{g}}^{\overline{m}}).$
    \end{proof}
\end{lem}
  \subsubsection{}
  The following is an analogue of Lemma 8.1 in ~\cite{Kbook}.
  
  \begin{lem}{Properties of centralizer}
      Assume that $\sigma$  is of order $2r.$ 
      \begin{enumerate}
          \item For all $i$  we have $\left[\mathfrak{z}^{2r-i},\mathfrak{z}^{i}\right]=0.$
          \item 
$\text{If } \mathfrak{\dot{h}} \subset \mathfrak{z}^{0} \oplus \mathfrak{z}^{r}, \text{then  }  \mathfrak{\dot{h}} = \mathfrak{z}^{0} \oplus \mathfrak{z}^{r}.$

      \end{enumerate}
  \end{lem}
  \begin{proof}
      The bilinear form $(-,-)$  is non-degenerate when restricted to both $\mathfrak{\dot{h}}$  and each space $\mathfrak{g}_\alpha+\mathfrak{g}_{-\alpha}$  for $\alpha\in\Delta.$ 
Consequently, by ~(\ref{eq:center_decomposition}), its restriction to $\mathfrak{z}$  remains non-degenerate. 

As $\mathfrak{\dot{h}}^{\sigma}$  is the Cartan subalgebra of $\mathfrak{\dot{g}}^{\sigma},$ we obtain $\mathfrak{z}^{0}=\mathfrak{\dot{h}}^{\sigma}.$  

For any $x\in\mathfrak{z}^{i}$  and $y\in\mathfrak{z}^{j},$ where $i+j\neq0,$  the $\sigma-$invariance of $(-,-)$  implies: 
$$\left(x,y\right)=\left(\sigma\left(x\right),\sigma\left(y\right)\right)=\left(\xi^{i}x,\xi^{j}y\right)=\xi^{i+j}\left(x,y\right).$$ 
Thus $\left(\mathfrak{z}^{i},\mathfrak{z}^{j}\right)=0,$ so the restriction of $(-,-)$  to $\mathfrak{z}^{0}$ is non-degenerate. One has 
$$\left(\left[\mathfrak{z}^{2r-i},\mathfrak{z}^{i}\right],\mathfrak{z}^{0}\right)=\left(\mathfrak{z}^{2r-i},\left[\mathfrak{z}^{i},\mathfrak{z}^{0}\right]\right)=0.$$
Since $\left[\mathfrak{z}^{2r-i},\mathfrak{z}^{i}\right]\subset\mathfrak{z}^{0},$  this implies $\left[\mathfrak{z}^{2r-i},\mathfrak{z}^{i}\right]=0$ and proves $(i).$ 

For $(ii)$, as  $\left[\mathfrak{z}^{r},\mathfrak{z}^{r}\right]=0,$ we obtain $\left[\mathfrak{z}^{0}\oplus\mathfrak{z}^{r},\mathfrak{z}^{0}\oplus\mathfrak{z}^{r}\right]=0.$ 
As $\mathfrak{\dot{h}}\subset\mathfrak{z}^{0}\oplus\mathfrak{z}^{r},$  it gives $\left[\mathfrak{\dot{h}},\mathfrak{z}^{0}\oplus\mathfrak{z}^{r}\right]=0.$  
Since $\mathfrak{\dot{g}^{\mathfrak{\dot{h}}}=\mathfrak{\dot{h}}},$ this gives $\mathfrak{z}^{0}\oplus\mathfrak{z}^{r}\subset\mathfrak{\dot{h}}$  as required.  
  \end{proof}
    \subsubsection{}
  \begin{lem}{Centralizer_of_K}
      One has $$\left\{ u\in\mathfrak{g}_{\overline{0}}\,|\,\left[u,\mathfrak{h}\cap\mathfrak{g}_{pr}\right]=0\right\} =\mathcal{K}_{\overline{0}}.$$ 
  \end{lem}
\begin{proof}
    Since $\mathfrak{\dot{h}}^\sigma$ is the Cartan subalgebra of $\mathfrak{\dot{g}}^\sigma,$ we have 
    $$\left(\mathfrak{h}\cap\mathfrak{g}_{pr}\right)\subset\left(\mathfrak{h}\cap\left[\mathfrak{g},\mathfrak{g}\right]\right)\subset\left(\mathfrak{\dot{h}}^{\sigma}\oplus\mathbb{C}K\right).$$ Hence, $\mathcal{K}_{\overline{0}}\subset\left\{ u\in\mathfrak{g}_{\overline{0}}\,|\,\left[u,\mathfrak{h}\cap\mathfrak{g}_{pr}\right]=0\right\} .$ Assume, for sake of contradiction, that there exists $u\not\in\mathcal{K}_{\overline{0}}$ and $\left[u,\mathfrak{h}\cap\mathfrak{g}_{pr}\right]=0.$  Without loss of generality, we can assume that $u\in\mathfrak{g}_\gamma$  for $\gamma\in\Delta_{\overline{0}}\backslash\mathbb{Z}\delta.$ Therefore $\gamma$  is not an imaginary root, and thus, by ~(\ref{root_space_in_pr}), $\mathfrak{g}_{\gamma}\subset\mathfrak{g}_{pr}.$ So $\gamma^{\vee}\in\left(\mathfrak{h}\cap\mathfrak{g}_{pr}\right)$  and then $$\left[\gamma^{\vee},u\right]=\gamma\left(\gamma^{\vee}\right)u\neq0,$$
a contradiction. 
\end{proof}
\subsection{Twisted affine Lie superalgebras obtained by an automorphism of order $2$ }
Every symmetrizable twisted affine Lie superalgebra, with exception of $A^{\left(4\right)}\left(2k,2l\right),$ arises from an automorphism of order 2, see (\ref{eq:Classification affine}). 
In this subsection, $\mathfrak{\dot{g}}$  denote one of the simple superalgebras from the classification, and  $\mathfrak{g}$ is the twisted Lie superalgebra associated to an automorphism of order $2.$

We begin with a lemma that provides the relations within $\mathcal{K}$ in these cases.

\subsubsection{}

\begin{lem}{g^1 in Cartan}
One has:
\begin{enumerate}
    \item   $\mathfrak{z}=\mathfrak{\dot{h}}^{\sigma}\oplus(\mathfrak{\dot{h}}\cap\mathfrak{\dot{g}}^{1}).$ 
    \item  $\left[\mathfrak{g}_{s\delta},\mathfrak{g}_{q\delta}\right]=\begin{cases}
0 & s+q\neq0\\
\mathbb{C}K & s+q=0
\end{cases}$ 
\end{enumerate}
\begin{proof}
In case that $\sigma$  is of order $2r$  with $r=1,$  it is clear that $\mathfrak{\dot{h}}\subset\mathfrak{z}^{0}\oplus\mathfrak{z}^{1},$  and then $(i)$  is given by \Lem{Properties of centralizer}. 

For $(ii)$ note that: $$\mathfrak{g}_{k\delta}=\left\{ t^{k}\otimes x\,|\,x\in\mathfrak{\dot{h}}\cap\mathfrak{\dot{g}}^{1}\right\}$$ for $\overline{k}=1$. Now the result follows from definition of the twisted affine superalgebra and commutativity of $\mathfrak{\dot{h}}$. The case $s+q=0$  follows from \cite{Kbook}, Theorem 2.2. 
\end{proof}
\end{lem}
\subsubsection{}
\begin{lem}{g^1 module}
\begin{enumerate}
    \item 
    In case $\mathfrak{\dot{g}}=A\left(m|2l-1\right),$ as $\mathfrak{\dot{g}}^{\sigma}$ modules we have: $S^{2}\left(L\left(\epsilon_{1}\right)\right)=\mathfrak{\dot{g}}^{1}\oplus\text{triv},$ and $\mathfrak{\dot{g}}^{1}= L\left(2\epsilon_{1}\right)$. 
    \item 
    In case $\mathfrak{\dot{g}}=D\left(k|l\right)$, we have: $\mathfrak{\dot{g}}^{1}=L\left(\epsilon_{1}\right)$  as $\mathfrak{\dot{g}}^{\sigma}$-module.   
\end{enumerate}
\begin{proof}
    We prove this by examining each case individually:
    \begin{enumerate}
        \item Assume $\mathfrak{\dot{g}}=A\left(2k|2l-1\right)$. 
        By ~\cite{vdL}, $\mathfrak{\dot{g}}^{\sigma}\simeq B\left(k|l\right)$.
Referring to the roots table in ~\cite{vdL}, we see that $\mathfrak{\dot{g}}^{1}$, as $\mathfrak{\dot{g}}^{\sigma}$-module, has the following set of weights: $$\Omega_{\mathfrak{\dot{g}}^{1}}=\left\{ \underbrace{0;\pm\epsilon_{i}\pm\epsilon_{j};\pm\epsilon_{i};\pm\delta_{g}\pm\delta_{h};\pm2\epsilon_{i}}_{\text{even}};\underbrace{\pm\delta_{g},\pm\epsilon_{i}\pm\delta_{g}}_{\text{odd}}\right\}, 1\leq i,j\leq k,\,g,1\leq h\leq l.$$
Each weight space, aside from $0$, is one-dimensional. By \Lem{g^1 in Cartan}, it follows that $\dim\mathfrak{g}_{\delta}=\dim\mathfrak{\dot{h}}-\dim\mathfrak{\dot{h}}^{\sigma}=k+l.$ Thus, using \Prop{symmetric standard module} and \Cor{Final}, the results follow. \\
Now, assume $\mathfrak{\dot{g}}=A\left(2k-1|2l-1\right).$ By ~\cite{vdL}, we know that $\mathfrak{\dot{g}}^{\sigma}\simeq D\left(k|l\right)$ and hence $\dim\mathfrak{\dot{h}}^{\sigma}=k+l$. From the roots table, we find that the set of weights of $\mathfrak{\dot{g}}^1$ is given by: $$\Omega_{\mathfrak{\dot{g}}^{1}}=\left\{ \underbrace{0;\pm\epsilon_{i}\pm\epsilon_{j};\pm\delta_{g}\pm\delta_{h};\pm2\epsilon_{i}}_{\text{even}},\underbrace{\pm\epsilon_{i}\pm\delta_{j}}_{\text{odd}}\right\}, 1\leq i,j\leq k,\,g,1\leq h\leq l.$$
By \Lem{g^1 in Cartan}, we have $\dim\mathfrak{g}_{\delta}=\dim\mathfrak{\dot{h}}-\dim\mathfrak{\dot{h}}^{\sigma}=k+l-1.$ Thus, applying \Prop{symmetric standard module} and \Cor{Final}, the results follows. 
    \item Assume
    $\mathfrak{\dot{g}}=D(1|l)=\,C\left(l+1\right)$. By ~\cite{vdL}, $\mathfrak{\dot{g}}^{\sigma}\simeq B\left(0|l\right)$ and therefore $\dim\mathfrak{\dot{h}}^{\sigma}=l$. From the roots table we see that the set of weights of $\mathfrak{\dot{g}}^1$ is $\Omega_{\mathfrak{\dot{g}}^{1}}=\left\{ \underbrace{0}_{\text{even}},\underbrace{\pm\delta_{g}}_{\text{odd}}\right\}$, $1\leq g\leq l$. By \Lem{g^1 in Cartan}, we find that $\dim\mathfrak{g}_{\delta}=\dim\mathfrak{\dot{h}}-\dim\mathfrak{\dot{h}}^{\sigma}=1$. Thus, using \Lem{Standard module}, the result follows. \\
Now assume
$\mathfrak{\dot{g}}=D\left(k+1|l\right)$. According to ~\cite{vdL}, we have $\mathfrak{\dot{g}}^{\sigma}\simeq B\left(k|l\right),$ which implies $\dim\mathfrak{\dot{h}}^{\sigma}=k+l$. From the roots table we find that the set of weights of $\mathfrak{\dot{g}}^{1}$ is: $$\Omega_{\mathfrak{\dot{g}}^{1}}=\left\{ \underbrace{0,\pm\epsilon_{i}}_{\text{even}},\underbrace{\pm\delta_{g}}_{\text{odd}}\right\} ,\,1\leq g\leq l,1\leq i\leq k$$By \Lem{g^1 in Cartan}, we have $\mathfrak{g}_{\delta}=\dim\mathfrak{\dot{h}}-\dim \mathfrak{\dot{h}}^{\sigma}=1.$ Hence, using \Lem{Standard module}, the result follows.
    \end{enumerate}
\end{proof}
\end{lem}

\subsection{The twisted affine  Lie superalgebra $A^{\left(4\right)}\left(2k,2l\right)$}
 For this subsection assume $\mathfrak{\dot{g}}=A\left(2k,2l\right),$ that is $\mathfrak{\dot{g}}=\mathfrak{psl}\left(2k+1,2l+1\right)$  if $k=l$  and $\mathfrak{\dot{g}}=\mathfrak{sl}\left(2k+1,2l+1\right)$  otherwise. For brevity, we will often use the notation: $\mathfrak{\dot{g}}=\left(\mathfrak{p}\right)\mathfrak{sl}\left(2k+1,2l+1\right).$ Equipped with an automorphism of order $4$ on $\mathfrak{\dot{g}}$,  we obtain a $\mathbb{Z}_{4}-$grading of $\mathfrak{\dot{g}}$  and the associated twisted affine Lie superalgebra $A\left(2k|2l\right)^{\left(4\right)}.$ 
For clarity and ease of reference in the upcoming discussion, we define the subspace $\mathfrak{\dot{g}}^{\#}:=\mathfrak{\dot{g}}^{\sigma}\oplus\mathbb{C}t\otimes\mathfrak{\dot{g}}^{1}\oplus\mathbb{C}t^{2}\otimes\mathfrak{\dot{g}}^{2}\oplus\mathbb{C}t^{3}\otimes\mathfrak{\dot{g}},$  which can be viewed as the embedding of $\mathfrak{\dot{g}}$ as vector space, into $\mathfrak{g}.$     
\subsubsection{}
\begin{lem}{dim}
\begin{enumerate}
\item $\mathfrak{\dot{h}}=\mathfrak{\dot{h}}^{\sigma}\oplus\mathfrak{z}^{2}.$  

Additionally,     
$\dim\mathfrak{\dot{h}}^{\sigma}=k+l$  and  $\dim\mathfrak{g}_{2\delta}=\begin{cases}
k+l+1 & k\neq l\\
2k & k=l
\end{cases}$
\item  $\mathfrak{g}_{\delta},\mathfrak{g}_{3\delta}\subset\mathfrak{g}_{\overline{1}},$ 
 and $\mathfrak{z}_{\overline{1}}=\mathfrak{z}^{1}\oplus\mathfrak{z}^{3}.$  Moreover, 
$$\dim\mathfrak{g}_{\delta}=\dim\mathfrak{g}_{3\delta}=1.$$
\end{enumerate}
\begin{proof}
First, by ~\cite{vdL}, $\mathfrak{\dot{g}}^{\sigma}$ is isomorphic as Lie superalgebra to $B(k|l)$. Since $\mathfrak{\dot{h}}^{\sigma}$ is the Cartan subalgebra of $\mathfrak{\dot{g}}^{\sigma}$, we conclude that $\dim\mathfrak{\dot{h}}^{\sigma}=k+l$. 

Next, note that if $\mathfrak{\dot{h}} \not\subset \mathfrak{z}^{0} \oplus \mathfrak{z}^{2},$  then since $\mathfrak{z}$ contains $\mathfrak{\dot{h}},$  there exists some $a\in\mathfrak{z}^{i}\cap\mathfrak{\dot{h}},$  where $i$ is either $1$  or $3.$  Thus, $\Delta_{\overline{0}}^{im}$  contains $\delta$  or $3\delta$  accordingly. This is contradicts the roots table of $\mathfrak{g},$  presented in ~\cite{vdL}.

Therefore, ~\Lem{Properties of centralizer} gives $\mathfrak{\dot{h}}=\mathfrak{z}^{0}\oplus\mathfrak{z}^{2}.$ 

Thus, $$\dim\mathfrak{z}\cap\mathfrak{\dot{g}}^{2}=\dim\mathfrak{\dot{h}}-\dim\mathfrak{z}\cap\mathfrak{\dot{g}}^{0}=\begin{cases}
k+l+1 & k\neq l\\
2k & k=l
\end{cases}$$ Finally, \Lem{centralizer} implies $\dim\mathfrak{g}_{2\delta}=\dim\mathfrak{z}\cap\mathfrak{\dot{g}}^{2}.$ 

For (ii), following the roots table of $\mathfrak{g}$ presented in ~\cite{vdL}, we see that there are $4k+4l+8kl$ real odd roots within $\mathfrak{\dot{g}}^{\#}$ and $\dim\mathfrak{\dot{g}}_{\overline{1}}^{\#}=\dim\mathfrak{(p)sl}\left(2k+1,2l+1\right)_{\overline{1}}=8kl+4k+4l+2.$ By considering the difference, we conclude that $\dim\mathfrak{\dot{g}}_{\overline{1}}^{\#}\cap\mathcal{K}=2.$ Following the roots table, we see that $\mathfrak{\dot{g}}^{1} \text{ and }\mathfrak{\dot{g}}^{3}$ have odd imaginary roots. Therefore, $\dim\mathfrak{g}_{\delta}=\mathfrak{g}_{3\delta}=1,$  $\mathfrak{g}_{\delta},\mathfrak{g}_{3\delta}\subset\mathfrak{g}_{\overline{1}},$  and then by \Lem{centralizer},  $\mathfrak{z}_{\overline{1}}=\mathfrak{z}^{1}\oplus\mathfrak{z}^{3}.$   
\end{proof}
\end{lem}
\subsubsection{}
\begin{cor}{Dimensions}
$$\dim\mathfrak{g}_{m\delta}=\begin{cases}
k+l & \overline{m}=0,\,m\neq0\\
1 & \overline{m}=1, \overline{m}=3\\
k+l+1 & \overline{m}=2,\,k\neq l\\
2k & \overline{m}=2,k=l
\end{cases}$$
\begin{proof}
Since $\mathfrak{g}_{m\delta}$ isomorphic to $\mathfrak{g}_{q\delta}$ as vector spaces if $\overline{m}=\overline{q}$, it follows from ~\Lem{dim} immediately.  
\end{proof}
\end{cor}

\subsubsection{}
\begin{lem}{g_i module}
  Let $L\left(\epsilon_{1}\right)$ be a simple $\mathfrak{\dot{g}}^{\sigma}$ module with the highest weight $\epsilon_{1}$. Then:
\begin{enumerate}

\item $\mathfrak{\dot{g}}^{1}$  and $\mathfrak{\dot{g}}^{3}$ are isomorphic as $\mathfrak{\dot{g}}^{\sigma}$ module to $\Pi(L\left(\epsilon_{1}\right)).$
\item If $k\not=l$, then $\mathfrak{\dot{g}}^{2}$ is isomorphic as $\mathfrak{\dot{g}}^{\sigma}$ module to $S^{2}\left(L\left(\epsilon_{1}\right)\right)$. 
\newline If $k=l$, then $\mathfrak{\dot{g}}^{2}= L\left(2\epsilon_{1}\right)$ and $S^{2}\left(L\left(\epsilon_{1}\right)\right)=\mathfrak{\dot{g}}^{2}\oplus \text{triv}$ as $\mathfrak{\dot{g}}^{\sigma}$ module.
\end{enumerate}. 
\begin{proof}
Denote by $\Omega_{\mathfrak{\dot{g}}^{i}}$ the set of weights of $\mathfrak{\dot{g}}^{i}$ as $\mathfrak{\dot{g}}^{\sigma}$ module. 
Following the roots table, presented in ~\cite{vdL}, we see that: $$\Omega_{\mathfrak{\dot{g}}^{2}}=\left\{ \underbrace{0;\pm\epsilon_{i}\pm\epsilon_{j};\pm\epsilon_{i};\pm\delta_{g}\pm\delta_{h};\pm2\epsilon_{i}}_{\text{even}};\underbrace{\pm\delta_{g},\pm\epsilon_{i}\pm\delta_{g}}_{\text{odd}}\right\}$$ and $$\Omega_{\mathfrak{\dot{g}}^{1}}=\Omega_{\mathfrak{\dot{g}}^{3}}=\left\{ \underbrace{\pm\delta_{g}}_{\text{even}};\underbrace{0;\pm\epsilon_{i}}_{\text{odd}}\right\} $$ Where $1\leq i,j\leq k, i\neq j,\,1\leq g,h\leq l,g\neq h.$
By \Lem{dim}, we know that the dimension of every weight space in $\Omega_{\mathfrak{\dot{g}}^{s}}$, for $s\in\left\{ 1,3\right\}$, is $1$. Using \Lem{Standard module}, we see that $\text{ch}L_B(\epsilon_1)=\text{ch}L(\mathfrak{\dot{g}}^s).$ However,  the parity of the weights in these modules is switched. Since $\mathfrak{\dot{g}}^{\sigma}$ isomorphic to $B(k|l)$, this gives $(i).$
Now, every weight space besides $0$ in ${\mathfrak{\dot{g}}^{2}}$ has dimension $1$, and the zero weight, which is isomorphic as vector space to $\mathfrak{g}_{2\delta}$, has dimension $k+l+1$ in case $k\not=l$ and $2k$ if $k=l$, by \Lem{dim}. Therefore, using \Prop{symmetric standard module}, we have: $\text{ch}(\mathfrak{\dot{g}}^{2})=\begin{cases}
\text{ch}(S^{2}\left(L_{B}\left(\epsilon_{1}\right)\right)) & k\neq l\\
\text{ch}(L_{B}\left(2\epsilon_{1}\right)) & k=l
\end{cases}.$ By \Cor{Final}, this gives $(ii)$. 
\end{proof}
\end{lem}

\subsubsection{}
\begin{lem}{K-relations}
The algebra  $\mathcal{K}$ has the following relations:
\begin{enumerate}
    \item If $i+j=0$, then $\left[\mathfrak{g}_{i\delta},\mathfrak{g}_{j\delta}\right] = \mathbb{C}K$.
    
    \item If $i+j \neq 0$, then:
\begin{equation}
    \left[\mathfrak{g}_{i\delta},\mathfrak{g}_{j\delta}\right] =
        \begin{cases}
        0 & i \neq 0, \, \overline{i} = 0 \\
        0 & \overline{i+j} = 0 \\
        \mathfrak{g}_{(i+j)\delta} & \overline{i} = 2, \, \overline{j} = 1,3 \\
        \mathbb{C} x, \quad x \text{ is a nonzero element in } \mathfrak{g}_{(i+j)\delta} & \overline{i} = \overline{j} = 1,3
    \end{cases}
    \label{relations_in_N}
\end{equation}

\end{enumerate}

\begin{proof}
(1)  follows from ~\cite{Kbook}, Thm. 2.2 (the details of the proof in Kac's book are essentially the same in the super case).
\newline For (2), we assume $i+j\neq0$ and divide the proof into the following parts:
\begin{enumerate}[label=\alph*)]
\item
Since $\mathfrak{g}_{i\delta}=t^{i}\otimes\mathfrak{\dot{h}}^{\sigma}$ for $\overline{i}=0, i\not=0$, first formula follows directly from Lemma  ~\ref{centralizer}. 
\item Second formula follows from \Lem{Properties of centralizer}. 
\item Let $at^{j}\in\mathfrak{g}_{j\delta},$  where $\overline{j}$  is either $1$  or $3.$ By \Lem{dim}, $a$ is odd. Thus, we have $a\notin\mathfrak{\dot{h}},$ and therefore $\left[a,\mathfrak{\dot{h}}\right]\neq0.$ By \Lem{dim}, we know that  $$\mathfrak{\dot{h}}=\mathfrak{z}^{0}\oplus\mathfrak{z}^{2}.$$  Since $\left[a,\mathfrak{\dot{h}}^{\sigma}\right]=0,$  there exists $h\in\mathfrak{z}^{2},$ such that $\left[a,h\right]\neq0.$ By \Lem{centralizer}, we have $ht^{i}\in\mathfrak{g}_{i\delta},$ where $\overline{i}=2.$ This gives the third formula. 
\item Let $\Sigma=\left\{ \alpha_{i}\right\} _{i=0}^{n},$  where $\alpha_1,....,\alpha_n$ 
 are the fix set of simple roots for $B(k|l)$ as in ~\ref{notation appendix}, and $\alpha_0=\delta-\theta,$  where $\theta$ is the highest weight of $\mathfrak{\dot{g}}^1$ as $B(k|l)$ module. Let $e_{\alpha_0},e_{\alpha_1},...,e_{\alpha_n}$ be the Chevally generators of $A\left(2k|2l\right)^{\left(4\right)}.$ 

By \Lem{g_i module}, $\theta=\epsilon_1,$ and $2\delta-2\theta$  is a weight in $t^{2}\otimes\mathfrak{\dot{g}}^{2}$  as $B(k|l)$  module. Thus, $\left[e_{0},e_{0}\right]\neq0.$ That is $\left[\mathfrak{\dot{g}}_{-\theta}^{1},\mathfrak{\dot{g}}_{-\theta}^{1}\right]\neq0,$  where $\mathfrak{\dot{g}}_{-\theta}^{1}$  is $-\theta$  weight space inside $\mathfrak{\dot{g}}^1$  as $B(k|l)$ module. 

Let $a_0$ be a non-zero element of zero weight in $\mathfrak{\dot{g}}^{1}\simeq\Pi\left(V_{st}\right).$ 
Since $\mathfrak{\dot{g}}^{1}\simeq V_{st},$  we have $$\mathfrak{\dot{g}}_{-\theta}^{1}=\mathfrak{\dot{g}}_{-\theta}^{0}a_{0}\,\text{ 
and}\,\,\,\left(\ad\mathfrak{\dot{g}}_{-\theta}^{0}\right)^{2}a_{0}=0,$$
where $\mathfrak{\dot{g}}_{-\theta}^{0}$  is the root space of the real root $-\theta$  inside $B(k|l).$ This gives $$\left(\ad\mathfrak{\dot{g}}_{-\theta}^{0}\right)^{2}\left[a_{0},a_{0}\right]=\left[\mathfrak{\dot{g}}_{-\theta}^{0}a_{0},\mathfrak{\dot{g}}_{-\theta}^{0}a_{0}\right]=\left[\mathfrak{\dot{g}}_{-\theta}^{1},\mathfrak{\dot{g}}_{-\theta}^{1}\right]\neq0.$$
Hence $\left[a_{0},a_{0}\right]\neq0$  as required. 
\end{enumerate}
The same argument works for $\mathfrak{\dot{g}}^{-1}=\mathfrak{\dot{g}}^{3}.$  This proves the last formula. 
\end{proof}
\end{lem} 

\subsection{}
\Lem{g^1 in Cartan}, and \Prop{K-relations} prove \Thm{sqdelta}.

\section{Conjugacy theorem for Affine Kac-Moody Lie superalgebras}

\subsection{Motivation}\label{motivation}
Assume $\left(A,\tau\right)$  and $\left(A_{1},\tau_{1}\right)$  are two Cartan data, that yield Kac-Moody superalgebras $\mathfrak{g}\left(A,\tau\right)$  and $\mathfrak{g}\left(A_{1},\tau_{1}\right),$  which are isomorphic to each other. A natural question comes to mind: are these two Cartan data equal? If not, how are they related?     

Clearly, they do not need to be equal. There are two basic operations that can be applied to a Cartan datum $\left(A,\tau\right)$ to obtain a new Cartan datum $ \left(A_{1},\tau_{1}\right)$ , which yield isomorphic contragredient Lie algebra:
\begin{enumerate}
    \item \textbf{Permutation of set of indices $\tau.$} The corresponding matrix will simply be a permutation of the rows of $A$  and same permutation of the columns. 
    \item \textbf{Multiplication of the Cartan matrix by an invertible diagonal matrix $D$ from the left.} The set of indices will remain unchanged. The isomorphism (which is not unique) can be defined by rescaling the co-roots, and consequently the Chevalley generators, as follows: $$e_{i}\longrightarrow e_{i},f_{i}\longrightarrow\epsilon_{i}f_{i},$$
where $\epsilon_i$  is the i-th entry on the diagonal of the matrix $D.$
\end{enumerate}

In addition, isotropic reflexion is another operation on a Cartan datum that produces a new Cartan datum, which yields isomorphic Kac-Moody superalgebra. We would like to know whether there are other such operations? Note that non isotropic reflexions do not change the Cartan datum.

To answer this question, we need to rely on a theorem proved by Victor Kac in 1980 for the non-super case, and by Vera Serganova in 2001 for the super case, which states the following:
\textit{any root base $\Sigma'$  is contragredient and $\Sigma'\in Sk$  or $-\Sigma'\in Sk.$} (See Proposition 5.9 in ~\cite{Kbook} and Theorem 8.3 in ~\cite{Sint}). 

Having this theorem in mind, we obtain the following Corollary:

\begin{cor}{Related Cartan datum}
   Let $\mathfrak{g}\left(A,\tau\right) $ be a Kac-Moody superalgebra, and let $\phi:\mathfrak{g}\left(A,\tau\right)\longrightarrow\mathfrak{g}\left(A_{1},\tau_{1}\right)$ be an isomorphism such that $\phi\left(\mathfrak{h}\right)=\mathfrak{h}_{1}.$ Then, $\left(A_{1},\tau_{1}\right)$ is obtained from some Cartan datum $\left(A',\tau'\right)$ which correspond to a root base $\Sigma'\in Sk$,  by applying none, one, or both of the operations described above (permutation and multiplication). 
\end{cor}
\begin{proof}
By assumption,  $\phi^{-1}\left(\Sigma_{1}^{\vee}\right)\subset\mathfrak{h}$ and $\phi^{*}\left(\Sigma_{1}\right)\subset\mathfrak{h}^{*}.$ The triple $\left(\mathfrak{h},\phi^{*}\left(\Sigma_{1}\right),\phi^{-1}\left(\Sigma_{1}^{\vee}\right)\right)$  forms a realization of the Cartan datum $\left(A_{1},\tau_{1}\right).$  In particular, $\phi^{*}\left(\Sigma_{1}\right)$  is a root base. Therefore, by the theorem above, $\phi^{*}\left(\Sigma_{1}\right)\in Sk,$  or $-\phi^{*}\left(\Sigma_{1}\right)\in Sk.$ In the later case, there exists an invertible diagonal matrix $D$  such that $DA_1=A.$ Hence, we can assume that  $\phi^{*}\left(\Sigma_{1}\right)\in Sk.$ This completes the proof. 
\end{proof}

For the anisotropic case, given the assumption $\phi\left(\mathfrak{h}\right)=\mathfrak{h}_{1},$ the Corollary states that $\left(A,\tau\right)$  can be obtained solely by applying the two operations described above, on $\left(A_{1},\tau_{1}\right).$ 

Thus, we see that this question reduces to the problem of conjugacy of Cartan subalgebras. In this chapter, we formulate and present the proof of the main theorem of the thesis, which concerns the conjugacy theorem of Cartan subalgebras for affine Kac-Moody superalgebras. We start by formally state the main theorem. 

\subsection{}
\begin{thm}{Conjugacy}
    Let $\mathfrak{g}$ be an indecomposable affine Kac-Moody superalgebra, $\mathfrak{h}$ its Cartan subalgebra, and let $\mathfrak{a}$  be any Kac-Moody superalgebra with the Cartan subalgebra $\mathfrak{h}_{\mathfrak{a}}.$ If $\iota:\mathfrak{g}\stackrel{\sim}{\rightarrow}\mathfrak{a}$  is an isomorphism and $\mathfrak{h}'=\iota^{-1}\left(\mathfrak{h}_{\mathfrak{a}}\right)\subset\mathfrak{g},$  then $\mathfrak{h}'=\phi(\mathfrak{h})$ for some inner automorphism of $\mathfrak{g}.$  
\end{thm}
\subsubsection{}
\begin{rem}{Affine-decomposable}
    Under the assumptions of the theorem above, $\mathfrak{a}$  is an image of affine indecomposable Kac-Moody superalgebra. Since $\mathfrak{g}$  is indecomposable, it is quasisimple. Since quasisimplicity is a property which does not rely on the presentation of $\mathfrak{g}$  as Kac-Moody superalgebra, it follows that $\mathfrak{a}$ is quasisimple. Thus, $\mathfrak{a}$  is indecomposable as well (see ~\cite{Sint} Lemma 2.4). In addition, $\mathfrak{a}$  is also affine Kac-Moody superalgebra. Indeed, a Lie superalgebra is indecomposable affine Kac-Moody superalgebra if and only if it is of infinite dimension but of finite Gelfand-Kirillov dimension. This is again a property which does not rely on the presentation of $\mathfrak{g}$ as Kac-Moody superalgebra. 
\end{rem}

\subsection{}
\begin{lem}{AffIso} 
$\mathfrak{g}$ is symmetrizable if and only if $\mathfrak{a}$ is symmetrizable.
\end{lem}
\begin{proof}
    Assume $\mathfrak{g}$ is symmetrizable. Therefore, $\mathfrak{g}$ admits a non-degenerate invariant bilinear form $(-,-)$. We define a non-degenerate invariant bilinear form $\left\langle -,-\right\rangle $  on $\mathfrak{a}$  by $\left\langle x,y\right\rangle =\left(\iota^{-1}\left(x\right),\iota^{-1}\left(y\right)\right).$ So  $A=\left(\left\langle \alpha^{\vee},\beta^{\vee}\right\rangle \right)_{\alpha,\beta\in\Sigma'}$  is symmetrizable Cartan matrix of $\mathfrak{a},$  where $\Sigma'$  is a set of simple roots. Hence, $\mathfrak{a}$  is symmetrizable.  
\end{proof} 
\subsubsection{}
\begin{cor}{Non.Sym}
$\mathfrak{g}=\mathfrak{q}\left(n\right)^{\left(2\right)}$ if and only if $\mathfrak{a}=\mathfrak{q}\left(n'\right)^{\left(2\right)},$  and $\mathfrak{g}=S(2|1;b)$  if and only if ${\mathfrak{a}}=S(2|1;b').$
\end{cor}
\begin{proof}
   By classification, the only indecomposable non-symmetrizable affine Kac-Moody superalgebras are $\mathfrak{q}\left(n\right)^{\left(2\right)}$  and  $S(2|1;b).$ These superalgebras are not isomorphic, since by ~\cite{Sint}, the intersection of $\left[\mathfrak{g}_{\overline{0}},\mathfrak{g}_{\overline{0}}\right]$ with the centre of $\mathfrak{g}$  is zero for $\mathfrak{q}\left(n\right)^{\left(2\right)}$ and non zero for $S(2|1;b).$ 
\end{proof}

\subsubsection{}
\begin{cor}{}
    Assume $\mathfrak{g}_{\overline{0}}$  is a symmetrizable Kac-Moody algebra. Then \Thm{Conjugacy} hold. 
    \begin{proof}
        Since $\mathfrak{g}_{\overline{0}}$  is symmetrizable Kac-Moody algebra, then by \Thm{Conjugacy-KM}, there exists $x\in\mathcal{S}_{min},$ such that: $$\Ad\left(x\right)\left(\mathfrak{h'}\right)\subset\mathfrak{h},$$
where $\mathcal{S}_{min}$ is the minimal Kac-Moody group associated to $\mathfrak{g}_{\overline{0}}.$ Since $\mathfrak{h'}$  is maximal commutative subalgebra, it follows $\text{Ad}\left(x\right)\left(\mathfrak{\mathfrak{h'}}\right)=\mathfrak{h}.$  
    \end{proof}
\end{cor}

\subsubsection{}
\begin{rem}{}
    We see that if $\mathfrak{g}_{\overline{0}}$  is a Kac-Moody algebra, then \Thm{Conjugacy} naturally follows from the non-super and symmetrizable case. This occurs, for instance, in $\mathfrak{g}=\mathfrak{osp}\left(1|2n\right)^{\left(1\right)}.$  However, in the majority of cases, $\mathfrak{g}_{\overline{0}}$  is far from being Kac-Moody. 

Consider for example the case $\mathfrak{g}=\mathfrak{osp}\left(2|2\right)^{\left(2\right)}$  (recall \ref{ex: osp(2|2)}). In this case, $N\oplus\mathbb{C}K$  is an infinite Heisenberg  subalgebra of $\mathfrak{g}_{\overline{0}}.$ 

The situation becomes even more challenging in the case $\mathfrak{g}=\mathfrak{sl}\left(m|n\right)^{\left(1\right)}$ (recall \ref{ex-sl(m|n)}).  Here,  $\bigoplus_{k\in\mathbb{Z}}\mathbb{C}t^{k}z\oplus\mathbb{C}K$  is an infinite Heisenberg subalgebra of $\mathfrak{g}_{\overline{0}}.$ Moreover,  $\mathfrak{sl}_{m}^{\left(1\right)}\times\mathfrak{sl}_{n}^{\left(1\right)}$  does not embed into $\mathfrak{g}_{\overline{0}}.$ Additionally, the "merging" of the derivations leads us to a condition where constructing a set of simple roots for this algebra becomes impossible. 

To overcome this obstacle, we will use \Thm{prop1},  which provides a nice description of $\mathfrak{g}_{\overline{0}}$  and an "invariant" description of $\mathfrak{g}_{pr},$  alongside 
with other results established in previous sections. 
\end{rem}

\subsection{Conjugacy theorem for Symmetrizable Case}
In this subsection, we will prove \Thm{Conjugacy} for affine symmetrizable $\mathfrak{g}.$  
\subsubsection{}
\begin{lem}{sigma-eigenvector}
Let $r$ be the order of the automorphism $\sigma:\dot{\mathfrak{g}}\longrightarrow\dot{\mathfrak{g}}$ that we applied in the affinization of $\mathfrak{\dot{g}}.$ If $h\in\dot{\fh}$ is a $\sigma$-eigenvector and $\sigma^j(h)=h,$ then 
$\dot{\fg}$ contains a $\sigma$-eigenvector $u$ such that $(\ad h)^j (u)=cu$ for some $c\not=0$.
\end{lem}
\begin{proof}
For each $s\in\mathbb{Z}_r$ let  $p_s:\dot{\fg}\to \dot{\fg}^s$ be the projection
along the decomposition $\dot{\fg}=\oplus_{i\in\mathbb{Z}_r} \dot{\fg}^i$. 
Take $h\in\dot{\fh}\cap \fg^s$. Then
$p_{i+s}\circ \ad h=\ad h\circ p_i$, so
$p_i\circ(\ad h)^j=(\ad h)^j \circ p_i$.
Since $h\in\dot{\fh}$ there exists $g\in\dot{\fg}$ such that $[h,g]=cg$ for
some $c\not=0$, so $(\ad h)^j (g)=c^j g$ and 
$(\ad h)^j (p_i(g))=c^j p_i(g)$ for all $i$.
\end{proof}
\subsubsection{}
\begin{lem}{lemghh}
Assume that there exists $g,g'\in\mathcal{S}_{min}$ satisfying
$$\left(\Ad g\right)\left(\mathfrak{h'\cap\mathfrak{g}}_{pr}\right)\subset\mathfrak{h},\quad\left(\Ad g'\right)\left(\mathfrak{h\cap\mathfrak{g}}_{pr}\right)\subset\mathfrak{h'}.$$
Then, $\left(\Ad g\right)\left(\mathfrak{h'}\right)=\mathfrak{h}.$ 
\end{lem}
\begin{proof}
Set $\fh'':=(\Ad g) (\fh'\cap\fg_{\pr})$.
By~\Thm{prop1}, $\mathfrak{g}_{pr}$  is invariant with respect to any automorphism of $\mathfrak{g}.$  In particular $(\Ad g)  (\fg_{\pr})=\fg_{\pr}$ and hence
$$\fh''\subset \mathfrak{h}\cap\fg_{pr}.$$
In particular, $\dim\left(\mathfrak{h'}\cap\mathfrak{g}_{pr}\right)\leq\dim\left(\mathfrak{h}\cap\mathfrak{g}_{pr}\right).$  Similarly, $\dim\left(\mathfrak{h}\cap\mathfrak{g}_{pr}\right)\leq\dim\left(\mathfrak{h'}\cap\mathfrak{g}_{pr}\right),$  and thus $\dim\left(\mathfrak{h}\cap\mathfrak{g}_{pr}\right)=\dim\left(\mathfrak{h'}\cap\mathfrak{g}_{pr}\right).$ 
By \Lem{Codim of g_pr}, this gives $\dim \mathfrak{h}=\dim \mathfrak{h}'.$   

Since $$\left(\Ad g\right)\left(\mathfrak{h'\cap\mathfrak{g}}_{pr}\right)=\mathfrak{h\cap\mathfrak{g}}_{pr},$$  we have 
$$\left(\Ad g\right)\left(\mathfrak{h'}\right)=\mathfrak{h\cap\mathfrak{g}}_{pr}\oplus\mathfrak{t},$$
where $\mathfrak{t}$  acts ad-semisimply and $[t,\mathfrak{h\cap\mathfrak{g}}_{pr}]=0.$ By 
\Lem{Centralizer_of_K}, we know that $\mathfrak{t}\subset\mathcal{K}_{\overline{0}}.$ Assume, for sake of contradiction, that there exists $a\in\mathfrak{t}\backslash\mathfrak{h}.$  Thus, $a=\sum_{i\in \mathbb{Z}}a_{i},$  where $a_i\in\mathfrak{g}_{i\delta}\cap\mathfrak{g}_{\overline{0}},$  and there exists $i\not=0$  such that $a_i\not=0.$ Let $S\subset\mathbb{Z}$  be the set of all integers $s\in S$  such that $a_{i}\not=0.$ Denote $s_0=\max S.$ Take $h\in\dot{\mathfrak{h}}$  such that $a_{s_0}=ht^{s_0}.$ Let $\sigma$ be the automorphism of order $r,$ that we applied in the affinization of $\mathfrak{\dot{g}}$ (for non-twisted affinization we have $r=1$). Note that $h$ is a $\sigma-$eigenvector and there exists $j$  such that $\sigma^j(h)=h.$ Thus, by \Lem{sigma-eigenvector} there exists a $\sigma-$eigenvector $u\in\mathfrak{\dot{g}}$ such that $(\ad h)^j(u)=cu$ for some $c\not=0.$ Denote by $m$ the integer satisfies $u\in\mathfrak{\dot{g}}^{m}.$ 
For each $i,$ we define the projection map $p_{i\cdot j\cdot s_{0}+m}$ onto the subspace $\mathbb{C}t^{i\cdot j\cdot s_{0}+m}\otimes\mathfrak{\dot{g}}^{\overline{i\cdot j\cdot s_{0}+m}}$ along the decomposition $\mathfrak{g}=\bigoplus_{k\in\mathbb{Z}}\mathbb{C}t^{k}\otimes\mathfrak{\dot{g}}^{\overline{k}}\oplus\mathbb{C}K\oplus\mathbb{C}d$.  Then, for every $i,$ we have 
$$p_{i\cdot j\cdot s_{0}+m}\left(\left(\ad a\right)^{i\cdot j}\left(t^mu\right)\right)=c^{i}t^{i\cdot j\cdot s_{0}+m}u\neq0.$$ Hence, $a$ is not acting ad-locally finitely and therefore it is not ad-semisimply, which leads to contradiction. Therefore $\left(\Ad g\right)\left(\mathfrak{h'}\right)\subset\mathfrak{h}$  and since $\dim\mathfrak{h}=\dim\mathfrak{h}',$ we get $\left(\Ad g\right)\left(\mathfrak{h'}\right)=\mathfrak{h}.$

\end{proof}
We will need the following technical lemma:
\subsubsection{}
\begin{lem}{lem1}
Let $\ft$ be a Lie superalgebra and $\ft_1$, $\ft_2$ are subalgebras.
If $\ft_1$ acts $\ad$-diagonally, then 
$\ft_1\cap\ft_2$   acts $\ad$-diagonally
on $\ft_2$.
\end{lem}
\begin{proof}
 Let $t\in\mathfrak{t}_{1}\cap\mathfrak{t}_{2}.$ Since $t\in\mathfrak{t}_1,$ $t$  acts ad-diagonally on $\mathfrak{t}.$ In addition, as $\mathfrak{t}_{1}\cap\mathfrak{t}_{2}$ is an ideal in $\mathfrak{t}_2,$ we know that $\mathfrak{t}_2$  is $\ad t$ invariant and therefore $t$  acts ad-diagonally on $\mathfrak{t}_2.$  
\end{proof}

\subsubsection{Case when $B_\pi$ is indecomposable}
We aim to prove that the conditions of \Lem{lemghh} hold. For this purpose, we first consider the case where
$B_\pi$ is indecomposable. In this scenario, by~\ref{properties_by_section_9}  (ii), 
$\fg(B_\pi)$ can be consider as subalgebra of $\mathfrak{g}_{\overline{0}},$ and $\fg_{\pr}=[\fg(B_\pi),\fg(B_\pi)]$.
By~\Lem{lem1}, the subalgebra  $\fh'\cap\fg(B_\pi)$ acts $\ad$-diagonally on $\fg(B_\pi).$ Consequently, $\fh'\cap\fg_{\pr}$ acts $\ad$-diagonally on $\fg(B_\pi)$.
Using \Thm{Conjugacy-KM}, there exists $g\in \mathcal{S}_{min}$ such that:
$$(\Ad g) (\fh'\cap\fg_{\pr})\subset {\mathfrak{h}},$$  where $\mathcal{S}_{min}$  is the minimal Kac-Moody group associated to $\fg(B_\pi).$ 

\subsubsection{Case when $B_\pi$ is decomposable} \label{decomposable}
We now consider the case where $B_\pi$ is decomposable. By ~\ref{properties_by_section_9} (i), $\fg(B_\pi)=\fg(B_1)\oplus\fg(B_2)\oplus\cdot\cdot\cdot\fg(B_k)$ where $k\in\left\{2,3\right\}$. We deal with case $k=2$, the latter is similar. 
Using ~\ref{properties_by_section_9} (ii), we view $\fg(B_i)$ as subalgebras of $\mathfrak{g}_{\overline{0}}$ and set 
$$\ft_i:= [\fg(B_i),\fg(B_i)].$$
Then,
$$\fg_{\pr}=\ft_1+ \ft_2,\ \ \  \ft_1\cap \ft_2=\mathbb{C}K,\ \ \ 
[\ft_1,\ft_2]=0.$$
Thus, each element $x\in\fg_{\pr}$ can be written (non-uniquely) in the form 
$x=x_1+x_2$ for $x_i\in \ft_i$. For two such decompositions
$x_1+x_2=y_1+y_2,$ we have $x_1-y_1,x_2-y_2\in\mathbb{C}K.$ \newline\newline
Note that $\mathfrak{g}_{pr}/\mathbb{C}K=\mathfrak{t}_{1}/\mathbb{C}K\oplus\mathfrak{t}_{2}/\mathbb{C}K$  and let $p_{1}:\mathfrak{g}_{pr}/\mathbb{C}K\longrightarrow\mathfrak{t}_{1}/\mathbb{C}K$ be the projection on $\mathfrak{t}_{1}/\mathbb{C}K$ with kernel $\mathfrak{t}_{2}/\mathbb{C}K.$ We define $\mathfrak{h}'_{1}\subset\mathfrak{t}_{1}$ to be the preimage of $p_{1}\left(\mathfrak{h}\cap\mathfrak{g}_{pr}/\mathbb{C}K\right)$ under the quotient map. $\mathfrak{h}'_{2}\subseteq\mathfrak{t}_{2}$  is defined similarly.  

\subsubsection{}
\begin{lem}{lem-diag-sub-t_i}
 $\mathfrak{h}'_1$ (resp. $\mathfrak{h}'_2$) is a commutative $\ad$-diagonalizable subalgebra of
 $\ft_1$ (resp. $\mathfrak{t}_2$). 
\end{lem}
\begin{proof}
We prove only for $\mathfrak{h}'_1,$ as for $\mathfrak{h}'_2$ it is exactly the same. 
Take  $h'$ in $\fh'\cap\fg_{\pr}$ and fix a decomposition
 $h'=h_1'+h_2',$  where $h_1',h_2'$  are in $\mathfrak{h}'_{1},\mathfrak{h}'_{2}$  respectively. Since $\ft_1$ and $\ft_2$ commute, it follows that $\text{ad}h_{1}'|_{\mathfrak{t}_{1}}=\text{ad}h'|_{\mathfrak{t}_{1}}$. Since $\ft_1$ is an ideal in $\mathfrak{g}_{pr},$ it is $\text{ad}h'$ invariant subspace. It follows that $\text{ad}h_{1}'$ is diagonalizable on $\ft_1$. In addition, the  elements $\{h'_1, h'\in\fh'\}$ are pairwise commute. Indeed, take $h''\in\fh'$ and fix a decomposition
 $h''=h''_1+h''_2$. Then
 $$0=[h',h'']=[   h_1'+h_2',h''_1+h''_2]=[   h_1',h''_1]+[ h_2',h''_2],$$
 so $[h_1',h''_1],[h_2',h''_2]\in\left(\mathfrak{t}_{1}\cap\mathfrak{t}_{2}\right)=\mathbb{C}K.$ Therefore $(\ad h'_1)^2(h''_1)=0$.
Since the restriction of  $(\ad h'_1)$  on $\ft_1$ is a diagonal operator, the fact that
 $(\ad h'_1)^2(h''_1)=0$ implies $(\ad h'_1)(h''_1)=0$ as required.
Therefore $\fh'_1$ is a commutative and ad-diagnosable subalgebra of $\ft_1$. 
 \end{proof}

\subsubsection{}
By ~\ref{properties_by_section_9} (ii),  $(\mathfrak{h}\cap\ft_{1})\oplus\mathbb{C}d$ is the Cartan subalgebra of $\fg(B_1).$ By ~\Lem{lem-diag-sub-t_i} and \Thm{Conjugacy-KM}, there exists $g_1\in \left(\mathcal{S}_{1}\right)_{\text{min}}$ such that:
$$(\Ad g_1) (\fh'_1)\subset(\mathfrak{h}\cap\ft_1)\oplus\mathbb{C}d,$$ where $\left(\mathcal{S}_{1}\right)_{\text{min}}$ is the minimal Kac-Moody group associated to $\fg(B_1)$. Similarly, there exists  $g_2\in \left(\mathcal{S}_{2}\right)_{\text{min}}$ such that: 
$$(\Ad g_2) (\fh'_2)\subset (\fh\cap\ft_2)\oplus\mathbb{C}d.$$ 
Since $\ft_{1}$ and $\ft_{2}$ commute, $\Ad g_i$ acts by identity on $\ft_{s}$ where $i=1,s=2$ or $i=2, s=1.$ Therefore, $\Ad g_1\circ\Ad g_2=\Ad g_1g_2$ is an inner automorphism of $\fg$ such that: $$\Ad g_1g_2(\fh'\cap\fg_{\pr})\subset\mathfrak{h}.$$ 
\subsubsection{}
Thus, we see that, whether $\mathfrak{g}\left(B_\pi\right)$  is decomposable or not, there exists $g\in\mathcal{S}_{min}$  such that: $$\text{Ad}g\left(\mathfrak{h'}\cap\mathfrak{g}_{pr}\right)\subset\mathfrak{h}.$$
By \Prop{AffIso} and \Rem{Affine-decomposable}, $\mathfrak{a}$ is also indecomposable affine and symmetrizable, and therefore we may use the same argument to show that there exists $g'\in\mathcal{A}_{min}$ such that: $$\text{Ad}g'\left(\iota\left(\mathfrak{h}\right)\cap\mathfrak{a}_{pr}\right)\subset\mathfrak{h_a},$$ where $\mathcal{\mathcal{A}}_{min}$ is Kac-Moody minimal group associated to $\mathfrak{a}\left(B_{\pi}\right).$   
This implies that there exists $g'\in\mathcal{S}_{min}$ such that: $$\text{Ad}g'\left(\mathfrak{h}\cap\mathfrak{g}_{pr}\right)\subset\mathfrak{h'}.$$
Using \Lem{lemghh}, this proves \Thm{Conjugacy} in the symmetrizable case.
\subsection{Conjugacy theorem for ${\mathfrak{q}^{(2)}_n}$}
\subsubsection{}\label{lem queer}
The proof of \Lem{lemghh} carries over to the twisted queer case, with  \Claim{form of pr} replacing \Thm{prop1}, and relying  on the fact that \Lem{Centralizer_of_K} remains valid in this setting. 

Furthermore, while \Lem{Codim of g_pr} does not hold in this case, \Cor{Non.Sym} implies: $$\dim\mathfrak{h}-\dim\mathfrak{h}\cap\mathfrak{g}_{pr}=\dim\mathfrak{h}'-\dim\mathfrak{h}'\cap\mathfrak{g}_{pr}=2,$$
which suffices for our argument.
\subsubsection{}
  Recall the surjective homomorphism $$\psi:\left[\mathfrak{g}\left(B_{\pi}\right),\mathfrak{g}\left(B_{\pi}\right)\right]\longrightarrow\mathfrak{g}_{pr}.$$
By ~\cite{Sint}, in case $\mathfrak{g}=\mathfrak{q}_{n}^{\left(2\right)}$ we have  $\mathfrak{g}\left(B_{\pi}\right)\simeq\mathfrak{sl}_{n}^{\left(1\right)}$ and  $\mathfrak{g}_{pr}\simeq\left[\mathfrak{sl}_{n}^{\left(1\right)},\mathfrak{sl}_{n}^{\left(1\right)}\right]/\mathbb{C}K.$ In particular,  $\psi$  is the quotient map with kernel $\mathbb{C}K.$   
Assume that 
\begin{equation} \tag{*} \label{eq:star-property}
    \psi^{-1}\left(\mathfrak{h'}\cap\mathfrak{g}_{pr}\right)\,\,\text{is} \ad\text{-diagnoziable and commutative.}
\end{equation}

Then, by \Thm{Conjugacy-KM}, there exists $g\in\mathcal{S}_{min},$ where $\mathcal{S}_{min}$ is the minimal Kac-Moody group associated to $\mathfrak{sl}_n^{\left(1\right)},$ such that: $$\Ad g\left(\psi^{-1}\left(\mathfrak{h'}\cap\mathfrak{g}_{pr}\right)\right)\subset\psi^{-1}\left(\mathfrak{h}\cap\mathfrak{g}_{pr}\right).$$
Since $\Ad g\left(\mathbb{C}K\right)=\mathbb{C}K,$  $\Ad g$  can be naturally induced an automorphism of $\mathfrak{sl}_{n}^{\left(1\right)}/\mathbb{C}K,$  that gives: $$\Ad g\left(\mathfrak{h'}\cap\mathfrak{g}_{pr}\right)\subset\mathfrak{h}\cap\mathfrak{g}_{pr}.$$
By \Cor{Non.Sym}, we can also find $g'\in\mathcal{S}_{min}$  such that $\Ad g'(\mathfrak{h}\cap\mathfrak{g}_{pr})\subset\mathfrak{h}\cap\mathfrak{g}_{pr}.$  
Together with Lemma ~\ref{lem queer}, we conclude that \Thm{Conjugacy} hold in the case $\mathfrak{g=\mathfrak{q}^{\left(2\right)}\left(n\right)}.$

Hence, all that remains is to prove property ~(\ref{eq:star-property}). 
\subsubsection{}
\begin{lem}{}
    Property ~(\ref{eq:star-property}) holds. 
\end{lem}
\begin{proof}
  Let $a\in\psi^{-1}\left(\mathfrak{h'}\cap\mathfrak{g}_{pr}\right).$ Then $\psi(a)$  is ad-diagnoziable and therefore ad-locally finite. In addition, $\ker\psi$ is finite-dimensional and thus $a$  is $\ad-$locally finite. By ~\cite{Kumar} Theorem 10.2.12, we have a decomposition 
  $$a=a_s+a_n,$$
where $a_s$ is ad-diagnoziable, $a_n$  is ad-locally nilpotent, and $[a_s,a_n]=0.$ Moreover, for any subspace $M\subset\mathfrak{sl_n}^{(1)}$  we have
$$\begin{array}{l}
\ [a,M]=M\ \ \ \Longrightarrow\ \ \ [a_s,M]=M,\\
\ [a, M]=0\ \ \ \Longrightarrow\ \ \ [a_s,M]=[a_n,M]=0\end{array}.$$
Since $\psi(a)$ is ad-diagnoziable, we can write $$\mathfrak{g}_{pr}=\bigoplus_{i=1}^{\infty}\mathbb{C}v_{i},$$  where $v_i$  is eigenvector of $\ad \psi(a).$ 

Since $[a,\psi^{-1}(\mathbb{C}v_{i})]\subset\psi^{-1}(\mathbb{C}v_{i}),$ it follows $[a_s,\mathbb{C}v_{i}]\subset\psi^{-1}(\mathbb{C}v_{i}).$  Thus, $[\psi(a_s),\mathbb{C}v_{i}]\subset \mathbb{C}v_{i},$  and $$[\psi(a_n),\mathbb{C}v_{i}]\subset \mathbb{C}v_{i}.$$
Hence, $\psi(a_n)$ acts ad-diagnoziable on $\mathfrak{g}_{pr}.$ On the other hand, since $a_n$  is ad-locally nilpotent on $\mathfrak{sl}_{n}^{\left(1\right)},$ then $\psi(a_n)$ is ad-locally nilpotent on $\mathfrak{g}_{pr}.$ Thus, $\ad \psi(a_n)$  is ad-diagnoziable and ad-nilpotent and therefore $\ad \psi(a_n)=0,$ that is, $\psi\left(a_{n}\right)\in\mathfrak{c}\left(\mathfrak{g}_{pr}\right).$ However, $\mathfrak{c}\left(\mathfrak{g}_{pr}\right)=0$  and thus $\psi(a_n)=0.$  It follows that $a_n\in \mathbb{C}K$ and hence $\ad a_n=0.$ Thus, $\ad a$ is ad-diagnoziable.  

 Next, let $a,b\in\psi^{-1}\left(\mathfrak{h'}\cap\mathfrak{g}_{pr}\right).$ Since $\mathfrak{h'}$  is commutative, we have $[a,b]\in\ker\psi=\mathbb{C}K.$  Thus, $(\ad a)^{2}\left(b\right)=0.$  Since $a$ is ad-diagnoziable, it follows $(\ad a)(b)=0.$ This completes the proof.

\end{proof}

\subsection{Conjugacy theorem for $S(2|1;b)$}
We are now assuming $\mathfrak{g}=S(2|1;b).$  
\subsubsection{}
\begin{lem}{loc.fin}
Let $x\in\mathfrak{g}_{\overline{0}}$ be an element acting ad-locally finite. Then $$x\in\mathfrak{g}_{pr}\oplus\mathbb{C}d\oplus\mathbb{C}L_0.$$ 
\begin{proof}
By \Lem{Inside pr}, $\mathfrak{g}_{pr}$ is an ideal of $\mathfrak{g}_{\overline{0}}.$ Thus, $\overline{x}:=x+\mathfrak{g}_{pr}$ is acting ad-locally finite on $\mathfrak{g}_{\overline{0}}/\mathfrak{g}_{pr}\simeq\mathcal{W}\oplus\mathbb{C}d,$ where $\mathcal{W}:=\mathcal{V}_{ir}/\mathbb{C}K$ is isomorphic to Witt algebra. Write $\bar{x}=d'+\sum_{s\in S}c_{s}L_{s},\,c_{s}\neq0,d'\in\mathbb{C}d.$ 
Note that $S$ is a finite set. If $S=\emptyset,$ then $x\in\mathfrak{g}_{pr}\oplus\mathbb{C}d$ as wished. Otherwise, denote $s_{0}=\max S$. Assume, for sake of contradiction that $s_0>0.$

Let $p_{i}:\mathcal{W}\oplus\mathbb{C}d\rightarrow\mathbb{C}L_{i}$ be the projection map with the kernel $\mathbb{C}d\oplus\bigoplus_{j\in\mathbb{Z},j\neq i}\mathbb{C}L_{j}.$ 
We aim to prove, by induction, that for every $q\in\mathbb{N}$, the following holds: 
$$\max\left\{ i\in\mathbb{N}\,|\,p_{(i+1)s_0+1}\left((\ad{\bar{x}})^{q}\left(L_{s_{0}+1}\right)\right)\neq0\right\} =q$$
For the base case $q=1$: $$\left[\bar{x},L_{s_{0}+1}\right]=\left[d',L_{s_{0}+1}\right]+\left[\sum_{s\in S}c_{s}L_{s},L_{s_{0}+1}\right]=-\left(s_{0}+1\right)L_{s_{0}+1}+\sum_{s\in S}c_{s}\left[L_{s},L_{s_{0}+1}\right]$$$$=-\left(s_{0}+1\right)L_{s_{0}+1}+\sum_{s\in S\backslash s_{0}}c_{s}\left[L_{s},L_{s_{0}+1}\right]-c_{s_{0}}L_{2s_{0}+1}.$$
Thus, $$p_{2s_{0}+1}\left(\ad{\bar{x}}\left(L_{s_{0}+1}\right)\right)=-c_{s_{0}}L_{2s_{0}+1}\neq0.$$ Additionally, clearly for any $t>1$: $$p_{(t+1)s_0+1}\left(\ad{\bar{x}}\left(L_{s_{0}+1}\right)\right)=0.$$
Now for general $q$: $$(\ad{\bar{x}})^{q}\left(L_{s_{0}+1}\right)=\left[d',\left[(\ad{\bar{x}})^{q-1},L_{s_{0}+1}\right]\right]+\left[\sum_{s\in S}c_{s}L_{s},\left[(\ad{\bar{x}})^{q-1},L_{s_{0}+1}\right]\right]$$
Thus, by induction hypothesis, the result follows. 
Therefore, $\bar{x}$ is not acting ad-locally finite and thus $x$  is not acting ad-locally finite, which contradicts our assumption. Therefore $s_{0}=\max S\leq0.$ Similarly, one can show that $\min S\geq0.$ Thus, $S=\left\{ 0\right\} .$ 
Hence, $\sum_{s\in S}c_{s}L_{s}$  is proportional to $L_0$ and thus $x\in\mathfrak{g}_{pr}\oplus\mathbb{C}d\oplus\mathbb{C}L_0$. 
\end{proof}
\end{lem}
\subsubsection{}
We are now ready to prove \Thm{Conjugacy} for case $S(2|1;b).$
\begin{proof}
Recall that, by ~\cite{Sint}, $\mathfrak{g}_{pr}\oplus\mathbb{C}d\simeq\mathfrak{sl}_{2}^{\left(1\right)}.$ Since ad-diagonalizable elements are in particular ad-locally finite elements, it follows by \Lem{loc.fin} that $\mathfrak{h}'\subset\mathfrak{sl}_{2}^{\left(1\right)}\oplus\mathbb{C}L_{0}.$ Recall that $L_0-d$  lies in $\mathfrak{c}(\mathfrak{g}_{\overline{0}}).$ Thus, by maximality of $\mathfrak{h}'$ we must have $L_{0}-d\in\mathfrak{h'}.$ Therefore: 
$$\mathfrak{h'}=\mathfrak{h'}\cap\mathfrak{sl}_{2}^{\left(1\right)}\oplus\mathbb{C}\left(L_{0}-d\right).$$
In addition, $$\mathfrak{h}=\mathfrak{h}\cap\mathfrak{sl}_{2}^{\left(1\right)}\oplus\mathbb{C}\left(L_{0}-d\right),$$
where $\mathfrak{h}\cap\mathfrak{sl}_{2}^{\left(1\right)}$  is the Cartan subalgebra of $\mathfrak{sl}_{2}^{\left(1\right)}.$ Note that $\mathfrak{h'\cap\mathfrak{g}}_{pr}$ is also ad-diagonalizable commutative subalgebra in $\mathfrak{sl}_{2}^{\left(1\right)}$. Indeed, for every $h\in\mathfrak{h'}\cap\mathfrak{sl}_{2}^{\left(1\right)},$ the subspace $\mathfrak{sl}_{2}^{\left(1\right)}$ is $\ad{h}-$invariant and therefore $\ad{h}$ is diagonalizable on $\mathfrak{sl}_{2}^{\left(1\right)}$.
By ~\Thm{Conjugacy-KM}, there exists $g\in\mathcal{S}_{min}$ such that $\Ad g\left(\mathfrak{h'}\cap\mathfrak{sl}_{2}^{\left(1\right)}\right)\subset\mathfrak{h}\cap\mathfrak{sl}_{2}^{\left(1\right)}.$ Denote this inner automorphism by $\phi.$ Since $L_{0}-d\in\mathfrak{c}\left(\mathfrak{g}_{\overline{0}}\right),$ we have $\phi\left(L_{0}-d\right)\in\mathbb{C}\left(L_{0}-d\right).$ Thus, $\phi\left(\mathfrak{h'}\right)\subseteq\mathfrak{h}.$ 
Note that $\mathfrak{h}'\subset\phi^{-1}\left(\mathfrak{h}\right)$ is ad-diagonalizable and commutative subalgebra of $\mathfrak{g}$ as $\phi$  is an automorphism of $\mathfrak{g}.$ By maximality of $\mathfrak{h}'$ it follows  $\mathfrak{h}'=\phi^{-1}\left(\mathfrak{h}\right)$  and therefore $\phi\left(\mathfrak{h'}\right)=\mathfrak{h},$  as wished.  
\end{proof}

\subsection{}
\begin{cor}{}
$\mathfrak{q}_{n}^{\left(2\right)}$  is isomorphic to $\mathfrak{q}_{n'}^{\left(2\right)}$  if and only if $n=n'.$ 

Similarly,  $S(2|1;b)$  is isomorphic to $S(2,1;b')$ if and only if $b\in\mathbb{Z}\pm b'$   
\end{cor}
\begin{proof}
    By ~\cite{GHS} Lemma 10.4.1, the Cartan datum $(A_b',\tau_b')$  is obtained from $(A_b,\tau_b)$ through the three operations described in \ref{motivation}, if and only if $b\in\mathbb{Z}\pm b'.$   
    Thus, the corollary is now follows directly from \Thm{Conjugacy} and \Cor{Related Cartan datum}.
\end{proof}

\section{Appendix }
We note and prove several facts about the representations of \( B(k|l) \) and \( D(k|l) \). While most of the proofs are presented for \( \mathfrak{g} = B(k|l) \), they can be adapted to \( \mathfrak{g} = D(k|l) \) with only minor adjustments.
  
\subsection{Notation} \label{notation appendix}
\begin{enumerate}
    \item We fix the following set of simple roots for \( B(k|l) \):

$$\Sigma_{B}:=\left\{ \epsilon_1 - \delta_1, \delta_1 - \epsilon_2, \epsilon_2 - \epsilon_3, \epsilon_3 - \epsilon_4, \dots, \epsilon_{k-1} - \epsilon_k, \epsilon_k - \delta_2, \delta_2 - \delta_3, \dots, \delta_{l-1} - \delta_l, \delta_l \right\}.$$

Similarly, we fix the following set of simple roots for \( D(k|l) \):
\[
\Sigma_D := \left\{ \epsilon_1 - \epsilon_2, \dots, \epsilon_{k-1} - \epsilon_k, \epsilon_k - \delta_1, \delta_1 - \delta_2, \dots, \delta_{l-1} - \delta_l, \delta_{l-1} + \delta_l \right\}.
\]
        \item We denote by \( L_B(\alpha) \) the simple \( B(k|l) \)-module with highest weight \( \alpha \). Similarly, we use the notation \( L_D(\alpha) \) for the simple \( D(k|l) \)-module with highest weight \( \alpha \).
        \item Denote by \( V_{st}^B \) the standard module for \( B(k|l) \) and by \( Ad_B \) the adjoint representation. Similarly, we use the notation \( V_{st}^D \) and \( Ad_D \) for the corresponding representations of \( D(k|l) \). We denote the trivial representation by \( \text{triv} \). For a module \( U \), we denote by \( U_0 \) the zero-weight space of the module.
        \item We also fix a set of simple roots for $B_k$  and $C_l$: $$\Sigma_{B_{k}}:=\left\{ \epsilon_{1}-\epsilon_{2},\epsilon_{2}-\epsilon_{3},...,\epsilon_{k-1}-\epsilon_{k},\epsilon_{k}\right\}\text{, } \Sigma_{C_{l}}:=\left\{ \delta_{1}-\delta_{2},\delta_{2}-\delta_{3},...,\delta_{l-1}-\delta_{l},2\delta_{l}\right\} .$$

\end{enumerate} 

\subsection{}
\begin{lem}{Standard module}
\begin{enumerate}
    \item $\text{ch}(L_B(\epsilon_1))=\text{ch}(V_{st}^{B})=\sum_{\alpha\in\Omega_{B}}e^{\alpha}$ , where  $\Omega_{B}=\left\{ \underbrace{\pm\delta_{g}}_{\text{odd}},\underbrace{0,\pm\epsilon_{i}}_{\text{even}}\right\} $. 
    \item $\text{ch}(L_D(\epsilon_1))=\text{ch}(V_{st}^{D})=\sum_{\alpha\in\Omega_{D}}e^{\alpha}$ , where  $\Omega_{D}=\left\{ \underbrace{\pm\delta_{g}}_{\text{odd}},\underbrace{\pm\epsilon_{i}}_{\text{even}}\right\} $. 

\end{enumerate}
\begin{proof}
   Consider the natural embedding of \( B(k|l) \) into \( \mathfrak{gl}(2k+1, 2l) \). The Cartan subalgebra of \( B(k|l) \) under this embedding, denoted by \( \mathfrak{h}_B \), is spanned by the basis 
\[
\Sigma_B := \left\{ e_i - e_{i+k} \right\}_{i=2}^{k+1} \cup \left\{ e_{i+2k+1} - e_{i+2k+l} \right\}_{i=1}^{l},
\]
where \( (e_s)_{ij} = \delta_{ij,ss} \).

Let \( V_{st} \) be the standard representation of \( \mathfrak{gl}(2k+1, 2l) \), so the character of \( V_{st} \) is given by
\[
\text{ch}(V_{st}) = \sum_{i=1}^{2k+1} e^{\epsilon_i} + \sum_{i=1}^{2l} e^{\delta_i}.
\]
We compute \( V_{st}^B \) by restricting the weights of \( V_{st} \) to \( \mathfrak{h}_B \). Denote the restricted weights as \( \epsilon_i' := \epsilon_i|_{\mathfrak{h}_B} \), \( \delta_i' := \delta_i|_{\mathfrak{h}_B} \), and let \( \{ \epsilon_i^B \}_{i=1}^k \cup \{ \delta_i^B \}_{i=1}^l \) be the dual basis to \( \Sigma_B \). Then, the following relations hold:
\( \epsilon_1' = 0 \),
 \( \epsilon_i' = \epsilon_{i-1}^B \) for \( 2 \leq i \leq k \),
 \( \epsilon_{i+k}' = -\epsilon_{i-1}^B \) for \( 2 \leq i \leq k \),
 \( \delta_i' = \delta_i^B \) for \( 1 \leq i \leq l \),
 \( \delta_{i+l}' = -\delta_i^B \) for \( 1 \leq i \leq l \).

Since \( \epsilon_1^B \) is the highest weight and \( V_{st}^B \) is irreducible, this proves (1). A similar approach can be used to prove (2).
\end{proof}
\end{lem}
\subsection{}
\begin{lem}{conditions for semi-simple}
 Let \( N \) be a finite-dimensional self-dual module with exactly three non-isomorphic simple subquotients: \( L_1, L_2, L_3 \), each occurring with multiplicity 1. Furthermore, each of these simple subquotients is self-dual. Then \( N \) is completely reducible.
   \begin{proof}
       First, note that every module has at least one irreducible submodule, since the socle is never zero for finite dimensional modules. Assume, without loss of generality, that \( L_1 \) is a submodule of \( N \). Since \( L_1 \) and \( N \) are self-dual, it follows that \( L_1 \) is also a quotient of \( N \), so there exists a surjective homomorphism \( \varphi_1 : N \longrightarrow L_1 \). Now, \( \ker \varphi_1 \) has irreducible subquotients \( L_2 \) and \( L_3 \). As before, one of them must be a submodule of \( \ker \varphi_1 \); suppose it is \( L_2 \). Since \( \ker \varphi_1 \) is a submodule of \( N \), it follows that \( L_1 \oplus L_2 \) is a submodule of \( N \). Again, since \( L_1 \oplus L_2 \) is self-dual, there exists a surjective homomorphism \( \varphi_2 : N \longrightarrow L_1 \oplus L_2 \). Therefore, \( \ker \varphi_2 \) has only \( L_3 \) as a simple subquotient, and hence \( \ker \varphi_2 = L_3 \). Since \( \varphi_2 \) is surjective, it follows that \( N / L_3 \simeq L_1 \oplus L_2 \), and therefore \( N = L_1 \oplus L_2 \oplus L_3 \).
\end{proof}
\end{lem}
\subsection{}
\label{ch_of_2epsilon_1}
For the next proposition, we need to know the character of the \( 2\epsilon_1 \)-highest weight module under the action of \( B_k \), denoted by \( L_{B_k}(2\epsilon_1) \), and the character of the \( \delta_1 + \delta_2 \)-highest weight module under the action of \( C_l \), denoted by \( L_{C_l}(\delta_1 + \delta_2) \).

By~\cite{Onishchik and Vinberg}, Table 5, as \( B_k \)-modules, we have \( S^p(R(\epsilon_1)) = \sum_{i=0}^{\infty} R((p - 2i) \epsilon_1) \), where \( R(\alpha) \) is the highest-weight module with weight \( \alpha \). For \( p = 2 \), it follows that \( S^2(R(\epsilon_1)) = R(2\epsilon_1) + R(0) \). Since \( R(\epsilon_1) \) is the standard \( B_k \)-module, we know its character, and thus we also know \( \text{ch}(S^2(R(\epsilon_1))) \). Therefore, we have:
\[
\text{ch}(L_{B_k}(2\epsilon_1)) = \sum_{\alpha \in \Omega_{B_k}} e^{\alpha} + k \cdot e^0, \quad \Omega_{B_k} = \left\{ \pm 2\epsilon_i, \pm \epsilon_i, \pm \epsilon_i \pm \epsilon_j \right\}_{1 \leq i \neq j \leq k}.
\]

For \( C_l \), again by~\cite{Onishchik and Vinberg}, Table 5, as \( C_l \)-modules, we have \( \Lambda^p R = \sum_{i \geq 0} R(\pi_{p - 2i}) \), where \( \pi_{p - 2i} \) is the \( p - 2i \) fundamental weight. Therefore, \( \Lambda^2 R(\delta_1) = R(\delta_1 + \delta_2) \oplus R(0) .\) Since \( R(\delta_1) \) is the standard module, we know its character, and thus we also know \( \text{ch}(\Lambda^2 R(\delta_1)) \). Therefore, we have:
\[
\text{ch}(L_{C_l}(\delta_1 + \delta_2)) = \sum_{\alpha \in \Omega_{C_l}} e^{\alpha} + (l - 1) e^0, \quad \Omega_{C_l} = \left\{ \pm \delta_i \pm \delta_j \right\}_{1 \leq i \neq j \leq l}.
\]
\subsection{}
\begin{lem}{In different blocks}
    Let $V$  be a $B(k|l)$ (resp. $D(k|l))$ module. Let $L_{B}\left(2\epsilon_{1}\right)$ (resp. $L_D(2\epsilon_1)$) be the highest weight $B(k|l)$ (resp. $D(k|l)$) module with weight $2\epsilon_1.$ Assume $V$  has only $\text{triv}$ and  $L_{B}\left(2\epsilon_{1}\right)$ (resp. $L_D(2\epsilon_1)$)  as its simple subquotients, of multiplicity $1$. Then $V=\text{triv}\oplus L_{B}\left(2\epsilon_{1}\right) $ (resp. $V=\text{triv}\oplus L_D(2\epsilon_1)$).
\end{lem}
\begin{proof}
    We prove this Lemma for $B(k|l)$ case, the proof for $D(k|l)$  is similar.  
    
    It is sufficient to show that $L_{B}\left(2\epsilon_{1}\right)$  and $\text{triv}=L_{B}\left(0\right)$  are belong to different blocks in the category $\mathcal{O}.$  Equivalently, we shall show that $\chi_{0}\neq\chi_{2\epsilon_{1}},$   where $\chi_{\alpha}$  is the central character associated with the weight $\alpha.$ 

First, assume $k>l.$  

For the purpose of this proof only, we choose the following set of simple roots: $$\left\{ \underbrace{\epsilon_{1}-\delta_{1},\delta_{1}-\epsilon_{2},\epsilon_{2}-\delta_{2},\delta_{2}-\epsilon_{3},\epsilon_{3}-\delta_{3},...,\delta_{l}-\epsilon_{l+1}}_{2+2\left(l-1\right)=2l\,\text{elements}},\underbrace{\epsilon_{l+1}-\epsilon_{l+2},...,\epsilon_{k-1}-\epsilon_{k}}_{k-1-l-1+1=k-l-1\,\text{elements}},\epsilon_{k}\right\}.$$
    By direct computation, it follows: $$\rho=x_{1}\left(\epsilon_{1}-\delta_{1}\right)+...+x_{2l}\left(\delta_{l}-\epsilon_{l+1}\right)+y_{1}\left(\epsilon_{l+1}-\epsilon_{l+2}\right)+...+y_{k-l-1}\left(\epsilon_{k-1}-\epsilon_{k}\right)+y_{k-l}\cdot\epsilon_{k},$$where: $$\begin{cases}
x_{1}=x_{3}=...=x_{2l-1}=k-l-\frac{1}{2}\\
x_{2}=x_{4}=...=x_{2l}=0\\
y_{i}=\left(k-l-\frac{1}{2}i\right)-\frac{1}{2}\left(i-1\right)i & 1\leq i\leq k-l
\end{cases}.$$ 
Now we compute the multisets $E_{2\epsilon_{1}}:=\left\{ \left|\left(2\epsilon_{1}+\rho,\epsilon_{i}\right)\right|\right\} _{i=1}^{k} \text{ and } D_{2\epsilon_{1}}:=\left\{ \left|\left(2\epsilon_{1}+\rho,\delta_{i}\right)\right|\right\} _{i=1}^{l},$  and find the sets $\overline{E}_{2\epsilon_{1}}=E_{2\epsilon_{1}}\backslash\left(E_{2\epsilon_{1}}\cap D_{2\epsilon_{1}}\right)$  and $\overline{D}_{2\epsilon_{1}}=D_{2\epsilon_{1}}\backslash\left(D_{2\epsilon_{1}}\cap E_{2\epsilon_{1}}\right).$  One can see that: $\overline{E}_{2\epsilon_{1}}=\left\{ 2+\left(k-l-\frac{1}{2}\right),\left(y_{2}-y_{1}\right),\left(y_{3}-y_{2}\right),...,\left(y_{k-l}-y_{k-l-1}\right)\right\} ,\,\overline{D_{\lambda}}=\emptyset.$
Note that for every $1\leq i\leq k-l-1,$ one has $y_{i+1}-y_{i}=k-l-\frac{1}{2}-i.$
Similarly: $$\overline{E_{0}}=\left\{ \left(k-l-\frac{1}{2}\right),\left(y_{2}-y_{1}\right),\left(y_{3}-y_{2}\right),...,\left(y_{k-l}-y_{k-l-1}\right)\right\} ,\,\overline{D_{0}}=\emptyset.$$ So as $\overline{E_{0}}\neq\overline{E}_{2\epsilon_{1}},$  it follows $\chi_{0}\neq\chi_{2\epsilon_{1}}$ (see ~\cite{Cores}).   
\\Second, assume $k\leq l.$  We take the following set of simple roots: $$\left\{ \underbrace{\epsilon_{1}-\delta_{1},\delta_{1}-\epsilon_{2},\epsilon_{2}-\delta_{2},\delta_{2}-\epsilon_{3},\epsilon_{3}-\delta_{3},...,\delta_{k-1}-\epsilon_{k},\epsilon_{k}-\delta_{k}}_{1+2\left(k-1\right)=2k-1\,\text{elements}},\underbrace{\delta_{k}-\delta_{k+1},...,\delta_{l-1}-\delta_{l}}_{l-1-k+1=l-k\,\text{elements}},\delta_{l}\right\} .$$
Then: $$\rho=x_{1}\left(\epsilon_{1}-\delta_{1}\right)+x_{3}\left(\epsilon_{2}-\delta_{2}\right)+...+x_{2k-1}\left(\epsilon_{k}-\delta_{k}\right)+y_{1}\left(\delta_{k}-\delta_{k+1}\right)+...+y_{l-k}\left(\delta_{l-1}-\delta_{l}\right)+y_{l-k+1}\delta_{l},$$ where: $$\begin{cases}
y_{i}=\left(l-k+\frac{1}{2}\right)i-\left(l-k+\frac{1}{2}\right)-\frac{1}{2}\left(i-1\right)i & 1\leq i\leq l-k+1\\
0=x_{2}=...=x_{2k-2}\\
x_{1}=x_{3}=...=x_{2k-1}=-l+k-\frac{1}{2}
\end{cases}$$ By direct computation, we get: $$\overline{D}_{2\epsilon_{1}}=\begin{cases}
\left\{ \left|-l+k-\frac{1}{2}\right|,\left|l-k+\frac{1}{2}-1\right|,\left|l-k+\frac{1}{2}-3\right|,...,\left|l-k+\frac{1}{2}-l+k\right|\right\}  & l-k\geq2\\
\left\{ \left|-l+k-\frac{1}{2}\right|,\left|l-k+\frac{1}{2}-1\right|,...,\left|l-k+\frac{1}{2}-l+k\right|\right\}  & else
\end{cases}$$
And $$\overline{D}_{0}=\left\{ \left|l-k+\frac{1}{2}-1\right|,...,\left|l-k+\frac{1}{2}-l+k\right|\right\} .$$  Therefore $\overline{D_{0}}\neq\overline{D}_{2\epsilon_{1}}$  and again it follows $\chi_{0}\neq\chi_{2\epsilon_{1}}$  as wished.
\end{proof}
\subsection{}
\begin{prop}{symmetric standard module}
\begin{enumerate}
    \item $\text{ch}S^{2}\left(L_{B}\left(\epsilon_{1}\right)\right)=\left(k+l+1\right)e^{0}+\sum_{\alpha\in\Omega'_{B}}e^{\alpha}$ 
and $S^{2}\left(L_{B}\left(\epsilon_{1}\right)\right)=L_{B}\left(2\epsilon_{1}\right)\oplus\text{triv},$ where: $$\Omega'_{B}=\left\{ \underbrace{\pm\epsilon_{i}\pm\epsilon_{j};\pm\epsilon_{i};\pm\delta_{g}\pm\delta_{h};\pm2\epsilon_{i}}_{\text{even}};\underbrace{\pm\delta_{g};\pm\epsilon_{i}\pm\delta_{g}}_{\text{odd}}\right\}.$$
    \item   $\text{ch}S^{2}\left(L_{D}\left(\epsilon_{1}\right)\right)=\left(k+l\right)e^{0}+\sum_{\alpha\in\Omega'_{D}}e^{\alpha}$ and $S^{2}\left(L_{D}\left(\epsilon_{1}\right)\right)=L_{D}\left(2\epsilon_{1}\right)\oplus\text{triv},$  where: $$\Omega'_{D}=\left\{ \underbrace{\pm\epsilon_{i}\pm\epsilon_{j};\pm\delta_{g}\pm\delta_{h};\pm2\epsilon_{1}}_{\text{even}};\underbrace{\pm\epsilon_{i}\pm\delta_{j}}_{\text{odd }}\right\} .$$
\end{enumerate}
\end{prop}
\begin{proof}
    Using \Lem{Standard module}, the character of \( S^2\left(L_B\left(\epsilon_1\right)\right) \) can be calculated straightforwardly. It follows immediately that the highest weight is \( 2\epsilon_1 \). Therefore, \( \left[S^2\left(L_B\left(\epsilon_1\right)\right) : L_B\left(2\epsilon_1\right)\right] = 1 \). 

Next, let us demonstrate that \( \dim \text{Hom}_{B(k|l)}\left(\text{triv}, S^2\left(V_{st}^B\right)\right) = 1 \). Since \( V_{st}^B \) and its dual have the same character and are irreducible modules, it follows that they are isomorphic.
Therefore: \begin{align*}
\dim \text{Hom}_{B\left(k,l\right)}\left(\text{triv},V_{st}^{B}\otimes V_{st}^{B}\right) 
&= \dim \text{Hom}_{B\left(k,l\right)}\left(\text{triv},V_{st}^{B}\otimes\left(V_{st}^{B}\right)^{*}\right) \\
&= \dim \text{Hom}_{B\left(k,l\right)}\left(V_{st}^{B},V_{st}^{B}\right) = 1.
\end{align*}

So, $V_{st}^{B}\otimes V_{st}^{B}$ contains a trivial submodule. Note that $V_{st}^{B}\otimes V_{st}^{B}=S^{2}\left(V^B_{st}\right)\oplus\Lambda^{2}\left(V^B_{st}\right).$  By direct calculation, we find that $\text{ch}\Lambda^{2}\left(V_{st}^{B}\right)=\text{ch}(ad_{B})$, and the zero weight space of $\Lambda^{2}\left(V_{st}^{B}\right)$ has dimension $k+l$ . Therefore, $\Lambda^{2}\left(V_{st}\right)$ cannot have a trivial submodule. It follows that $\dim Hom_{B(k|l)}\left(\text{triv},S^{2}\left(V_{st}^{B}\right)\right)=1$. In particular, this implies that there is a unique $\text{triv}$ submodule within $S^{2}\left(L_{B}\left(\epsilon_{1}\right)\right).$

Recall that $\mathfrak{g}:=B(k|l)=\mathfrak{osp}(2k+1,2l)$ and  $\mathfrak{osp}(2k+1,2l)_{\overline{0}}=\mathfrak{o}_{2k+1}\times\mathfrak{sp}_{2l}.$ Let $V_{B_k}$ be the module generated by the highest weight vector $v_{2\epsilon_1}$ in  $ L_{B}\left(2\epsilon_{1}\right),$ under the action of the subalgebra $\mathfrak{o}_{2k+1}\times\left\{ 0\right\}\simeq B_{k}.$ This leads to $\left[V_{B_{k}}:L_{B_{k}}\left(2\epsilon_{1}\right)\right]=1,$ where $L_{B_{k}}\left(2\epsilon_{1}\right)$  is the irreducible $B_k$  module with highest weight $2\epsilon_1.$   
By acting with the isotropic reflexion $r_{\epsilon_{1}-\delta_{1}},$  we obtain a new set of simple roots for $B(k|l)$:  $$\Sigma'_{B}:=r_{\epsilon_{1}-\delta_{1}}\Sigma_{B}=\left\{ \delta_{1}-\epsilon_{1},\delta_{2}-\epsilon_{2},\epsilon_{1}-\epsilon_{2},..,\epsilon_{k-1}-\epsilon_{k},\epsilon_{k}-\delta_{3},\delta_{3}-\delta_{4},...,\delta_{l-1}-\delta_{l},\delta_{l}\right\} $$ Therefore, $L_{\Sigma_{B}}\left(2\epsilon_{1}\right)=L_{\Sigma'_{B}}\left(\delta_{1}+\delta_{2}\right).$ Let $V_{C_{l}}$ be the module generated by the highest weight vector in $L_{\Sigma'_{B}}\left(\delta_{1}+\delta_{2}\right),$  under the action of the subalgebra $\left\{ 0\right\} \times\mathfrak{sp}_{2l}\simeq C_{l}.$ Then, with respect to the set of simple roots $\Sigma'_B$, we have $\left[V_{C_{l}}:L_{C_{l}}\left(\delta_{1}+\delta_{2}\right)\right]=1,$ where $L_{C_{l}}\left(\delta_{1}+\delta_{2}\right)$  is an irreducible $C_l$  module with highest weight $\delta_{1}+\delta_{2}.$

Now, by  \ref{ch_of_2epsilon_1},  we observe that: $$\Omega_{B_{k}}\cup\Omega_{C_{l}}=\left\{ \pm2\epsilon_{i},\pm\epsilon_{i},\pm\epsilon_{i}\pm\epsilon_{j}\right\} _{1\leq i\neq j\leq k}\cup\left\{ \pm\delta_{i}\pm\delta_{j}\right\} _{i\neq j}\subset\Omega\left(L_{B}\left(2\epsilon_{1}\right)\right).$$ All these weights are of dimension $1$  and furthermore,  $\dim\left(L_{B}\left(2\epsilon_{1}\right)\right)_{0}\geq k+l-1.$ 

Note that $\left[S^{2}\left(V_{st}^{B}\right)/\text{triv}:L_{B}\left(2\epsilon_{1}\right)\right]=1$ and  $\dim (S^{2}\left(V_{st}^{B}\right)/\text{triv})_0=k+l.$ Hence, $\dim L_{B}\left(2\epsilon_{1}\right)_{0}\leq k+l.$  

Let $f_{\epsilon_{1}-\delta_{1}}\in\mathfrak{g}_{-\left(\epsilon_{1}-\delta_{1}\right)}.$  Clearly $v_{\epsilon_{1}+\delta_{1}}:=f_{\epsilon_{1}-\delta_{1}}v_{2\epsilon_{1}}\neq0.$  Since $\mathfrak{n}_{\overline{0}}^{+}v_{\epsilon_{1}+\delta_{1}}=0,$  it follows that $\left[L_{B}\left(2\epsilon_{1}\right),L_{\mathfrak{g}_{\overline{0}}}\left(\epsilon_{1}+\delta_{1}\right)\right]=1,$  where $L_{\mathfrak{g}_{\overline{0}}}\left(\epsilon_{1}+\delta_{1}\right)$  is the submodule generated by $v_{\epsilon_{1}+\delta_{1}},$  under the action of $\mathfrak{g}_{\overline{0}}=B_{k}\oplus C_{l}.$ Now, since $L_{\mathfrak{g}_{\overline{0}}}\left(\epsilon_{1}+\delta_{1}\right)=L_{B_{k}}\left(\epsilon_{1}\right)\otimes L_{C_{l}}\left(\delta_{1}\right),$  it follows that  $\left\{ \pm\delta_{g},\pm\epsilon_{i}\pm\delta_{g}\right\} \subset\Omega\left(L_{B}\left(2\epsilon_{1}\right)\right).$ Therefore, we conclude that: $$\Omega\left(L_{B}\left(2\epsilon_{1}\right)\right)=\Omega_B',$$
where $\dim S^{2}\left(L_{B}\left(\epsilon_{1}\right)\right)_{\alpha}=\dim\left(L_{B}\left(2\epsilon_{1}\right)\right)_{\alpha}=1$ , for every $\alpha\in\Omega'_{B}.$  Since there exists a unique trivial submodule in $S^2(L_B(\epsilon_1),$ it follows $\dim L_{B}\left(2\epsilon_{1}\right)_{0}=k+l.$  Therefore,  $\text{triv}$  and $L_{B}\left(2\epsilon_{1}\right)$  must be the only simple subquotients of $S^{2}\left(L_{B}\left(\epsilon_{1}\right)\right).$
Thus, $V_{st}^{B}\otimes V_{st}^{B}$  has three non-isomorphic simple subquotients, each of multiplicity $1$, which are self-dual: $L_{B}\left(2\epsilon_{1}\right),\text{triv},\Lambda^{2}\left(V_{st}^{B}\right).$ By \Lem{conditions for semi-simple}, it follows $V_{st}^{B}\otimes V_{st}^{B}=L_{B}\left(2\epsilon_{1}\right)\oplus\Lambda^{2}\left(V_{st}^{B}\right)\oplus\text{triv}.$ Consequently, $S^{2}\left(L_{B}\left(\epsilon_{1}\right)\right)=L_{B}\left(2\epsilon_{1}\right)\oplus\text{triv}.$ 
This gives (1).
\\ (2) can be proved in a very similar approach. 
\end{proof}
\subsection{}
\begin{cor}{Final}
    Let $V$  be a $B(k|l)$  (resp. $D(k|l)$) module, such that $\text{ch}V=\text{ch}S^{2}\left(L_{B}\left(\epsilon_{1}\right)\right)$ (resp. $\text{ch}S^{2}\left(L_{D}\left(\epsilon_{1}\right)\right)$).  Then, $V=\text{triv}\oplus L_{B}\left(2\epsilon_{1}\right)$ (resp.  $L_{D}\left(2\epsilon_{1}\right)).$
\end{cor}
\begin{proof}
    This is immediate from \Prop{symmetric standard module}   and \Lem{In different blocks}. 
\end{proof}

\end{document}